\documentclass[
12pt,a4paper]
{amsart}

\usepackage{amssymb, amsmath, amsthm, mathrsfs}
\usepackage[linktoc=page]{hyperref}
\usepackage{cleveref} 
\usepackage[alphabetic,lite]{amsrefs}
\usepackage{tikz-cd}
\usepackage{comment}
\usepackage{fullpage} 
\usepackage{time}
\usepackage[pagewise, mathlines]{lineno} 

\theoremstyle{plain}
\newtheorem{lemma}{Lemma}[section]
\newtheorem{theorem}[lemma]{Theorem}
\newtheorem{prop}[lemma]{Proposition}
\newtheorem{cor}[lemma]{Corollary}

\theoremstyle{definition}
\newtheorem{remark}[lemma]{Remark}
\newtheorem*{ack}{Acknowledgments}

\newtheorem{example}[lemma]{Example}
\newtheorem{defn}[lemma]{Definition}
\newtheorem{strategy}[lemma]{Strategy}

\newtheorem{question}[lemma]{Question}

\newcommand{\bb}[1]{{\mathbb{#1}}}
\newcommand{\mf}[1]{{\mathfrak{#1}}}
\newcommand{\mca}[1]{{\mathcal{#1}}}
\newcommand{\Z}{\bb{Z}}
\newcommand{\Q}{\bb{Q}}
\newcommand{\R}{\bb{R}}
\newcommand{\C}{\bb{C}}
\newcommand{\F}{\bb{F}} 

\newcommand{\To}{\longrightarrow}
\newcommand{\inj}{\hookrightarrow}
\newcommand{\surj}{\twoheadrightarrow}
\newcommand{\act}{\curvearrowright}
\newcommand{\congto}{\overset{\cong}{\to}}

\newcommand{\ol}{\overline}
\newcommand{\ul}{\underline}
\newcommand{\wh}[1]{{\widehat{#1}}}
\newcommand{\wt}[1]{{\widetilde{#1}}}
\renewcommand{\mod}{\ {\rm mod}\ }

\usepackage{bm} 

\usepackage{xcolor}
     
\newcommand{\ds}{\displaystyle}

\theoremstyle{plain}

\newtheorem{claim*}{Claim}

\numberwithin{equation}{section}
\numberwithin{table}{section}
\subjclass[2020]{Primary 57K32, 57K31, 20E18; Secondary 37E30, 37D20, 57K14, 57R30, 11R06}

\keywords{profinite rigidity, twisted multivariable Alexander polynomial, taut polynomial, 3-manifold, link, veering triangulation, pseudo-Anosov flow, small dilatation}

\title{{\Large Taut polynomials from finite quotients of\\ fibered hyperbolic 3-manifold groups}} 
\author{\large Tamunonye Cheetham-West, Biao Ma, Jun Ueki, Youheng Yao}
\date{\today \ \now \ (GMT = JST$-9$ = CST$+6$)} 
\address{Department of Mathematics \\ Yale University; New Haven, CT, 06511, US}
  \email{tamunonye.cheetham-west@yale.edu} \email{youheng.yao@yale.edu}

\address{School of Mathematical Sciences, Key Laboratory of Intelligent Computing and Applications (Ministry of Education), Tongji University; Shanghai 200092, China} 
\email{24024@tongji.edu.cn}

\address{Department of Mathematics, Faculty of Science, Ochanomizu University; 2-1-1 Otsuka, Bunkyo-ku, 112-8610, Tokyo, Japan}
\email{uekijun46@gmail.com}

\makeatletter
\let\origmaketitle\maketitle
\def\maketitle{
  \begingroup
  \def\uppercasenonmath##1{} 
  \let\MakeUppercase\relax   
  \origmaketitle
  \endgroup
}
\makeatother

\begin{document}
\pagestyle{plain}

\begin{abstract}
We prove that, whenever the monodromy maps are fully-punctured, the finite quotients of a fibered hyperbolic $3$-manifold group detect the taut polynomials of the fibered faces of the Thurston norm ball, as well as the Teichm\"{u}ller polynomials of suitable Dehn fillings. 
Toward this, we develop a general framework for the profinite invariance of twisted multivariable Alexander polynomials, which is of independent interest; we also prove that being fully-punctured and preserving transverse orientations are profinite properties of the monodromy maps. 
As an application, we identify specific one-cusped hyperbolic $3$-manifolds whose fundamental groups are profinitely rigid among $3$-manifold groups, by a strategy using normalized dilatations and the veering census; 
notably, the taut polynomials distinguish a pair of mapping tori sharing the fiber and the dilatation.
\end{abstract}

\maketitle 

\setcounter{tocdepth}{2} 
\tableofcontents

\section{Introduction}

\setcounter{subsection}{-1}

\subsection{Overview} 
The main goal of this paper is to establish a connection between the pseudo-Anosov dynamics on fibered hyperbolic 3-manifolds and the finite quotients of their fundamental groups.

By a foundational theorem of Thurston, a fibered 3-manifold admits a complete hyperbolic structure of finite volume if and only if its monodromy map is pseudo-Anosov and the fibered surface has a negative Euler characteristic.
In this setting, Agol--Gu\'{e}ritaud (cf.\,\cite{Agol2011Contemp}) proved that the exterior of the singular orbits of the suspension pseudo-Anosov flow admits a canonical finite combinatorial structure called the \emph{veering triangulation}.   
To encode these fibered, dynamical, and combinatorial structures into a tractable algebraic invariant, Landry--Minsky--Taylor \cite{LandryMinskyTaylor2024JEMS} introduced the \emph{taut polynomial} for veering triangulations, which specializes to McMullen's Teichm\"{u}ller polynomial in the fibered setting.
 
Motivated by this context, we will prove the following main result: 
\emph{ 
The finite quotients of $\pi_1(M)$ of a fibered hyperbolic 3-manifold $M$ determine the taut polynomial, provided that the pseudo-Anosov monodromy map is fully-punctured} 
(\Cref{thm:tautpoly}). 

Parlak \cite{Parlak2023JT} proved that the taut polynomial may be interpreted as a twisted multivariable Alexander polynomial. 
So, we first develop a general algebraic framework over the profinite group ring $\wh{\Lambda}=\wh{\Z}[\![t_1^{\wh{\Z}},\ldots,t_r^{\wh{\Z}}]\!]$,
which significantly extends previous single-variable cases \cites{Ueki5,Ueki6}, to show that:  
\emph{Twisted multivariable Alexander polynomials are profinite invariants of profinitely compatible representations} (\Cref{thm:twisted}). 
Its proof draws on subtle tools from commutative algebra and number theory, and also yields results on classical invariants such as linking numbers and Mahler measures. 

We remark that our proof of \Cref{thm:tautpoly} crucially demonstrates that the relevant edge-orientation homomorphisms are profinitely compatible, and this allows us to apply \Cref{thm:twisted}. 
The compatibility is based on 
\emph{the profinite detection of transverse orientations of the invariant foliations} (\Cref{thm:transverse}). 

To settle the prerequisite in \Cref{thm:tautpoly}, we also prove that 
\emph{the finite quotients of $\pi_1(M)$ detect whether the monodromy map is fully-punctured}  (\Cref{thm:fully-punctured}). 
Combined with Tsang's dilatation bounds \cite{TsangCC2025-arXiv} and the veering census \cite{VeeringCensus}, this result also provides a practical strategy for identifying hyperbolic 3-manifolds that are profinitely rigid (\Cref{cor:rigidity}). 

The precise structural formulations of these results and their broader context will be detailed in the remainder of this introduction. 
Throughout this paper, all $3$-manifolds are assumed to be connected and orientable, and hyperbolic $3$-manifolds are moreover assumed to be complete and of finite volume. 

\subsection{Profinite rigidity} \label{ss.profiniterigidity} 
It is a highly interesting question to ask to what extent a 3-manifold group is \emph{profinitely rigid}, 
namely, which geometric, dynamical, or topological properties of a 3-manifold $M$ are determined by the isomorphism class of the profinite completion $\wh{\pi}_1(M)$ of the fundamental group 
(cf.\,\cites{BoileauFriedl2020AMS,BridsonReid2020AMS,Reid-ICM2018,YiLiu2022ICM,BJZR2023.problems}). 
This inquiry is equivalent to asking what the set of finite quotients of $\pi_1(M)$ knows (cf.\,\Cref{ss.hatpi}). 
Thus, it stems from foundational interests in classical group theory, 
as well as practical methods for detection problems in low-dimensional topology (cf.\,\cite{Sakuma-RIMS2011}*{Problem 3}). 

Among the fundamental results in this direction, which our arguments build on, Wilton--Zalesskii showed that the profinite completion detects the geometry of a closed geometric $3$-manifold \cite[Theorems A, B, and 8.4]{WiltonZalesskii2017GT}, as well as the prime and JSJ decompositions of a closed $3$-manifold \cite[Theorems A and B]{WiltonZalesskii2019CM}; their study of pro-$p$ subgroups \cite{WiltonZalesskii2017JLMS} makes the arguments available in the cusped case. 
Wilkes extended the latter to compact $3$-manifolds with boundary \cite{Wilkes2018NZ}, and showed that the profinite completion detects which geometries are involved in the geometric decomposition, distinguishing graph manifolds from the others \cite[Theorem B]{Wilkes2018JA}. 
As a consequence, if $\wh{\pi}_1(M)\cong\wh{\pi}_1(N)$ and $M$ is a finite-volume hyperbolic
$3$-manifold, then so is $N$, and they share the number of cusps; see \cite[Proposition 2.3]{Xu2024fillings-arXiv} for this summarized statement.  We will use this fact throughout Sections \ref{sec:punctured}--\ref{sec:App}. 

In general, the question of profinite rigidity is known to be subtle and difficult. Among the notable achievements, Liu \cite[Theorem 1.1]{YiLiu2023Invent} and Xu \cite[Theorem D]{Xu2025AdvMath} established the almost profinite rigidity, namely, the finiteness of the homeomorphism classes sharing the profinite completion, for all finite-volume hyperbolic $3$-manifolds and for all compact orientable $3$-manifolds with empty or toroidal boundary, respectively.
For hyperbolic $3$-manifolds, the proof relies on the virtual fibering due to Agol and the profinite invariance of the dilatation and the Euler characteristic of fibrations; the finitely many candidates sharing these quantities remain to be distinguished. 
In this paper, we address the remaining finite ambiguity for certain mapping tori, proving \Cref{cor:rigidity} by means of the taut polynomials.

\subsection{Twisted multivariable Alexander polynomials} 
Polynomial invariants have been indispensable throughout the history of low-dimensional topology. 
Results in profinite rigidity have been established for one-variable (twisted) Alexander polynomials \cites{BoileauFriedl2020AMS,Ueki5,Ueki6,YiLiu2023Invent} and the $A$-polynomials \cites{CheethamWestYao2025-arXiv-Apoly,Xu2025regularity-arXiv} so far. 
Pursuing this direction, we first establish the following, 
where we define the completion of a representation $\rho:\pi\to {\rm GL}_k\Z$ to be the induced continuous representation $\wh{\rho}:\wh{\pi}\to {\rm GL}_k\wh{\Z}$.

\begin{theorem}\label{thm:twisted}
Let $M$ and $N$ be finite-volume hyperbolic 3-manifolds with an isomorphism $\Phi:\wh{\pi}_1(M)\congto \wh{\pi}_1(N)$ of profinite completions. 
Let $\rho:\pi_1(M)\to {\rm GL}_k\Z$ and $\varrho:\pi_1(N)\to {\rm GL}_k\Z$ be representations such that their completions satisfy $\wh{\rho}=\wh{\varrho} \circ \Phi$. 
Let $\alpha:\pi_1(M)\surj H_M$ and $\beta:\pi_1(N)\surj H_N$ denote the maximal free abelian quotient maps. 
Then, the twisted Alexander polynomials $\Delta^{\alpha}_\rho \in \Z[H_M]$ and $\Delta^{\beta}_{\varrho}\in \Z[H_N]$ coincide 
via the induced isomorphism $\Phi_\star:H_M\congto H_N$, up to multiplication by units of $\Lambda=\Z[H_M]$. 
\end{theorem}

Note that the induced isomorphism $\Phi_\ast:\wh{H}_M\congto \wh{H}_N$ restricts to an isomorphism $\Phi_\star:H_M\congto H_N$ of discrete groups by the regularity result due to Liu, Xu, and Hanany \cites{YiLiu2023Invent,Xu2025regularity-arXiv,Hanany2026regularity-arXiv} (cf.\,\Cref{sec:regular}). 
To prove this result, we establish a general result (\Cref{thm.regular}) in the algebraic setting where the given profinite isomorphism $\Phi$ is \emph{regular} in the sense of Boileau--Friedl. 
We further establish a similar result under a weaker assumption called \emph{$\wh{\Z}^\times$-regular} 
(\Cref{thm.hatZtimes-regular}), where the argument becomes more subtle. 

In the course of the argument, in a more general algebraic setting without regularity assumptions, 
we discuss presentation matrices, 
faithfully flat ring extensions (invoking the theory of multivariable Iwasawa modules),  
and order ideals (Lemmas \ref{lem:share}, \ref{lem:hatLambdaisflat}, \ref{lem:oderideal}),
to establish 
\emph{the well-definedness and the profinite rigidity of the profinite Alexander ideal in the profinite ring $\wh{\Lambda}=\wh{\Z}[\![t_1^{\wh{\Z}},\ldots,t_r^{\wh{\Z}}]\!]$} 
up to the natural ${\rm GL}_r\wh{\Z}$-action (\Cref{prop:order}, \Cref{thm:completeA}). 
We also derive results on linking numbers (\Cref{cor:lk}),
Milnor's triple linking numbers (\Cref{cor:Milnor}), 
and Euclidean/$p$-adic Mahler measures (Corollaries \ref{cor:Mahler}, \ref{cor:pMahler}). 

\subsection{Fully-punctured pseudo-Anosov maps} 
Fibrations of a 3-manifold $M$ over $S^1$ are organized by fibered faces of the Thurston norm ball in $H^1(M;\R)$. 
Given a hyperbolic 3-manifold $M$, there is a celebrated \emph{three-way correspondence} 
among fibered faces of the Thurston norm ball, 
suspension pseudo-Anosov flows without perfect fits, 
and layered veering triangulations of the exterior of the singular fibers,  
as pioneered by Thurston, Fried, Agol, and others (cf.\,\Cref{ss.tri}).
To effectively utilize this correspondence, we employ the following notion:  
A pseudo-Anosov map on a surface is said to be \emph{fully-punctured} 
if the set of singularities of the stable/unstable leaves is contained in the set of punctures. 
This property of a fibration depends only on its fibered face. (See also Sections \ref{ss.pAmap}, \ref{ss.ff}.) 

Based on the profinite correspondence of fibered classes (\Cref{lem:Liu-correspondence}), we will prove the following. 

\begin{theorem} \label{thm:fully-punctured} 
Let $M$ and $N$ be fibered cusped hyperbolic 3-manifolds with an isomorphism $\wh{\pi}_1(M)\congto\wh{\pi}_1(N)$, and 
let $\varphi\in H^1(M;\Z)$ and $\psi\in H^1(N;\Z)$ be the corresponding fibered classes.
Let $f_\varphi$ denote the pseudo-Anosov monodromy map on the fiber $S_\varphi$ of $\varphi$, and similarly for $\psi$. 
Then $f_\varphi$ is fully-punctured if and only if $f_\psi$ is fully-punctured.
\end{theorem} 

We reduce the argument to the closed case, adapting the proof of \cite{YiLiu2023Invent}*{Lemma 9.5}. 
Other important ingredients are 
the profinite correspondence of the closed orbits \cite{YiLiu2025Peking}*{Lemma 5.2} 
and the fact that the twisted Lefschetz numbers detect the prong numbers of periodic orbits (cf.\,\cite{YiLiu2023Invent}*{Lemma 8.5}). 
The latter, combined with the peripheral regularity, moreover yields that each pair of corresponding cusps of $M$ and $N$ carries the same number of punctures of the fibers, with the same prong number (\Cref{cor:prong-correspondence}); 
this refinement will be used in Sections \ref{ss.profinite-transverse}, \ref{ss.filled}, and \ref{sec:App}. 

\subsection{Transverse orientations of invariant foliations} 

The invariant foliations of a pseudo-Anosov map carry a further piece of dynamical data beyond their combinatorial type. 
Let $f$ be a pseudo-Anosov map on a surface $S$ with stable foliation $W^s(f)$. 
We say that $W^s(f)$ is \emph{transversely orientable} if its normal line field admits a continuous orientation, and that $f$ \emph{preserves} the transverse orientation if this orientation can be chosen to be
$f$-invariant. 
This condition governs the transverse orientability of the stable lamination of the suspension flow, and hence the edge-orientability of the associated veering triangulation. 

We will prove that this condition is detected by the profinite completion: 

\begin{theorem} \label{thm:transverse} 
Let $M$ and $N$ be fibered cusped hyperbolic $3$-manifolds with an isomorphism
$\Phi: \wh{\pi}_1(M)\congto\wh{\pi}_1(N)$ on the profinite
completions, and let $\varphi\in H^1(M;\Z)$ and
$\psi\in H^1(N;\Z)$ be the corresponding fibered classes, with
pseudo-Anosov monodromy maps $f_\varphi$ and $f_\psi$. Then the stable
foliation $W^s(f_\varphi)$ is transversely orientable if and only if
$W^s(f_\psi)$ is transversely orientable. Moreover, in this case, $f_\varphi$ preserves the transverse
orientation of $W^s(f_\varphi)$ if and only if $f_\psi$ preserves that of
$W^s(f_\psi)$.
\end{theorem}

For closed fibers, the transverse orientability of the invariant foliations is known to be a profinite invariant by Liu \cite[Corollary 8.2]{YiLiu2023Invent} (see also \cite[Theorem 11.1]{YiLiu2023Duke}). 
Our theorem covers cusped fibers and moreover detects the preservation of the orientation; 
it is a key ingredient in the proofs of Theorems \ref{thm:tautpoly} and \ref{thm:filled}. 
Its proof combines a punctured-surface version of Rykken's homological criterion (\Cref{thm:Rykken-punctured}; cf.\  \cite[Theorem 3.3]{Rykken1999Michigan}) 
with the coincidence of the dilatations of the corresponding monodromy maps (\Cref{lem:Liu-correspondence}). 

\subsection{The taut polynomials of veering triangulations}

Let $M$ be a hyperbolic 3-manifold. 
McMullen \cite{McMullen2000ASENS} introduced \emph{the Teichm\"uller polynomial} $\Theta_{\mathcal{F}}$ in $\Z[H_1(M;\Z)_{\rm free}]$ associated to a fibered face $\mathcal{F}$ of the Thurston norm ball in $H^1(M;\R)$, 
which encodes dynamical information of the corresponding suspension pseudo-Anosov flow. 
In particular, the largest absolute value of the roots of the specialization of $\Theta_{\mathcal{F}}$ at 
an integral class $\alpha\in \R_+\mathcal{F}$ is \emph{the dilatation} (\emph{stretch factor}) of the corresponding fibration in the open cone $\R_+\mathcal{F}$. 

More recently, Landry--Minsky--Taylor \cite{LandryMinskyTaylor2024JEMS} introduced the \emph{taut polynomial} $\Theta_\mathcal{V}$ of a veering triangulation $\mathcal{V}$, an invariant built from the combinatorics of the triangulation. 
In a fibered setting, $\Theta_\mathcal{V}$ specializes to $\Theta_\mathcal{F}$ of the corresponding fibered face $\mca{F}$, 
and they coincide if the monodromy maps are fully-punctured (\Cref{lem:tautTeich}). 
Significantly, Parlak \cite{Parlak2023JT} proved that $\Theta_{\mathcal{V}}$ may be regarded as a twisted Alexander polynomial of the edge-orientation homomorphism $\omega:\pi_1(M)\to {\rm GL}_1\Z=\Z^\times$. 

Liu \cite[Theorem 5.1, Corollary 5.3]{YiLiu2023Invent} proved that a profinite isomorphism of hyperbolic $3$-manifold groups induces a bijective correspondence between fibered faces of the Thurston norm balls. 
Based on this correspondence and \Cref{thm:fully-punctured}, we state the following. 

\begin{theorem}\label{thm:tautpoly} 
Let $M$ and $N$ be fibered cusped hyperbolic 3-manifolds with an isomorphism $\Phi:\wh{\pi}_1(M)\congto \wh{\pi}_1(N)$ on the profinite completions. 
Let $\mathcal{F}$ and $\mathcal{G}$ be a corresponding pair of fibered faces of the Thurston norm balls of $M$ and $N$, respectively, and suppose that the pseudo-Anosov monodromy maps of fibrations in $\mathcal{F}$ are fully-punctured. 
Then, so are those in $\mathcal{G}$, and 

{\rm (1)} The corresponding veering triangulations are both edge-orientable or both not edge-orientable, and 

{\rm (2)} Their taut {\rm(}$=$Teichm\"{u}ller{\rm)} 
polynomials $\Theta_\mathcal{F}$ and $\Theta_\mathcal{G}$ 
coincide via the induced isomorphism $\Phi_\star:H_M\congto H_N$ up to multiplication by units of $\Z[H_M]$, 
where $H_M=H_1(M;\Z)_{\rm free}$ and $H_N=H_1(N;\Z)_{\rm free}$. 
\end{theorem} 

The assertion (1) follows from \Cref{thm:transverse}: the edge-orientability is translated into the transverse orientability of the stable laminations of the suspension flows (\Cref{lem:eo-to}), and further into a condition on the monodromy maps (\Cref{lem:suspension}).

For the assertion (2), we verify that the edge-orientation homomorphisms $\omega_{\mathcal{F}}$ and $\omega_{\mathcal{G}}$ correspond via $\Phi$ (\Cref{lem:eo-compatible}), which allows us to apply \Cref{thm:twisted}. 
The proof of \Cref{lem:eo-compatible} is based on the celebrated theorem of Nikolov--Segal for finitely generated groups \cite{NikolovSegal2007AnnMath1}*{Theorem 1.1} (\Cref{lem:NS}) and the universality of the edge-orientation homomorphisms due to  Parlak \cite{Parlak2023JT}*{Lemma 4.9} (\Cref{lem:eo-universality}); 
a key step is to lift the three-way correspondence to a pair of double covers of $M$ and $N$, where the assertion (1) applies. 

Xu \cite{Xu2024fillings-arXiv}*{Theorem A} proved that a profinite isomorphism $\Phi\colon\wh{\pi}_1(M)\congto\wh{\pi}_1(N)$ of cusped hyperbolic $3$-manifold groups induces a bijective correspondence between Dehn fillings of $M$ and $N$, together with a profinite isomorphism on the results of each corresponding pair. 
Combining with our \Cref{thm:tautpoly}, we further obtain the following.

\begin{theorem}\label{thm:filled} 
Let the setting be as in {\rm \Cref{thm:tautpoly}}. Let $M'$ and $N'$ be hyperbolic $3$-manifolds obtained by a corresponding pair of Dehn fillings on $M$ and $N$, and suppose that 
$M'$ and $N'$ admit fibrations restricting to fibrations in $\R_+\mathcal{F}$ and in $\R_+\mathcal{G}$, respectively. 
Let $\mathcal{F}'$ and $\mathcal{G}'$ denote the fibered faces of these fibrations.  
Then the Teichm\"{u}ller polynomials $\Theta_{\mathcal{F}'}$ and $\Theta_{\mathcal{G}'}$ coincide via the induced isomorphism $\Phi'_\star\colon H_{M'}\congto H_{N'}$, up to multiplication by units of $\Z[H_{M'}]$.
\end{theorem}

\subsection{Application: Small normalized dilatations}  

When $M$ is the mapping torus of a fully-punctured pseudo-Anosov map, 
the number of tetrahedra in the veering triangulation is bounded by the dilatation 
(Tsang \cite{TsangCC2025-arXiv}, see \Cref{thm:Tsang}). 
Hence, our \Cref{thm:tautpoly} on the taut polynomials, combined with \Cref{thm:fully-punctured}, 
provides a practical strategy for detecting profinite rigidity in concrete cases (\Cref{strategy}), 
by using \emph{A census of veering structures} \cite{VeeringCensus}. 
For instance, we may derive the profinite rigidity of 17 examples of the mapping tori of fully-punctured pseudo-Anosov maps on once-punctured surfaces, 
including the following crisper new examples. 
Let $S_{g,n}$ denote the $n$-punctured orientable connected surface with genus $g$.

\begin{cor} \label{cor:rigidity}
The fundamental groups of the fully-punctured mapping tori presented by the following layered veering triangulations, listed with their fibers and taut polynomials, are profinitely rigid among $3$-manifold groups:
{\rm \begin{enumerate}
\item (dLQacccjsnk\_200, $S_{5,1}$, $a^{10} - a^9 + a^7 - a^6 + a^5 - a^4 + a^3 - a + 1$)
\item (eLPkaccddjnkaj\_2002, $S_{2,1}$, $a^4 + a^3 - a^2 + a + 1$)
\item (eLPkbcdddhrrcv\_1200, $S_{2,1}$, $a^4 - a^3 - a^2 - a + 1$)
\end{enumerate}} 
\end{cor}

These manifolds are one-cusped hyperbolic $3$-manifolds, 
with the normalized dilatations $(\text{Lehmer's number})^{9}\approx 4.311$ for (1) 
and $|LT_{1,2}|^{3}\approx 5.107$ for (2) and (3). 
Notably, (2) and (3) share the fiber and the dilatation, while their taut polynomials differ. 
See \Cref{sec:App} for the proof and background. 

\subsection*{Organization} 
The rest of this paper is organized as follows. 
In \Cref{sec:regular}, we recall profinite completions, goodness, and the regularity results (regular, $\wh{\Z}^\times$-regular, and peripheral regular) for hyperbolic 3-manifold groups. 
\Cref{sec:twisted} deals with twisted multivariable Alexander polynomials. 
We first recollect classical cases and discuss natural ambiguities. 
Next, we employ 1-goodness and faithful flatness to establish the 
result on complete ideals (\Cref{prop:order}, \Cref{thm:completeA}), 
and also derive results on classical invariants.
Finally, we prove the results under the regularity assumptions (\Cref{thm:twisted}, \Cref{cor:twist_alex.cusped}). 

\Cref{sec:veering} reviews the three-way correspondence among fibered faces of the Thurston norm ball, suspension pseudo-Anosov flows, and layered veering triangulations, and explains that the taut polynomial of a veering triangulation coincides with a twisted Alexander polynomial of the edge-orientation homomorphism. 
In \Cref{sec:punctured}, we prove that the property of a pseudo-Anosov map being fully-punctured is a profinite property of fibered faces (\Cref{thm:fully-punctured}). 
In \Cref{sec:transverse}, we extend a homological criterion of Rykken to punctured surfaces (\Cref{thm:Rykken-punctured}) and prove that the preservation of the transverse orientation of the invariant foliations by the monodromy map is a profinite invariant (\Cref{thm:transverse}). 
\Cref{sec:taut} combines the preceding results to establish that the taut and Teichm\"uller polynomials are profinite invariants (Theorems \ref{thm:tautpoly}, \ref{thm:filled}). 
Finally, \Cref{sec:App} applies these rigidity theorems together with small normalized dilatations and the veering census to produce explicit new examples of profinitely rigid hyperbolic one-cusped 3-manifolds (\Cref{cor:rigidity}).

\begin{remark} \label{rem:Xu-fullrigidity}
After this work was completed, 
Xu \cite{Xu2026-fullrigidity-arXiv} announced that every lattice in ${\rm PSL}_2\C$ is profinitely rigid among such lattices. Combined with the profinite detection of hyperbolicity (\Cref{ss.profiniterigidity}), this in particular implies \Cref{cor:rigidity}; see also \Cref{rem:Xu-q}. 
Our results were obtained independently, by a direct method.  
\end{remark}

\begin{ack}
The authors are grateful to 
Spencer Dowdall, Teruhisa Kadokami, Hisatoshi Kodani, Chris Leininger, Yi Liu, Jiming Ma, 
Tomoki Mihara, Yair Minsky, 
Anna Parlak,  
Alan Reid, Sohei Tateno, Sam Taylor, Lorenzo Traldi, 
and Xiaoyu Xu 
for useful comments and helpful discussions. 
In particular, they thank Anna Parlak for her correction to an earlier version of this paper. 
A part of this work was carried out while J.\,U. and B.\,M. were visiting the Agol Lab at ICMAT, Madrid, in June 2023. They would like to thank Andrei Jaikin-Zapirain for his warm hospitality.
Part of this work was also done when Y.\,Y. and T.\,C. visited the Simons Laufer Mathematical Sciences Institute in March 2026. They thank the organizers of the {\it Topological and Geometric Structures in Low Dimensions} program for their hospitality and welcome.
J.\,U. has been partially supported by JSPS KAKENHI Grant Number JP23K12969. 
B.\,M. is partially supported by the National Natural Science Foundation of China (Grant No. 12401086). 
\end{ack}

\section{Profinite completions and regularities} \label{sec:regular}

\subsection{The profinite completion} \label{ss.hatpi}

Let $\pi$ be a group. The set of finite quotients $\{\pi/\Gamma\mid \Gamma\lhd \pi,\ [\pi:\Gamma]<\infty\}$ forms a directed inverse system with respect to the quotient maps. 
\emph{The profinite completion} $\wh{\pi}$ of $\pi$ is the inverse limit $\varprojlim \pi/\Gamma$ 
endowed with the relative topology of the direct product topology of discrete finite groups, 
namely, the weakest topology of a topological group such that every projection $\wh{\pi}\surj \pi/\Gamma$ is continuous. 
A topological group is said to be a \emph{profinite group} if it is the profinite completion of some group, 
or equivalently, if it is totally disconnected and compact Hausdorff. 
It is known that the system of finite quotients may be reconstructed from the set of isomorphism classes of finite quotients (\cite{DixonFormanekPolandRibes1982JPAA}, 
\cite{RibesZalesskii2010}*{Corollary 3.2.8}).

Let $\iota:\pi\to \wh{\pi}$ denote the natural map. 
For any group homomorphism $v:\pi\to G$ to a finite group,  
there is a unique group homomorphism $\wh{v}:\wh{\pi}\to G$ satisfying $v=\wh{v}\circ \iota$, 
namely, $v$ uniquely factors through $\iota$. 
Moreover, for any group homomorphism $v:\pi\to G$ to a profinite group, 
there is a unique homomorphism $\wh{v}:\wh{\pi}\to G$ of topological groups satisfying $v=\wh{v}\circ \iota$. 
This property is called \emph{the universality of the profinite completion}. 
Conversely, 
\begin{lemma} \label{lem:NS}
Suppose that $\pi$ is finitely generated. 
Then, for any surjective group homomorphism $\wh{v}:\wh{\pi}\surj G$ to a finite group, the composition $v=\wh{v}\circ \iota:\pi\xrightarrow{\iota} \wh{\pi}\to G$ is again surjective. 
\end{lemma}
\begin{proof} 
The Nikolov--Segal theorem \cite{NikolovSegal2007AnnMath1}*{Theorem 1.1} asserts that every surjective group homomorphism from a topologically finitely generated profinite group to a profinite group is continuous, so the given $\wh{v}$ is continuous.  
The image of a dense subset $\iota(\pi)$ of $\wh{\pi}$ via the continuous map $\wh{v}$ is again dense, and a dense subset of a finite group $G$ is $G$ itself, so $v=\wh{v}\circ \iota$ is surjective. 
\end{proof}

If $\pi$ is not finitely generated, then there may exist a finite quotient of $\wh{\pi}$ which is not a finite quotient of $\pi$. 
A typical non-finitely generated example is the product $\pi$ of infinitely many copies of $\Z/2\Z$.

\subsection{Residual finiteness and goodness} \label{ss.good} 

A discrete group $\pi$ is said to be \emph{residually finite} if the natural map $\iota$ is injective. This is a natural assumption for the question of profinite rigidity. 
For any 3-manifold $M$, the group $\pi=\pi_1(M)$ is residually finite (Hempel \cite{Hempel1987} + Perelman's solution of the geometrization conjecture). 

A discrete group $\pi$ is said to be \emph{$n$-good}  
if the natural map $\iota^\ast:H^\ast(\wh{\pi};V)\to H^\ast(\pi;V)$ 
is an isomorphism for every finite $\pi$-module $V$ and $\ast\leq n$.  
A discrete group $\pi$ is said to be \emph{good} (in the sense of Serre, \cite{Serre-CohGal}*{Chap.\,I, \S 2.6, Exercise 2)}) if it is $n$-good for every $n\in \Z_{\geq 0}$.

\begin{lemma}[Serre, \cite{Serre-CohGal}*{Chap.\,I, \S 2.6, Exercise 1)-(b)}] \label{lem:1-good}
Any group is 1-good. 
\end{lemma} 

The fundamental group of any compact 3-manifold is known to be good, 
after deep results of Agol and Wise (cf.\,\cite{GJZPZ2014}*{Remark 5.14}). 

If $\pi$ is $n$-good, the profinite completion of a resolution $\mathscr{F}_\ast$ of $\Z$ over $\Z[\pi]$ is a resolution $\wh{\mathscr{F}}_\ast$ of $\wh{\Z}$ over $\wh{\Z}[\![\wh{\pi}]\!]$ up to $\ast\leq n$ 
(cf.\,\cite{JaikinZapirain2020GT}*{Proposition 3.2}, \cite{HughesKielak2025RMI}*{Proposition 4.6 + Remark 4.3}). 

A group $\pi$ is said to be FP$_n$ (resp.\,FP$_\infty$) if it admits a resolution $\mathscr{F}_\ast$ whose terms with $\ast\leq n$ (resp.\,all $n$) are finitely generated. 
A group $\pi$ is FP$_1$ if and only if it is finitely generated. A finitely presented group is FP$_2$, and a group of finite type is FP$_\infty$; the converse implications do not hold in general.

\subsection{Abelianization and regularities} 

The profinite completion is compatible with the abelianization, 
namely, there is a natural isomorphism $(\wh{\pi})^{\rm ab}\cong (\pi^{\rm ab})^{\wedge}$, 
so we may simply write $\wh{\pi}^{\rm ab}$ for both (cf.\,\cite{Neukirch}*{Chapter IV, Section 1, Exercise 5}). 
The following notions were initially introduced by Boileau--Friedl \cite{BoileauFriedl2020AMS} (preprint 2015) and Liu \cite{YiLiu2023Invent} 
to state the profinite rigidity of the Alexander polynomials under ``a mild condition'' and to discuss a milder one, respectively. 

\begin{defn} 
Let $\pi$ and $\varpi$ be groups with an isomorphism $\Phi:\wh{\pi}\congto \wh{\varpi}$ on their profinite completions.
Note that the induced isomorphism $\Phi^{\rm ab}:\wh{\pi}^{\rm ab}\congto \wh{\varpi}^{\rm ab}$ on the abelianizations may be regarded as that of $\wh{\Z}$-modules. 

If $\Phi^{\rm ab}$ is the profinite completion of an isomorphism $\pi^{\rm ab}\congto \varpi^{\rm ab}$ on the abelianizations of the original groups, then $\Phi$ is said to be \emph{regular}. 

If $\Phi^{\rm ab}$ is the composition of the profinite completion of some isomorphism $\pi^{\rm ab}\congto \varpi^{\rm ab}$ and the multiplication by some unit $u\in \wh{\Z}^\times$, then $\Phi$ is said to be \emph{$\wh{\Z}^\times$-regular}. 
\end{defn}

In addition, Xu \cite{Xu2025regularity-arXiv} introduced the following to refine the notion. 
\begin{defn} 
If $M$ and $N$ are compact manifolds with incompressible toroidal boundary, $\Phi:\wh{\pi}_1(M)\congto \wh{\pi}_1(N)$ is said to be \emph{peripheral regular} if it respects the peripheral structures; 
namely, the profinite isomorphism $\wh{\pi}_1(\partial_i M)\congto \wh{\pi}_1(\partial_j N)$ on the boundaries is the completion of some $\pi_1(\partial_i M)\congto \pi_1(\partial_j N)$ up to conjugation.
\end{defn}

We will make use of the following fundamental results in the study of profinite completions of hyperbolic 3-manifold groups. 

\begin{theorem}[{Liu \cite{YiLiu2023Invent}*{Theorem 1.2}, Xu \cite{Xu2025regularity-arXiv}*{Theorem 1.4}, Hanany \cite{Hanany2026regularity-arXiv}*{Theorem 1.1}}] \label{thm:regularity}
If $M$ and $N$ are finite-volume hyperbolic $3$-manifolds, then any profinite isomorphism $\Phi:\wh{\pi}_1(M)\congto \wh{\pi}_1(N)$ is regular. 
Moreover, if in addition both $M$ and $N$ are cusped, then $\Phi$ is peripheral regular. 
\end{theorem}

To specify the contributions: Liu established the $\wh{\Z}^\times$-regularity in both the closed and the cusped cases, 
Xu obtained regularity and peripheral regularity in the cusped case, 
and Hanany deduced the regularity in the closed case from Liu's result by purely algebraic and number-theoretic arguments.  

\section{Twisted multivariable Alexander Polynomials}\label{sec:twisted} 

\subsection{Twisted Alexander modules} \label{ss.twisted}

\subsubsection{Definition} \label{sss.def.twisted}

Let $H$ be a free abelian multiplicative group of rank $r\in \Z_{>0}$ with a basis $(t_1,\ldots,t_r)$ and put $\Lambda=\Z[H]=\Z[t_1^\Z,\ldots,t_r^\Z]$. 
Let $\pi$ be a group with a surjective homomorphism $\alpha:\pi\surj H$ and a representation $\rho:\pi\to {\rm GL}_k\Z$. Put $\Gamma={\rm Ker}\,\alpha$. 

Let $V_\rho$ denote the module $\Z^k$ regarded as a right $\pi$-module via the transpose of $\rho$. 
Then $\Gamma$ acts on $V_\rho$ via the restriction of $\rho$. 
Let $\mathscr{P}_\ast$ be a projective resolution of $\Z$ over $\Z[\Gamma]$. Then the homology $H_\ast(\Gamma,V_\rho)$ with local coefficients is defined to be that of the complex $\mathscr{C}_\ast(\Gamma,\rho)=\mathscr{P}_\ast\otimes_{\Z[\Gamma]} V_\rho$. We write $H_\ast(\Gamma,\rho)=H_\ast(\Gamma, V_\rho)$ for simplicity. 
The conjugate action of $\pi$ on $\Gamma$ induces a natural action of $H$ on $H_\ast(\Gamma,\rho)$, and $H_\ast(\Gamma,\rho)$ becomes a $\Lambda$-module (see \cite{Brown1982GTM}*{Chapter III, Corollary 8.2}). 

Let $M$ be a connected CW complex with $\pi_1(M)=\pi$ and let $M_\alpha\to M$ denote the $H$-cover corresponding to $\Gamma$. Then, for $\ast=0,1$, the Hurewicz theorem ensures that there is a natural isomorphism $H_\ast(\Gamma,\rho)\cong H_\ast(M_\alpha,\rho)$ of finitely generated $\Lambda$-modules. 
Note that Shapiro's lemma \cite{Brown1982GTM}*{Chapter III, Proposition 6.2} yields that $H_\ast(\Gamma,\rho)\cong H_\ast(\pi,\rho\otimes \alpha)$ and $H_\ast(M_\alpha,\rho)\cong H_\ast(M,\rho\otimes \alpha)=H_\ast(\mathscr{C}_\ast(\widetilde{M})\otimes_{\Z[\pi]}\Lambda^k)$ for the tensor representation $\rho\otimes \alpha:\pi\to {\rm GL}_k\Lambda$, where  $\mathscr{C}_\ast(\widetilde{M})$ denotes the $\pi$-module, which is a lifted chain complex of the universal cover $\widetilde{M}$ of $M$. 
These four isomorphic $\Lambda$-modules are all called \emph{the $\ast$-th twisted Alexander module} of $(\rho,\alpha)$. 

\begin{defn}
If the 1st twisted Alexander module $\mathcal{A}_\rho^\alpha = H_1(\Gamma,\rho)$ of $(\rho,\alpha)$ is a torsion $\Lambda$-module, 
then its \emph{order element}, namely, a generator of the \emph{smallest} principal ideal containing 
the 0-th Fitting ideal, is called \emph{the multivariable twisted Alexander polynomial} $\Delta_\rho^\alpha(t_1,\ldots,t_r)$. 
\end{defn} 
So, $\Delta_\rho^\alpha$ is well-defined up to multiplication by units of $\Lambda$. The equality up to this ambiguity will be denoted by $\doteq$. 
A basic reference is Friedl--Vidussi's survey \cite{FriedlVidussi2011survey}. 

\begin{example} Let $\bm{1}$ denote the trivial ${\rm GL}_1\Z$-representation. 
For an $r$-component link $L=K_1\sqcup\cdots \sqcup K_r$ in $S^3$ with $\pi=\pi_L=\pi_1(S^3-L)$, the classical multivariable Alexander polynomial $\Delta_L$ satisfies $\Delta_L(t_1,\ldots,t_r)\doteq \Delta_{\bm{1}}^{\rm ab}(t_1,\ldots,t_r)$, where  
${\rm ab}$ denotes the maximal abelianization map of $\pi_L$, and $t_i$'s denote the images of the meridians of $K_i$'s in $\pi_L^{\rm ab} \cong \Z^r$. 
Hence the polynomials coincide up to order changing and orientation reversing. 
\end{example} 

\subsubsection{Comparison} 
A straightforward argument shows the following. 

\begin{prop} \label{prop:GLkZ}
Let $\rho:\pi\to {\rm GL}_k \Z$ and $\alpha:\pi\surj H$ with $(t_1,\ldots, t_r)$ be as above 
in {\rm \Cref{sss.def.twisted}}. 
Let $\beta:\varpi\surj K$ be another surjective group homomorphism, where $K$ is a free abelian multiplicative group of rank $r$ with a fixed basis $(s_1,\ldots,s_r)$. 
Given a pair $(\Phi,\Phi_\ast)$ of group isomorphisms $\Phi:\pi\congto \varpi$ and $\Phi_\ast:H\congto K$ such that $\beta\circ \Phi=\Phi_\ast\circ \alpha$ holds, 
define $\varrho=\rho\circ \Phi^{-1}:\varpi\to {\rm GL}_k\Z$. 
Write $\Phi_\ast(t_i)=t_i$ by abuse of notation. Then there is an isomorphism 
\[H_\ast(\pi,\rho\otimes \alpha)\cong H_\ast(\varpi,\varrho\otimes \beta)\]
of $\Lambda$-modules. 

Write $s_i=t_1^{v_{i1}}\cdots t_r^{v_{ir}}$ for each $1\leq i\leq r$, so that $V=(v_{ij})\in {\rm GL}_r\Z$ presents an isomorphism $\Z^r\congto \Z^r$ determined by $\Phi_\ast$ and the fixed bases of $K$ and $H$. 
Then, we have 
\[\Delta_\rho^\alpha(t_1,\ldots, t_r)\doteq \Delta_\varrho^\beta(t_1^{v_{11}}\cdots t_r^{v_{1r}},\ldots, t_1^{v_{r1}}\cdots t_r^{v_{rr}})\]
in $\Lambda$. 
Namely, $\Delta_\rho^\alpha(t_1,\ldots,t_r)$ and $\Delta_\varrho^\beta(s_1,\ldots,s_r)$ coincide up to the ${\rm GL}_r\Z$-action on $\Z^r$ and multiplication by units of $\Lambda$. 
\end{prop}

\begin{example}
(1) When $\rho$ is the trivial ${\rm GL}_1\Z$-representation $\bm{1}$, then we always have $\varrho=\rho\circ \Phi^{-1}=\bm{1}$. 

(2) In the proof of \Cref{thm:tautpoly}, we will discuss the coincidence of nontrivial ${\rm GL}_1\Z$-representations called the edge-orientation homomorphisms by using a certain universality. 
\end{example}

\Cref{prop:GLkZ} implies that the isomorphism class of the link group $\pi_L$ of an $r$-component link determines $\Delta_L$ only up to the ${\rm GL}_r\Z$-action. 

\begin{example} \label{eg:Alex} 
Let $J=4^2_1=L4a1=T(2,4)=J_1\sqcup J_2$ be Solomon's link in $S^3$ with an orientation so that  
$\Delta_J(t_1,t_2)=t_1t_2+1$. 

(1) Let $L$ be the same link with the orientation of $J_2$ reversed. Then $\pi_J\cong \pi_L$ and $\Delta_L(t_1,t_2)\doteq t_1+t_2$. 

(2) Let $n \in \Z$. 
Performing $1/n$-surgery along $J_2$, we obtain a link $L_n$ in $S^3$ 
consisting of the unknot and the torus knot $T(2,1-2n)$. 
By ${\rm lk}(J_1,J_2)=2$, we obtain 
$\Delta_{L_n}(t_1,t_2)\doteq \Delta_{J}(t_1,t_1^{-2n}t_2)=t_1^{1-2n}t_2+1$, 
whereas $\pi_J\cong \pi_{L_n}$. 
No combination of variable inversions or swaps can transform $\Delta_J$ into $\Delta_{L_n}$ if $n\neq 0,1$. 
\end{example}

\subsubsection{Hyperbolic link exteriors} 
\label{sss.hyplinks} 
In \Cref{prop:GLkZ}, when we start only with the existence of an isomorphism $\Phi:\pi\congto \varpi$, 
one might find the ambiguity substantial. Here, we try to reduce this ambiguity.

For cusped hyperbolic 3-manifolds, the Mostow--Prasad rigidity theorem
\cites{Mostow1973AnnMS,Prasad1973Invent,Marden1974} 
ensures the correspondence of the peripheral systems of the cusps up to conjugacy (cf.\,\cite{AschenbrennerFriedlWilton2015}*{Theorem 1.7.2}). 
For instance, let $J=J_1\sqcup\cdots\sqcup J_r$ and $L=L_1\sqcup \cdots \sqcup L_r$ be hyperbolic links in $S^3$ and put $M=S^3-J$ and $N=S^3-L$. 
For each $i$, let $T_i$ denote the peripheral torus of $J_i$, and recall that the preferred longitude of $J_i$ has the image $\sum_{j\neq i}{\rm lk}(J_i,J_j)\,t_j$ in $H_1(M;\Z)=\Z^r$, where $t_j$ denotes the class of the meridian of $J_j$. 
Hence every 2-component sublink of $J$ has linking number zero if and only if the image of $\pi_1(T_i)$ in $H_1(M;\Z)$ is of rank one for every $i$, that is, if and only if this image is generated by the class of the meridian for every $i$. 
The latter condition depends only on the group together with the peripheral structure, 
and the peripheral structures of $M$ and $N$ correspond under $\Phi$ by the Mostow--Prasad rigidity theorem. 
Thus, if $J$ satisfies this condition, then so does $L$, and each meridian of a component of $J$ corresponds to a meridian of a component of $L$ up to sign. 
We remark that, by the Torres formula, this condition implies $\prod_i(t_i-1)\mid \Delta_J$. 
Consequently, the induced map $\Phi_*:\pi_1(M)^{\rm ab}_{\rm free}\congto \pi_1(N)^{\rm ab}_{\rm free}$ is presented by a signed permutation matrix. Namely, we have 

\begin{prop} \label{prop.hyplinks} 
Let the setting be as in {\rm \Cref{prop:GLkZ}}, and suppose that $\pi=\pi_1(M)$ and $\varpi=\pi_1(N)$ for the exteriors $M$ and $N$ of hyperbolic links $J$ and $L$ in $S^3$, 
and that $\alpha$ and $\beta$ are the abelianization maps with the bases given by the meridians. 
If every 2-component sublink of $J$ has linking number zero, then $L$ satisfies the same condition, and
\[\Delta_\rho^\alpha(t_1,\ldots, t_r)\doteq 
\Delta_\varrho^\beta(t_{\sigma(1)}^{\varepsilon_1},\ldots, t_{\sigma(r)}^{\varepsilon_r})\]
holds for some $\sigma\in \mca{S}_r$ and $\varepsilon_i\in \{\pm 1\}$.
\end{prop} 

\begin{example} \label{eg.Borromean}
Let $0\neq k,l\in \Z$. 
The Borromean rings $B=K_1\sqcup K_2\sqcup K_3$ form a 3-component hyperbolic link with 
${\rm lk}(K_i,K_j)=0$ for every $i\neq j$. 
The $(-1/k)$-filling along a component yields the odd twisted Whitehead link $W_{2k-1}$, 
which is a 2-component hyperbolic link with linking number zero. 
The $(-1/l)$-filling along a component of $W_{2k-1}$ yields the double twist knot $J(2k,2l)$ in $S^3$ (cf.\,\cite{BenardTangeTranUeki-Whitehead}). 
We have $\Delta_B(t_1,t_2,t_3)=(t_1-1)(t_2-1)(t_3-1)$, 
$\Delta_{W_1}(t_1,t_2)\doteq (t_1-1)(t_2-1)$, 
$J(2,2)=3_1$, $J(2,-2)=4_1$. 
Although $\Delta_{3_1}(t)=t^2-t+1$ and $\Delta_{4_1}(t)=t^2-3t+1$, 
the polynomials $\Delta_B$ and $\Delta_{W_1}$ persist under 
the $(-1/k,-1/l)$-surgery on $B$ and the $(-1/l)$-surgery on $W_1$. 
Namely, the $3$-component link obtained from $B$ by the $(-1/k,-1/l)$-surgery on two components, which contains $J(2k,2l)$ as a component, has the exterior homeomorphic to that of $B$ and the Alexander polynomial $\Delta_B$, in accordance with \Cref{prop.hyplinks}; the same holds for $W_1$. 
\end{example}

\subsection{Complete twisted Alexander modules} \label{ss.hattwisted} 
\subsubsection{Complete modules}
Let the setting be as before in \Cref{sss.def.twisted}. 
For the arbitrarily given group representation $\rho:\pi\to {\rm GL}_k\Z$, 
there is an induced continuous representation 
\[\wh{\rho}:\wh{\pi}\to {\rm GL}_k\wh{\Z}.\] 
Indeed, as $n$ runs through $\Z_{>0}$, the family $(\rho\mod n:\pi\to {\rm GL}_k\Z/n\Z)_n$ defines a representation 
$\rho:\pi\to {\rm GL}_k\wh{\Z}\cong \varprojlim_n {\rm GL}_k \Z/n\Z$.
By the universality of the profinite completion \cite{RibesZalesskii2010}*{Lemma 3.2.1}, this uniquely extends to a continuous representation $\wh{\rho}$. 
If $\rho$ has a finite image, then ${\rm Im}\,\rho={\rm Im}\,\wh{\rho}$. 

Let $\wh{\Lambda}=\wh{\Z}[\![\wh{H}]\!]=\wh{\Z}[\![t_1^{\wh{\Z}},\ldots,t_r^{\wh{\Z}}]\!]$ denote the complete group ring, that is, the profinite completion of $\Lambda$. 
Then the continuous group homology $H_\ast(\wh{\Gamma}, \wh{\rho})$ is defined by the complex  $\mathscr{C}_\ast(\wh{\Gamma},\wh{\rho})
=\wh{\mathscr{P}}_\ast \wh{\otimes}_{\wh{\Z}[\![\wh{\Gamma}]\!]} V_{\wh{\rho}}$, 
where $\wh{\mathscr{P}}_\ast$ is a projective resolution of $\wh{\Z}$ over $\wh{\Z}[\![\wh{\Gamma}]\!]$.   
Since ${\rm Ker}\,(\wh{{\rm GL}}_k \Z\to {\rm GL}_k \wh{\Z})$ acts trivially on $\wh{\Z}^k$, 
this definition is a natural one. 

By \cite{RibesZalesskii2010}*{Proposition 6.5.7}, there is an isomorphism 
\begin{align*}
H_\ast(\wh{\Gamma}, \wh{\rho}) \cong \varprojlim_{(G,n)} H_\ast(\Gamma/G, \rho\mod n)
\tag{$\star$}
\end{align*} 
of $\wh{\Lambda}$-modules, where $(G,n)$ runs through the pairs of a finite-index normal subgroup $G$ of $\Gamma$ and $n\in \Z_{>0}$ with $G\subset {\rm Ker}(\rho\mod n)$. 
We define \emph{the completed 1st twisted Alexander module} to be $\wh{\mathcal{A}}_\rho^\alpha=H_1(\wh{\Gamma}, \wh{\rho})$. 

Now let $\ast \leq 1$. 
The five-term exact sequence yields a natural surjective homomorphism $H_\ast(\Gamma,\rho)\to H_\ast(\Gamma/G, \rho\mod n)$ for every such pair $(G,n)$.
Hence ($\star$), \Cref{lem:1-good} on 1-goodness, and the right-exactness of the profinite completion imply that  
\begin{lemma} \label{lem:share} For $\ast \leq 1$, 
$H_\ast(\wh{\Gamma},\wh{\rho})$ is the profinite completion of $H_\ast(\Gamma, \rho)$, 
and they share the same presentation matrix. 
\end{lemma} 

\subsubsection{Faithfully flat ring extensions}

Let $R$ be a ring. An $R$-module $V$ is said to be \emph{flat} if $\otimes_R V$ is exact, 
namely, it preserves injective homomorphisms. 
In addition, $V$ is said to be \emph{faithfully flat} if 
$W\otimes_R V=0$ implies $W=0$ for an $R$-module $W$, or equivalently, 
if $\mf{m}V\subsetneq V$ holds for every maximal ideal $\mf{m}$ of $R$ 
(cf.\,\cite{Bourbaki1961AlgCommChap1-2}*{Chap.\,I, \S 2, Prop.\,1, \S 3, Prop.\,1}).
 
Similar terminology applies to ring extensions. 
If a ring extension $R\to R'$ is flat, 
then every ideal $\mf{a}$ of $R$ satisfies $\mf{a}\otimes_R R'\cong \mf{a}R'$. 
If it is faithfully flat, then 
an inclusion $\mf{a}\subsetneq \mf{b}$ of ideals of $R$ implies $\mf{a}R'\subsetneq \mf{b}R'$. 
In other words, for any ideal $\mf{a}$ of $R$, $\mf{a}R'\cap R=\mf{a}$ holds. 

Non-zero free $R$-modules are faithfully flat 
(\cite{Bourbaki1961AlgCommChap1-2}*{Chap.\,I, \S 3, Exem.\,2, below D\'ef.\,1}). 
In particular, the maps of group rings induced by group inclusions are faithfully flat (see also \cite{Bourbaki1961AlgCommChap1-2}*{Chap.\,I, \S 3, Prop.\,9 b}).

\begin{lemma} \label{lem:hatLambdaisflat} 
The ring extension $\Lambda \to \wh{\Lambda}=\wh{\Z}[\![t_1^{\wh{\Z}},\ldots, t_r^{\wh{\Z}}]\!]$ is faithfully flat. 
\end{lemma} 
\begin{proof} \ul{Flatness.} 
For each prime number $p$, let $\C_p=\wh{\ol{\Q}}_p$ denote the completion of an algebraic closure of $\Q_p$ and fix an embedding $\ol{\Q}\inj \C_p$ of an algebraic closure of $\Q$. 
Then we have 
\begin{align*}
\wh{\Lambda}&=\prod_p \Z_p[\![t_1^\wh{\Z},\ldots,t_r^\wh{\Z}]\!] 
\cong \prod_p \Z_p[\![(\prod_{l\neq p}{\Z_l})^r]\!][\![t_1^{\Z_p},\ldots,t_r^{\Z_p}]\!]\\
&\cong \prod_p \prod_{\chi} \Z_p[{\rm Im}\,\chi][\![T_1,\ldots,T_r]\!]
\end{align*}
where $\chi:(\prod_{l\neq p}\Z_l)^r\to \C_p^\times$ runs through all continuous characters up to ${\rm Gal}(\ol{\Q}_p/\Q_p)$-conjugacy.
The last isomorphism combines 
the character decomposition and the multivariable Iwasawa--Serre isomorphism 
\[\Z_p[{\rm Im}\,\chi][\![t_1^{\Z_p},\ldots,t_r^{\Z_p}]\!]\congto \Z_p[{\rm Im}\,\chi][\![T_1,\ldots,T_r]\!];\ t_i\mapsto 1+T_i\] 
to the ring of $p$-adic formal power series (cf.\,\cites{Greenberg1973AJM,TatenoUeki2025JLMS}). 

Let $\bm{\mu}'$ denote the set of roots of unity of order prime to $p$ in $\C_p$. 
Then ${\rm Im}\,\chi \subset \bm{\mu}'$. 
Since $\zeta,\xi \in \bm{\mu}'$ with $\zeta\neq \xi$ satisfy $|\zeta-\xi|_p=1$, 
where $|\cdot|_p$ denotes the $p$-adic absolute value, the subset $\bm{\mu}'$ is discrete in $\C_p$. 
Since $(\prod_{l\neq p}\Z_l)^r$ is compact, $\chi$ is continuous, and $\bm{\mu}'$ is discrete, we see that the image ${\rm Im}\,\chi$ must be finite. 
Thus $\Z_p[{\rm Im}\,\chi]$ is a finite extension of $\Z_p$ by adjoining some elements of $\bm{\mu}'$. 

In general, the product of flat modules over a Noetherian ring is flat 
(cf.\,\cite{Bourbaki1961AlgCommChap1-2}*{Chap.\,I, \S 2, exerc.\,12 b)}, 
\cite{CartanEilenbergbook}*{Exer.\,VI.4}). 
Each $\Z_p[{\rm Im}\,\chi][\![T_1,\ldots,T_r]\!]$ is the $(\mathfrak{p},t_1-\zeta_1,\ldots,t_r-\zeta_r)$-adic completion of the free $\Lambda$-module $\Z[{\rm Im}\,\chi][t_1^\Z,\ldots,t_r^\Z]$, where $\mathfrak{p}$ is the prime of $\Z[{\rm Im}\,\chi]$ over $p$ given by the embedding $\ol{\Q}\inj \C_p$ and $\zeta_i$ is the image under $\chi$ of the $i$-th standard generator of $(\prod_{l\neq p}\Z_l)^r$, hence a flat $\Lambda$-module (\cite{AtiyahMacdonald}*{Proposition 10.14}). Since $\Lambda$ is a Noetherian ring, we conclude that $\wh{\Lambda}$ is flat.

\ul{Faithful flatness.} Let $\mf{m}$ be a maximal ideal of $\Lambda$. 
Since $\Lambda$ is a finitely generated $\Z$-algebra, the general Nullstellensatz ensures that the residue field $\Lambda/\mf{m}$ is a finite field $\F_q$. 
The induced maps $\Z\to \F_q$ and $H\to \F_q^\times$ have finite images, so they extend to continuous maps $\wh{\Z}\to \F_q$ and $\wh{H}\to \F_q^\times$ by the universality of the profinite completion, 
and hence the projection $\Lambda\surj \F_q$ extends to a continuous ring homomorphism $\wh{\Lambda}\surj \F_q$. 
Therefore $\mf{m}\wh{\Lambda}\subsetneq \wh{\Lambda}$ holds for every maximal ideal $\mf{m}$ of $\Lambda$. 
\end{proof} 

\begin{lemma} \label{lem:oderideal} 
If an ideal $\mathfrak{a}$ of $\wh{\Lambda}$ is 
finitely generated by elements of $\Lambda'=\Z[t_1^\wh{\Z},\ldots,t_r^\wh{\Z}]$, then 
$\mathfrak{a}$ has a well-defined order ideal in $\wh{\Lambda}$,
namely, the smallest principal ideal of $\wh{\Lambda}$ that contains $\mathfrak{a}$ and is generated by some element $\Delta\in \Lambda'$.
\end{lemma}

\begin{proof} 
Consider the ring extensions $\Lambda \to \Lambda' \to \wh{\Lambda}$. 
The extension $\Lambda \to \Lambda'$ is faithfully flat, since it is a group ring inclusion. 
Combined with \Cref{lem:hatLambdaisflat}, 
the base change property of faithful flatness for commutative algebras
(cf.\,\cite{Bourbaki1961AlgCommChap1-2}*{Chap.\,I, \S 3, Prop.\,6}) 
shows that \emph{$\Lambda' \to \wh{\Lambda}$ is also faithfully flat}. 

Now, let $\{f_1, \ldots, f_m\} \subset \Lambda'$ be a generating set of $\mathfrak{a}$ 
and let $A=\Z[W]$ denote the group ring of the subgroup $W\subset \wh{\Z}^r$ generated by $\Z^r$ together with the exponents of all the monomials appearing in $f_1,\ldots,f_m$. 
Since $W$ is a finitely generated torsion-free abelian group, hence free of finite rank, the ring $A$ is a Laurent polynomial ring over $\Z$ and in particular a UFD. 
Therefore the order ideal of $(f_1, \ldots, f_m)$ in $A$ is well-defined and generated by the greatest common divisor $\Delta \in A$. 
Since the group ring inclusion $A \to \Lambda'$ is faithfully flat, 
it preserves the divisibility relations among elements of $A$. 
Moreover, if $g\in \Lambda'$ divides every $f_i$, then $g\in A'=\Z[W']$ for some finitely generated $W'\supset W$, and the greatest common divisor of $f_1,\ldots,f_m$ in $A'$ is again $\Delta$, since coprimality in $\Z[W]$ is preserved under $\Z[W]\to \Z[W']$; here the torsion part of $W'/W$ is handled by a norm argument. Hence $g\mid \Delta$ in $\Lambda'$. 
Consequently, the ideal $(\Delta)\Lambda'$ serves as the well-defined order ideal of $\mathfrak{a} \cap \Lambda'$ in $\Lambda'$.

Finally, since $\mathfrak{a} = (\mathfrak{a} \cap \Lambda')\wh{\Lambda}$, 
the faithful flatness of $\Lambda'\to \wh{\Lambda}$ guarantees that the divisibility and inclusion relations in $\Lambda'$ are precisely reflected in $\wh{\Lambda}$.
Therefore, the principal ideal $(\Delta)\wh{\Lambda}$ is the unique well-defined order ideal of $\mathfrak{a}$ in $\wh{\Lambda}$, completing the proof.
\end{proof} 

\subsubsection{Complete Alexander ideals}

By Lemmas \ref{lem:share}, \ref{lem:oderideal}, we may conclude: 

\begin{prop} \label{prop:order}
$\Delta_\rho^\alpha\in \Lambda$ is also an order element of $H_1(\wh{\Gamma},\wh{\rho})$ over $\wh{\Lambda}$. 
\end{prop} 

The following is an analogue of \Cref{prop:GLkZ} for completed modules. 

\begin{theorem} \label{thm:completeA} 
Let $\rho:\pi\to {\rm GL}_k \Z$ and $\alpha:\pi\surj H$ with $(t_1,\ldots, t_r)$ be as in {\rm \Cref{sss.def.twisted}}. 
Let $\beta:\varpi\surj K$ be another surjective group homomorphism, 
where $K$ is a free abelian multiplicative group of rank $r$ with a fixed basis $(s_1,\ldots,s_r)$. 
Given a pair $(\Phi,\Phi_\ast)$ of isomorphisms $\Phi: \wh{\pi}\congto \wh{\varpi}$ and $\Phi_\ast:\wh{H}\congto \wh{K}$ such that $\wh{\beta}\circ \Phi=\Phi_\ast\circ \wh{\alpha}$ holds, 
let $\varrho:\varpi\to {\rm GL}_k\Z$ be another representation with $\wh{\varrho}=\wh{\rho}\circ \Phi^{-1}$. 
Write $\Phi_\ast(t_i)=t_i$ by abuse of notation. Then there is an isomorphism 
\[H_\ast(\wh{\pi},\wh{\rho}\otimes \wh{\alpha})\cong H_\ast(\wh{\varpi},\wh{\varrho}\otimes \wh{\beta})\]
of $\wh{\Lambda}$-modules. 

Write $s_i=t_1^{v_{i1}}\cdots t_r^{v_{ir}}$ for each $1\leq i\leq r$, so that $V=(v_{ij})\in {\rm GL}_r\wh{\Z}$ presents an isomorphism $\wh{\Z}^r\congto \wh{\Z}^r$ determined by $\Phi_\ast$ and the fixed bases of $\wh{K}$ and $\wh{H}$. Then the equality
\[(\Delta_\rho^\alpha(t_1,\ldots, t_r))=(\Delta_\varrho^\beta(t_1^{v_{11}}\cdots t_r^{v_{1r}},\ldots, t_1^{v_{r1}}\cdots t_r^{v_{rr}}))\]
of ideals of $\wh{\Z}[\![t_1^{\wh{\Z}},\ldots,t_r^{\wh{\Z}}]\!]$ holds. 
Namely, $\Delta_\rho^\alpha(t_1,\ldots,t_r)$ and $\Delta_\varrho^\beta(s_1,\ldots,s_r)$ coincide up to the natural ${\rm GL}_r\wh{\Z}$-action on $\wh{\Z}^r$ induced by $\Phi$ and multiplication by units of $\wh{\Lambda}$. 
\end{theorem}

\begin{proof} 
The former half is obvious. 
By our definition of the order ideals in \Cref{lem:oderideal}, 
the isomorphism of $\wh{\Lambda}$-modules implies the coincidence of the order ideals. 
By \Cref{prop:order}, they are generated by the Alexander polynomials. 
\end{proof} 

\begin{remark} \label{rem.good} 
(1) In the above proof, we could also work over a resolution for $\pi$, instead of that for $\Gamma$. 

(2) A similar argument proves the following: If both $\pi$ and $\varpi$ are FP$_n$ and $n$-good, then 
their $\ast$-th Alexander ideals coincide in the sense of \Cref{thm:completeA} for every $\ast\leq n$. 
\end{remark} 


\subsection{Linking numbers and Mahler measures} 
Immediate consequences of \Cref{thm:completeA} are the following. 
\begin{cor} \label{cor:lk}  
Let $J=J_1\sqcup J_2$ and $L=L_1\sqcup L_2$ be two-component links in $S^3$ with 
$\wh{\pi}_1(S^3-J)\cong \wh{\pi}_1(S^3-L)$. 
Then their linking numbers satisfy $|{\rm lk}(J_1,J_2)|=|{\rm lk}(L_1,L_2)|$. 
\end{cor}

\begin{proof} By the Torres formula, we have $|{\rm lk}(J_1,J_2)|=|\Delta_J(1,1)|$ and the same for $L$. 
Since the ${\rm GL}_2\wh{\Z}$-action on $(t_1,t_2)$ does not affect the value at $(1,1)$, we obtain the result. 
\end{proof}

\begin{remark} The value $\Delta_L(1,1)$ may be interpreted as a Casson--Lin type invariant \cite{HarperSaveliev2010PJM}, that is, the number of certain ${\rm SU}(2)$-representations, or that of sections of a certain principal bundle. Even if $r>2$ with $\Delta_J(1,\ldots,1)=0$, the Casson--Lin type invariant is defined using the values at roots of unity \cite{BenardConway2020Fourier}. 
From this viewpoint, our \Cref{cor:lk} fits into the program of extracting gauge-theoretic representation counts from the profinite completion, in the spirit of arithmetic gauge theory \cite{MKim2020}; we expect the multivariable Casson--Lin invariants of \cite{BenardConway2020Fourier} to admit a similar profinite control, which would be worth pursuing.
\end{remark} 

\begin{cor} \label{cor:Milnor} 
Let $J=\sqcup_i J_i$ and $L=\sqcup_i L_i$ be 3-component links in $S^3$ with $\wh{\pi}_1(S^3-J)\cong \wh{\pi}_1(S^3-L)$. 
If ${\rm lk}(J_i,J_j)=0$ for every $i\neq j$, then ${\rm lk}(L_i,L_j)=0$ for every $i\neq j$, and 
Milnor's triple linking numbers satisfy $|\mu_J(1,2,3)|=|\mu_L(1,2,3)|$. 
\end{cor}

\begin{proof} 
By Milnor's presentation $\pi/\gamma_3\pi\cong \langle x_1,x_2,x_3\mid [x_i,\lambda_i],\ \gamma_3F\rangle$ of the link group \cite{Milnor1957Isotopy}, where $\gamma_k$ denotes the $k$-th term of the lower central series, $F$ the free group on $x_1,x_2,x_3$, and $\lambda_i$ a word representing the longitude of $J_i$, we have $\gamma_2\pi/\gamma_3\pi\cong \bigwedge^2\Z^3/\langle x_i\wedge \sum_{j\neq i}{\rm lk}(J_i,J_j)x_j\rangle$, whose rank is $3$ if and only if all the linking numbers vanish. 
Since the profinite completion of a finitely generated group is compatible with the lower central series, this rank is determined by $\wh{\pi}$, and hence $L$ satisfies the same condition. 
Now, by Murasugi's formula \cite{Murasugi1966TransAMS}*{Theorem 4.3} (see also Traldi \cite{Traldi1984}*{Theorem 5.3}), 
the condition on $J$ implies that $\mu_J(1,2,3) \in \Z$ is defined and 
\[\Delta_J(1+T_1,1+T_2,1+T_3)\doteq
\mu_J(1,2,3)^2T_1T_2T_3+(\text{higher terms})\] holds, and the same for $L$. 
For each prime number $p$, consider the image of $\Delta_J$ via 
\[\wh{\Lambda}=\wh{\Z}[\![t_1^{\wh{\Z}},t_2^{\wh{\Z}},t_3^{\wh{\Z}}]\!]
\surj \Z_p[\![t_1^{\Z_p},t_2^{\Z_p},t_3^{\Z_p}]\!]
\congto \Z_p[\![T_1,T_2,T_3]\!];\ t_i\mapsto 1+T_i.\] 
Then the property clearly persists under the ambiguity in \Cref{thm:completeA}.
\end{proof}

\begin{remark} This \Cref{cor:Milnor} partially answers Traldi's question asked via an email to the third author in February 2017. 
A further possible approach to this invariant is to view it as arising from the lower central series, which is compatible with the profinite completion; indeed, $|\mu_J(1,2,3)|$ is read off from the torsion subgroup $(\Z/\mu_J(1,2,3))^2$ of $\gamma_3\pi/\gamma_4\pi$. 
\end{remark}

\begin{defn}
The \emph{Mahler measure} of $0\neq f\in \Lambda=\Z[t_1^\Z,\ldots,t_r^\Z]$ is defined by 
\[\textsc{m}(f):=\,{\rm exp}\,\int_{\mathbb{T}^r}\log|f(z)|\frac{dz_1}{z_1}\cdots \frac{dz_r}{z_r}, 
\] 
where $\mathbb{T}^r=\{z=(z_i)_i\in \C^r \mid |z_i|=1\}$ is the $r$-torus.
\end{defn}

\begin{example}
The 3-chain link $L=6^3_1$ with 
$\Delta_L(x,y,z)=xy+xz+yz-x-y-z=xy(1-1/x-1/y)+(x+y-1)z$
has 
\[\log \textsc{m}(\Delta_L) =\log \textsc{m}(x+y-1)=\frac{3\sqrt{3}}{4\pi}L(2,\chi_{-3})\approx 1.381\]
 (cf.\,\cite{Smyth1981}*{Example 5}).
\end{example}

We use the following results of L\^e. Note that $H\cong \Z^r$ via the basis $(t_i)_i$. 

\begin{lemma}[L\^e \cite{Le2014CMH}] \label{lem:Le}
Let $\mathcal{B}$ be a finitely generated torsion $\Lambda$-module with the order element $\Delta\neq 0$. 
For a finite-index subgroup $\Gamma$ of $H$, let $I(\Gamma)\subset \Lambda$ denote the ideal generated by $h-1$ $(h\in \Gamma)$, so that $\mathcal{B}/I(\Gamma)\mathcal{B}\cong \mathcal{B}\otimes_\Lambda \Z[H/\Gamma]$, and let $\ell(\Gamma)$ denote the minimal Euclidean norm $|\bm{x}|$ of nonzero elements $\bm{x}\in \Gamma$. For $\bm{k}\in \Z^r$, put $\bm{k}^\perp=\{\bm{x}\in \Z^r\mid \bm{x}\cdot\bm{k}=0\}$, where $\cdot$ denotes the standard inner product, and call $\bm{k}$ primitive if its entries are coprime. For a finitely generated abelian group $A$, let ${\rm tor}_\Z(A)$ denote its torsion subgroup. 

{\rm (1)} {\rm (\cite{Le2014CMH}*{Theorem 3})} 
$\ds \limsup_{\ell(\Gamma)\to\infty}\frac{\log|{\rm tor}_\Z(\mathcal{B}/I(\Gamma)\mathcal{B})|}{[H:\Gamma]}=\log \textsc{m}(\Delta)$ holds.
In particular, for any sequence $(\Gamma_s)_s$ with $\ell(\Gamma_s)\to\infty$, 
we have $\ds \limsup_s \frac{\log|{\rm tor}_\Z(\mathcal{B}/I(\Gamma_s)\mathcal{B})|}{[H:\Gamma_s]}\leq \log \textsc{m}(\Delta)$. 

{\rm (2)} {\rm (\cite{Le2014CMH}*{Lemma 5.6})} Let $(\bm{k}_s)_{s\in \Z_{>0}}$ be a sequence of primitive vectors in $\Z^r$ with $\ell(\bm{k}_s^\perp)\to\infty$. 
Then there exist $j_s\in \Z_{>0}$, which may be taken to be multiples of any prescribed integers,
such that the subgroups $\Gamma_s=\{\bm{x}\in \Z^r\mid \bm{x}\cdot\bm{k}_s\in d_s\Z\}$ with $d_s=j_s|\bm{k}_s|^2$ satisfy 
$\ds \lim_{s\to\infty}\frac{\log|{\rm tor}_\Z(\mathcal{B}/I(\Gamma_s)\mathcal{B})|}{d_s}=\log\textsc{m}(\Delta)$. 
\end{lemma}

\begin{cor} \label{cor:Mahler}
Given $\alpha:\pi\surj H$ and $\rho:\pi\to {\rm GL}_k\Z$, 
the Mahler measure of the twisted multivariable Alexander polynomial $\Delta_\rho^\alpha$ is a profinite invariant of $(\rho,\alpha)$. Namely, under the setting of {\rm \Cref{thm:completeA}}, 
$\textsc{m}(\Delta_\rho^\alpha)=\textsc{m}(\Delta_\varrho^\beta)$ holds. 
\end{cor}

\begin{proof}
Put $f=\Delta_\rho^\alpha\in \Lambda$, $g=\Delta_\varrho^\beta\in \Z[K]$, $\mathcal{B}=\Lambda/(f)$, and $\mathcal{B}'=\Z[K]/(g)$. 
Since $f$ and $g$ are the order elements of torsion modules, we have $f\neq 0$ and $g\neq 0$, so that $\mathcal{B}$ and $\mathcal{B}'$ are torsion modules with the order elements $f$ and $g$. 
For a finite-index subgroup $\Gamma$ of $H$, the isomorphism $\Phi_\ast:\wh{H}\congto \wh{K}$ maps the closure $\overline{\Gamma}\subset \wh{H}$ onto the closure $\overline{\Gamma'}\subset \wh{K}$ of a finite-index subgroup $\Gamma'$ of $K$ with $[K:\Gamma']=[H:\Gamma]$, and \Cref{thm:completeA} yields isomorphisms 
$\wh{\Z}\otimes_\Z \mathcal{B}/I(\Gamma)\mathcal{B}\cong \wh{\Lambda}/((f)+I(\Gamma)\wh{\Lambda})\cong \wh{\Z}\otimes_\Z \mathcal{B}'/I(\Gamma')\mathcal{B}'$, hence $|{\rm tor}_\Z(\mathcal{B}/I(\Gamma)\mathcal{B})|=|{\rm tor}_\Z(\mathcal{B}'/I(\Gamma')\mathcal{B}')|$. 
For $\bm{y}\in K$, put $\bm{u}_{\bm{y}}=\Phi_\ast^{-1}(\bm{y})\in \wh{H}=\wh{\Z}^r$, and extend the inner product to $\wh{\Z}^r\times \Z^r\to \wh{\Z}$. 
Take the sequence $(\bm{k}_s)_{s\in \Z_{>0}}$ of vectors $\bm{k}_s=(1,m_s,\ldots,m_s^{r-1})\in \Z^r$ with $m_s\geq s+1$ large enough 
so that $\bm{u}_{\bm{y}}\cdot \bm{k}_s\neq 0$ for every $\bm{y}\in K$ with $0<|\bm{y}|\leq s$; 
this is possible, since $\bm{u}_{\bm{y}}\neq 0$ has a nonzero $p$-component for some $p$ and a nonzero polynomial over $\Z_p$ has only finitely many roots. 
Then $\ell(\bm{k}_s^\perp)>s$ holds. 
Since $\bigcap_{n} n\wh{\Z}=0$, by \Cref{lem:Le} (2) we may take $j_s$ so that $\bm{u}_{\bm{y}}\cdot \bm{k}_s\notin d_s\wh{\Z}$ for these $\bm{y}$. 
Let $\Gamma'_s\subset K$ correspond to $\Gamma_s$ as above. Since $\overline{\Gamma_s}=\{\bm{u}\in \wh{H}\mid \bm{u}\cdot \bm{k}_s\in d_s\wh{\Z}\}$, we obtain $\bm{y}\notin \Gamma'_s$ for $0<|\bm{y}|\leq s$, that is, $\ell(\Gamma'_s)>s$. 
Hence \Cref{lem:Le} (1) for $\mathcal{B}'$ and (2) for $\mathcal{B}$ yield 
$\log \textsc{m}(g)\geq \limsup_s \frac{\log|{\rm tor}_\Z(\mathcal{B}'/I(\Gamma'_s)\mathcal{B}')|}{d_s}=\log \textsc{m}(f)$. 
Applying the same argument to $\Phi_\ast^{-1}$, we obtain $\textsc{m}(f)=\textsc{m}(g)$. 
\end{proof}

\begin{remark}[Spectral radius and entropy] 
For the mapping torus of a pseudo-Anosov automorphism, 
the multivariable Alexander polynomial coincides with a multivariable Lefschetz zeta function, 
and hence its Mahler measure carries dynamical information about the suspension flow \cite{YiLiu2020JAMS}*{Theorem 2.2}. 
Liu \cite{YiLiu2020JAMS}*{Theorem 1.2} confirmed McMullen's conjecture \cite[Conjecture 1.1]{YiLiu2020JAMS}: 
an automorphism of a compact orientable surface admits a virtual homological spectral radius greater than 
$1$ if and only if its mapping class entropy is positive, 
the key step being that suitable finite covers have multivariable Alexander polynomials of Mahler measure greater than $1$ \cite{YiLiu2020JAMS}*{Theorem 4.3}. 
Combined with \Cref{cor:Mahler}, this shows that \emph{positivity of the mapping-class entropy is a profinite property}. 
Moreover, the proof of \Cref{thm:fully-punctured} shows that $f_\varphi$ admits an inner singularity of odd order if and only if $f_\psi$ does, so the absence of odd-order inner singularities is a profinite property. 
By McMullen \cite{McMullen2013CMH}*{Theorem 1.1}, this absence is equivalent to \emph{the entropy being detected homologically}, namely, to the entropy of the map being the supremum of the logarithms of the spectral radii of the actions on the first homology of the finite covers to which the map lifts. 
Hence the latter property is a profinite property as well.
\end{remark} 

\begin{remark}[Hyperbolic volume] 
For a hyperbolic 3-manifold, the relationship among the Mahler measure of the twisted Alexander polynomial of the holonomy representation, the torsion growth in a sequence of coverings, and the hyperbolic volume has been conjectured and (partially) established by Silver--Williams \cite{SilverWilliams2002Topology}, Le \cite{Le2014CMH,Le2018AIF}, Bergeron--Venkatesh \cite{BergeronVenkatesh2013}, and Bergeron--\c{S}eng\"un--Venkatesh \cite{BergeronSengunVenkatesh2016}. 
Our result (\Cref{cor:Mahler}) would yield a lower bound for the volume of a hyperbolic link, 
provided the holonomy representation can be recovered from the profinite completion, which is currently open. 
\end{remark}

Now let $p$ be a prime number and 
let $\log_p$ denote the $p$-adic logarithm regularized by $\log_p p=0$. 
Then 
\begin{defn} Besser--Deninger's \emph{$p$-adic log Mahler measure} of $f\in \Z[t_1^\Z,\ldots,t_r^\Z]$ with $f(z)\neq 0$ on $\mathbb{T}_p^{\,r}=\{z\in \C_p^{\,r}\mid |z_i|=1\}$ is defined by 
\[m_p(f)=\int_{\mathbb{T}_p^{\,r}}\log_p f(z)\frac{dz_1}{z_1}\cdots \frac{dz_r}{z_r}
:=\lim_{\substack{n\to \infty\\ p\,\nmid\,n}} \frac{1}{n^r} \sum_{\zeta\in \bm{\mu}_n^{\,r}}\log_p f(\zeta),\]
where $\bm{\mu}_n^{\,r}=\{\zeta=(\zeta_i)_i \in \C_p^{\,r}\mid \zeta_i^n=1\}$ is the set of $r$-tuples of $n$-th roots of unity. 
\end{defn} 
Its relation to the syntomic regulator, 
and hence to $p$-adic heights and $p$-adic $L$-functions, 
is studied in \cite{BesserDeninger1999}. 
The following profinite invariance is of arithmetic interest.

\begin{cor} \label{cor:pMahler}
Under the same setting as in \Cref{cor:Mahler}, 
$m_p(\Delta_\rho^\alpha)$ is defined if and only if $m_p(\Delta_\varrho^\beta)$ is defined.
If they are defined, then 
$m_p(\Delta_\rho^\alpha)=m_p(\Delta_\varrho^\beta)$ holds. 
\end{cor}

\begin{proof}
Let $f=\Delta_\rho^\alpha(t_1,\ldots,t_r)$ and $g=\Delta_\varrho^\beta(s_1,\ldots,s_r)$ be as in \Cref{thm:completeA}. 
Since $f$ is a nonzerodivisor of $\wh{\Lambda}$, we have $g=uf$ for some $u\in \wh{\Lambda}^\times$. 
For $z\in \mathbb{T}_p^{\,r}$, the powers $z_i^{v}$ $(v\in \wh{\Z})$ are defined through the decomposition of $z_i$ into the Teichm\"{u}ller part and the $1$-unit part, and the evaluation $t_i\mapsto z_i$ extends to a continuous ring homomorphism $\wh{\Lambda}\to \C_p$. 
Hence $g(s(z))=u(z)f(z)$ with $u(z)\in \C_p^\times$, where $s(z)=(s_1(z),\ldots,s_r(z))$, and $z\mapsto s(z)$ is a bijection on $\mathbb{T}_p^{\,r}$. 
Thus $f$ has no zero on $\mathbb{T}_p^{\,r}$ if and only if $g$ has no zero on $\mathbb{T}_p^{\,r}$. 
Next, for $\zeta\in \bm{\mu}_n^{\,r}$ with $p\nmid n$, the evaluation $t_i\mapsto \zeta_i$ yields $g(s(\zeta))=u(\zeta)f(\zeta)$ with $u(\zeta)\in (\wh{\Z}\otimes \Z[\xi_n])^\times$, where $\xi_n\in \ol{\Q}$ is a primitive $n$-th root of unity, and let ${\rm Nr}:\wh{\Z}\otimes \Z[\xi_n]\to \wh{\Z}$ denote the map induced by the norm map of $\Q(\xi_n)/\Q$. 
Taking the norms to $\wh{\Z}$, we obtain ${\rm Nr}(g(s(\zeta)))={\rm Nr}(u(\zeta)){\rm Nr}(f(\zeta))$ with ${\rm Nr}(f(\zeta)),{\rm Nr}(g(s(\zeta)))\in \Z\setminus\{0\}$ and ${\rm Nr}(u(\zeta))\in \wh{\Z}^\times$, hence ${\rm Nr}(g(s(\zeta)))=\pm {\rm Nr}(f(\zeta))$. 
Since $\log_p$ is a homomorphism with $\log_p(-1)=0$ and $\sum_{a\in (\Z/n)^\times}\log_p f(\zeta^a)=\log_p {\rm Nr}(f(\zeta))$ for $\zeta^a=(\zeta_i^a)_i$, we obtain $\sum_{\zeta\in \bm{\mu}_n^{\,r}}\log_p f(\zeta)=\sum_{\eta\in \bm{\mu}_n^{\,r}}\log_p g(\eta)$ for every $n$ with $p\nmid n$, and the assertion follows. 
\end{proof}

\subsection{Under the regularity assumption} \label{ss.prf.reg} 

Here, we prove the following theorem to deduce \Cref{thm:twisted}, and further discuss the ambiguity. 
\begin{theorem} \label{thm.regular} 
Let the setting be as in {\rm \Cref{thm:completeA}}. 
Suppose that the induced isomorphism $\Phi_\ast:\wh{H}\congto \wh{K}$ is the completion of an isomorphism $\Phi_\star:H\congto K$ of discrete groups. 
Then, the twisted Alexander polynomials $\Delta^{\alpha}_\rho \in \Z[H]$ and $\Delta^{\beta}_{\varrho}\in \Z[K]$ coincide via $\Phi_\star:H\congto K$, up to multiplication by units of $\Lambda=\Z[H]$.  
\end{theorem}

\begin{proof}[{Proof of {\rm \Cref{thm.regular}}}] 
Since $\Phi_\ast$ is the completion of $\Phi_\star:H\congto K$, 
the matrix $V=(v_{ij})$ in \Cref{thm:completeA} lies in ${\rm GL}_r\Z$, 
so that $s_i=t_1^{v_{i1}}\cdots t_r^{v_{ir}}\in H$ for each $i$ and $\Delta_\varrho^\beta(s_1,\ldots,s_r)\in \Lambda$. 
By \Cref{thm:completeA}, the equality 
\[(\Delta_\rho^\alpha(t_1,\ldots,t_r))\wh{\Lambda}=(\Delta_\varrho^\beta(s_1,\ldots,s_r))\wh{\Lambda}\] 
of ideals of $\wh{\Lambda}$ holds. 
Since the ring extension $\Lambda\to \wh{\Lambda}$ is faithfully flat (\Cref{lem:hatLambdaisflat}), 
every ideal $\mf{a}$ of $\Lambda$ satisfies $\mf{a}\wh{\Lambda}\cap \Lambda=\mf{a}$. 
Applying this to the principal ideals of $\Lambda$ generated by $\Delta_\rho^\alpha(t_1,\ldots,t_r)$ and by $\Delta_\varrho^\beta(s_1,\ldots,s_r)$, 
we obtain the equality 
\[(\Delta_\rho^\alpha(t_1,\ldots,t_r))=(\Delta_\varrho^\beta(s_1,\ldots,s_r))\] 
of ideals of $\Lambda$, namely, $\Delta_\rho^\alpha(t_1,\ldots,t_r)\doteq \Delta_\varrho^\beta(s_1,\ldots,s_r)$ in $\Lambda$. 
This completes the proof. 
\end{proof}

\begin{remark} \label{rem:Fitting}
The proof of \Cref{thm.regular} applies to any ideal of $\wh{\Lambda}$ generated by elements of $\Lambda$, so it yields a refinement. 
For $i\in \Z_{\geq 0}$, let ${\rm Fitt}_i(\mathcal{A}_\rho^\alpha)\subset \Lambda$ denote the $i$-th Fitting ideal of the twisted Alexander module $\mathcal{A}_\rho^\alpha$, and similarly for $\mathcal{A}_\varrho^\beta$. 
Since Fitting ideals are compatible with base change (\cite{EisenbudGTM}*{Corollary 20.5}) and $\wh{\mathcal{A}}_\rho^\alpha$ shares a presentation matrix with $\mathcal{A}_\rho^\alpha$ (\Cref{lem:share}), we have ${\rm Fitt}_i(\wh{\mathcal{A}}_\rho^\alpha)={\rm Fitt}_i(\mathcal{A}_\rho^\alpha)\wh{\Lambda}$. 
Hence, under the assumption of \Cref{thm.regular}, the isomorphism of $\wh{\Lambda}$-modules in \Cref{thm:completeA} and the faithful flatness of $\Lambda\to \wh{\Lambda}$ (\Cref{lem:hatLambdaisflat}) imply that ${\rm Fitt}_i(\mathcal{A}_\rho^\alpha)$ and ${\rm Fitt}_i(\mathcal{A}_\varrho^\beta)$ coincide via $\Phi_\star$ for every $i\in \Z_{\geq 0}$. 
If in addition $\pi$ and $\varpi$ are FP$_n$ and $n$-good, then the same holds for the $\ast$-th twisted Alexander modules with $\ast\leq n$, in the setting of \Cref{rem.good} (2). 
\end{remark}

\begin{proof}[{Proof of {\rm \Cref{thm:twisted}}}] 
By \Cref{thm:regularity}, $\Phi$ is regular. Hence \Cref{thm.regular} ensures the assertion. 
\end{proof}

It is a natural question to ask whether we can reduce the ambiguity arising from ``via $\Phi_\star:H_M\congto H_N$'' in \Cref{thm:twisted}. 
As the Mostow--Prasad rigidity ensures, given an isomorphism $\pi_1(M)\cong \pi_1(N)$, the peripheral systems of cusps correspond (cf.\,\cite{AschenbrennerFriedlWilton2015}*{Theorem 1.7.2}). 
Xu 
proved that this also holds true for profinite isomorphism. 
Hence, we may further prove the following. 

\begin{cor} \label{cor:twist_alex.cusped} 
Let the setting be as in {\rm \Cref{thm:twisted}} and suppose that 
$M$ and $N$ are the exteriors of hyperbolic links $J$ and $L$ in $S^3$. 
If the linking numbers of all two-component sublinks of $J$ are zero, then 
$L$ satisfies the same condition, and 
$\Delta^{\alpha}_\rho$ and $\Delta^{\beta}_{\varrho}$ coincide 
only up to orientation-reversing and order-changing of the meridians, 
and multiplication by units.  
\end{cor}

\begin{proof}  
The argument for \Cref{prop.hyplinks} combined with Xu's result on the peripheral regularity \cite{Xu2025regularity-arXiv}*{Theorem 1.4} ensures that $\Phi_\star:H_M\congto H_N$ maps 
each $t_i$ to some $s_j$ or $s_j^{-1}$. Thus the result immediately follows from \Cref{thm:twisted}.  
\end{proof}

Typical examples that satisfy this assumption are in \Cref{eg.Borromean}.  

\subsection{Under the \texorpdfstring{$\wh{\Z}^\times$}{Zhat-times}-regularity assumption}  
Here, we prove the following slightly strengthened version of \Cref{thm.regular}. 

\begin{theorem} \label{thm.hatZtimes-regular} 
Let the setting be as in {\rm \Cref{thm:completeA}}. 
Suppose that the induced isomorphism $\Phi_\ast:\wh{H}\congto \wh{K}$ and the fixed bases define an isomorphism $\wh{\Z}^r\congto \wh{\Z}^r$ presented by some $V\in {\rm GL}_r\Z$ and $v\in \wh{\Z}^\times$ as $\bm{x}\mapsto vV\bm{x}$. 
Then, 

{\rm (i)} $\Delta^{\alpha}_\rho \in \Z[H]$ and $\Delta^{\beta}_{\varrho}\in \Z[K]$ are both generalized cyclotomic polynomials and they coincide via $V$ up to multiplication by units of $\Z[H]$, or 

{\rm (ii)} $v\in \{\pm 1\}$, $\Phi_\ast$ restricts to an isomorphism $\Phi_\star:H\congto K$, 
and {\rm \Cref{thm.regular}} applies. 
\end{theorem}

\subsubsection{Reduced polynomials} 
Let $\bm{m}=(m_1,\ldots,m_r)\in \Z^r$ and let $r_{\bm{m}}:\Z[H]\to \Z[t^\Z]$ denote the ring homomorphism defined by $t_i\mapsto t^{m_i}$. 
We call $r_{\bm{m}}(\Delta_\rho^\alpha)=\Delta_\rho^\alpha(t^{m_1},\ldots,t^{m_r})$ \emph{the $\bm{m}$-reduced Alexander polynomial}. 
The homomorphism $r_{\bm{m}}$ extends to a continuous ring homomorphism $\wh{\Lambda}\to \wh{\Z}[\![t^{\wh{\Z}}]\!]$, which we denote by the same symbol. 
The following elementary lemma recovers $\Delta_\rho^\alpha(t_1,\ldots,t_r)$ from the family $(r_{\bm{m}}(\Delta_{\rho}^{\alpha}))_{\bm{m}}$.

\begin{lemma}[Cimasoni, {\cite[Lemma 2.2]{Cimasoni2004BSMM}}] \label{lem:Cimasoni}
Let $f,g\in \Z[t_1^\Z,\ldots,t_r^\Z]$. If
\[
f(t^{m_1},\dots,t^{m_r})\doteq g(t^{m_1},\dots,t^{m_r})
\]
holds for all but finitely many $(m_i)_i\in \Z^r$, then $f\doteq g$.
\end{lemma}

\begin{remark} At the beginning of the proof of \cite[Lemma 2.2]{Cimasoni2004BSMM}, the author assumes $a_{0,\ldots,0}>0$ and $b_{0,\ldots,0}>0$ to deduce $\varepsilon=1$ and $\nu=0$, according to the original terminology. 
In fact, we cannot assume this without loss of generality, since we have $f(t_1,t_2)=t_1+t_2$ for instance. 
Nevertheless, by keeping $\varepsilon\in \{\pm 1\}$ and $\nu\in \Z$ as unknown constants, the original argument essentially persists, provided that $N$ is taken greater than twice the degrees so that $\nu$ is determined by the minimal exponents. 
\end{remark}

\subsubsection{Generalized cyclotomic polynomials}

\begin{defn} We say that $f\in \Z[t_1^\Z,\ldots,t_r^\Z]$ is a \emph{generalized cyclotomic polynomial} if every reduction $r_{\bm{m}}(f)=f(t^{m_1},\ldots,t^{m_r})\in \Z[t^\Z]$ with $\bm{m}\in \Z^r$ is the product of cyclotomic polynomials and a constant.
\end{defn}

\begin{lemma}\label{lem:gen-cyclotomic}
Let $0\neq \Delta\in\Z[t_1^{\Z},\ldots,t_r^{\Z}]$ and let $\Delta_0$
denote its primitive part. Then the following are equivalent.
\begin{enumerate}
\item[{\rm (i)}] $\Delta$ is a generalized cyclotomic polynomial.
\item[{\rm (ii)}] Every nonzero reduction $r_{\bm{m}}(\Delta)$ has no root other
than roots of unity.
\item[{\rm (iii)}] $\Delta_0\doteq \prod_i \Phi_{n_i}(t_1^{w_{i1}}\cdots t_r^{w_{ir}})$
holds for some $n_i\in\Z_{>0}$ and $(w_{i1},\ldots,w_{ir})\in\Z^r$.
\item[{\rm (iv)}] The Mahler measure satisfies $\textsc{m}(\Delta_0)=1$.
\end{enumerate}
\end{lemma}

\begin{proof}
The implications (i) $\Leftrightarrow$ (ii) and (iii) $\Rightarrow$ (i) are
elementary. 
(i) $\Rightarrow$ (iv) follows from Lawton's
theorem \cite[Theorem 2]{Lawton1983JNT} on the convergence of the Mahler measures of
one-variable specializations. (iv) $\Rightarrow$ (iii) is the
multivariable Kronecker theorem of Boyd \cite[Theorem 1]{Boyd1981JNT}.
\end{proof}

\begin{example} The following are generalized cyclotomic polynomials. 

(i) The Alexander polynomial $\Delta_{4^2_1}(t_1,t_2)=t_1t_2+1=\Phi_2(t_1t_2)$ of Solomon's link  (\Cref{eg:Alex}), 
which is non-hyperbolic. 

(ii) The Alexander polynomial 
$\Delta_B(t_1,t_2,t_3)=(t_1-1)(t_2-1)(t_3-1)=\Phi_1(t_1)\Phi_1(t_2)\Phi_1(t_3)$ 
of the Borromean rings (\Cref{eg.Borromean}). 
\end{example}

\subsubsection{$\wh{\Z}^\times$-regularity} 

We utilize the following key assertion. 
\begin{lemma}[{\cite[Lemma 3.6]{Ueki5}, Hanany \cite[Corollary 1.5]{Hanany2026regularity-arXiv}}] \label{lem:fdoteqg}
Let $f,g\in \Z[t^\Z]$. 
If $(f(t)) = (g(t^v))$ as ideals in $\wh{\Z}[\![t^{\wh{\Z}}]\!]$ for some $v\in \wh{\Z}^\times$, then at least one of the following holds:  

{\rm (i)} $f(t)$ and $g(t)$ are products of cyclotomic polynomials and constants with $f(t)\doteq g(t)$. 

{\rm (ii)} $v\in \{\pm1\}$ and $f(t)\doteq g(t^v)$. 
\end{lemma}

To specify the contribution: Fried \cite{Fried1988} initially proved that if $f\in \Z[t]$ is reciprocal and does not vanish on roots of unity, then it is recovered from the cyclic resultants by using Artin--Mazur dynamical zeta functions. His result was extended to non-reciprocal cases by Hillar \cite{Hillar2005}. 
Afterward, the third author proved that if $f,g\in \Z[t^\Z]$ are reciprocal, then $(f(t))=(g(t^v))$ in $\wh{\Z}[\![t^{\wh{\Z}}]\!]$ with $v\in \wh{\Z}^\times$ implies $f(t)\doteq g(t)$, by an elementary argument using the NAK lemma. 
Hanany further proved that if $f,g\in \Z[t^\Z]$ with $(f(t))=(g(t^v))$ and $f$ or $g$ has a root that is not a root of unity, then $v\in \{\pm1\}$ and $f(t)\doteq g(t^v)$, by purely algebraic and number-theoretic arguments. 

\begin{proof}[Proof of {\rm \Cref{thm.hatZtimes-regular}}] 
For each $\bm{m}\in\Z^r$, the continuous ring homomorphism
$r_{\bm{m}}:\wh{\Lambda}\to \wh{\Z}[\![t^{\wh{\Z}}]\!]$
sends the equality of ideals in \Cref{thm:completeA} to
$\bigl(r_{\bm{m}}(\Delta_\rho)(t)\bigr)=\bigl(r_{V\bm{m}}(\Delta_\varrho)(t^{v})\bigr)$.

If there exists some reduction $r_{\bm{m}}$ such that $r_{\bm{m}}(\Delta_\rho)$ or $r_{V\bm{m}}(\Delta_\varrho)$ 
is nonzero and has a root which is not a root of unity, 
then Hanany's result \Cref{lem:fdoteqg} (ii) yields $v\in \{\pm1\}$. 
Hence $vV\in{\rm GL}_r\Z$ and the assertion (ii) holds.

If such a reduction does not exist, then, since $\bm{m}\mapsto V\bm{m}$ is a bijection on $\Z^r$, both $\Delta_\rho$ and $\Delta_\varrho$ are generalized cyclotomic by \Cref{lem:gen-cyclotomic}; here $\Delta_\rho=0$ if and only if $\Delta_\varrho=0$, and in this case the assertion (i) trivially holds. In addition, putting $u_i=t_1^{v_{i1}}\cdots t_r^{v_{ir}}\in H$, 
we have $r_{\bm{m}}(\Delta_\rho)\doteq r_{V\bm{m}}(\Delta_\varrho)=r_{\bm{m}}(\Delta_\varrho(u_1,\ldots,u_r))$ for every $\bm{m}$ by \Cref{lem:fdoteqg} (i), 
and hence \Cref{lem:Cimasoni} applied to $\Delta_\rho$ and $\Delta_\varrho(u_1,\ldots,u_r)$ ensures that they coincide via $V$. Thus the assertion (i) holds.
\end{proof}

\begin{remark} 
By using JSJ-decompositions and considering splices, one can make another attempt at the profinite rigidity of the classical multivariable Alexander polynomial of a general link $L$ in $S^3$, which we hope to pursue elsewhere \cite{Ueki-F2byZ-arXiv}.
\end{remark}

\section{Polynomial invariants of veering triangulations} \label{sec:veering} 
\subsection{Fibered faces, pseudo-Anosov flows, and triangulations} \label{ss.tri} 
\subsubsection{Fibrations and the Thurston norm ball} \label{sss.TNB}
Let $M$ be a 3-manifold. 
A 3-manifold is said to be fibered if it admits a fibration over $S^1$. 
An integral class $\varphi\in H^1(M;\Z)={\rm Hom}(\pi_1(M),\Z)$ is said to be \emph{fibered} if $M$ admits a fibration over $S^1$ that induces $\varphi$. 

W.\,P.\,Thurston introduced in \cite{Thurston1986norm} a certain norm on $H^1(M;\R)$ and proved that 
the unit norm ball is a finite-sided polyhedron, which organizes fibrations of $M$ over $S^1$ in the following sense. 
If $\varphi\in H^1(M;\Z)$ is a fibered class, then it lies in the open cone $\R_{+} \mathcal{F}$ on some top-dimensional face $\mathcal{F}$ of the norm ball.  
Conversely, every integral point on this cone is again a fibered class. 
Hence, such $\mathcal{F}$ is called a \emph{fibered face}, and the polyhedron is called \emph{the Thurston norm ball}. 
(A thorough survey on the Thurston norm is \cite{Kitayama2022survey}.) 

Profinite rigidity of the Thurston norm (and fibered classes, fibered faces, fibered cones, etc.) was initially proved for 3-manifolds with empty or toroidal boundary under the regularity assumption by Boileau--Friedl \cite{BoileauFriedl2020AMS},  
and for orientable finite-volume hyperbolic 3-manifolds by Liu \cite{YiLiu2023Invent}. 

\subsubsection{Fully-punctured pseudo-Anosov maps} \label{ss.pAmap} 
Here, let $S$ be a surface that may have punctures but has no boundary, namely, 
$S$ is the exterior of a finite subset of a closed surface. 

An automorphism $f$ on a hyperbolic surface $S$ is said to be \emph{pseudo-Anosov} if 
it is neither periodic nor reducible in the Nielsen--Thurston classification; 
see \cite{FLP1979Asterisque66-67,FarbMargalit2012book} for the general theory, 
including the case with boundaries. 
On the torus $S_{1,0}$, we regard Anosov maps as pseudo-Anosov; 
the other non-hyperbolic surfaces $S^2=S_{0,0}$, $S_{0,1}$, and $S_{0,2}$ carry no such dynamics. 
A pseudo-Anosov map admits stable/unstable measured foliations that intersect transversely. 
The punctures are $k\geq 1$-pronged,
that is, each of them intersects with $k$ stable/unstable leaves, and is not a spiral locus. 
Away from the punctures, there may exist finitely many $k\geq 3$-prong points, 
which we call \emph{inner singularities}. 
We say that $f$ is \emph{fully-punctured} if it admits no inner singularities; 
2-pronged punctures are allowed (see also \Cref{rem:prf.thm1.2} (1) for a narrower variant). 
Any pseudo-Anosov map $f$ on $S$ restricts to a fully-punctured pseudo-Anosov map on 
the exterior $S-{\rm Sing}(f)$ of the inner singularities. 

\subsubsection{Fibered faces and suspension pseudo-Anosov flows without perfect fits} \label{ss.ff}
As Thurston proved \cite{WPThurston1986HS2}*{Theorem 0.1}, 
a fibered 3-manifold is hyperbolic if and only if its monodromy map is pseudo-Anosov and the fiber has a negative Euler characteristic. 
In such a case, the suspension flow belongs to the class of so-called pseudo-Anosov flows; namely, 
it admits stable and unstable foliations that transversely intersect, and there are finitely many $k\geq 3$-pronged singularities. 
Furthermore, such a suspension pseudo-Anosov flow is without perfect fits (cf.\,Fenley \cite{Fenley2012GT}*{Introduction}). 

Conversely, Fried \cite{Fried1979Ast} showed the uniqueness of pseudo-Anosov flows transverse to every fibration in $\R_{+}\mathcal{F}$. 
In particular, being fully-punctured is a property of the fibered face.

\subsubsection{Veering triangulations and pseudo-Anosov flows without perfect fits}

The notions of taut and \emph{veering triangulations} were initially introduced by Lackenby \cite{Lackenby2000GT} and Agol \cite{Agol2011Contemp} in terms of taut foliations and mapping tori of pseudo-Anosov mappings, respectively. 

There is a correspondence between veering triangulations and pseudo-Anosov flows without perfect fits on a 3-manifold in the following sense. 
In unpublished work, Agol--Gu\'{e}ritaud constructed, from a given pseudo-Anosov flow without perfect fits, a veering triangulation on the complement of the singular orbits of the flow. 
A detailed construction is presented in a published paper by Landry--Minsky--Taylor \cite{LandryMinskyTaylor2023ETDS}*{Section 4}. 
Conversely, Frankel--Schleimer--Segerman \cite{FrankelSchleimerSegerman2019-arXiv-v5} and Agol--Tsang \cite{AgolTsang2024AGT} showed that a veering triangulation of a cusped hyperbolic $3$-manifold determines
a pseudo-Anosov flow without perfect fits on the result of the Dehn filling along a suitable slope.

In particular, each fibered hyperbolic 3-manifold endowed with a suspension flow of a fully-punctured monodromy map bijectively corresponds to a layered veering triangulation of a cusped hyperbolic 3-manifold.

\subsection{Veering triangulations}

In this subsection, we recollect the notion of veering triangulations and the associated edge-orientation homomorphism, 
which will be used to describe the taut polynomial. Our convention is mainly due to \cite{LandryMinskyTaylor2024JEMS}.

\subsubsection{Veering triangulations}

Let $M$ be an irreducible compact connected orientable $3$-manifold with toroidal boundary. An \emph{ideal} triangulation of $M$ is  a decomposition of $M-\partial M$ into a finite collection of ideal tetrahedra (i.e., tetrahedra minus their vertices) 
with faces identified by orientation-reversing simplicial homeomorphisms. 

A \emph{taut ideal tetrahedron} is an ideal tetrahedron along with a transverse orientation on each face such that two of its faces point into the tetrahedron
and two of its faces point out of the tetrahedron.  Each of its edges is then assigned angle $\pi$ or $0$ depending on whether the transverse orientations on the adjacent faces agree or disagree,
respectively. An edge in a taut ideal tetrahedron $\mathfrak t$ is called the \emph{top} (resp.\,\emph{bottom}) \emph{diagonal} of $\mathfrak t$ if it has angle $\pi$ and the coorientations on 
its adjacent faces 
point out of (resp.\,point into) $\mathfrak t$. The faces of $\mathfrak t$ adjacent to the top (resp.\,bottom) diagonal are called the \emph{top} (resp.\,\emph{bottom}) \emph{faces}. We refer to $\mathfrak t$ as the tetrahedron \emph{immediately below} (resp.\,\emph{immediately above}) its top (resp.\,bottom) faces.

An ideal triangulation of $M$ is \emph{taut} if each of its faces has been cooriented so that each ideal tetrahedron is taut and the angle sum around each edge is $2\pi$.  
Note that every face in a taut ideal triangulation is a top face of a tetrahedron and a bottom face of another tetrahedron. 
\begin{defn}
A \emph{veering triangulation} $\mathcal V$ of $M$ is a taut ideal triangulation of $M$ in which each edge
has a consistent veer; that is, each edge is labeled to be either right or left veering, such that each tetrahedron of $\mathcal V$ admits an orientation-preserving isomorphism to the following model veering tetrahedron (\Cref{fig:veer}), in which the veers of the 0-edges are specified. 
\end{defn}

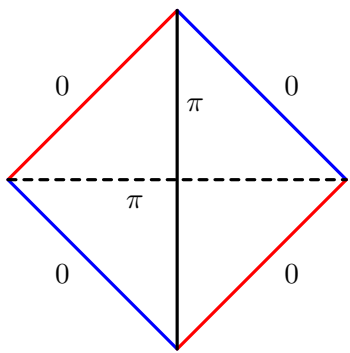
\begin{figure}[h]
\centering
\begin{tikzpicture}[scale=1.6, line cap=round, line join=round]

\coordinate (T) at (0,1.4);
\coordinate (L) at (-1.4,0);
\coordinate (R) at (1.4,0);
\coordinate (B) at (0,-1.4);

\draw[red, very thick]  (L) -- (T);
\draw[blue, very thick] (T) -- (R);
\draw[blue, very thick] (L) -- (B);
\draw[red, very thick]  (B) -- (R);

\draw[dashed, very thick] (L) -- (R);
\draw[very thick] (T) -- (B);

\node at (0.15,0.62) {$\pi$};
\node at (-0.35,-0.18) {$\pi$};

\node at (-0.95,0.78) {$0$};
\node at (0.95,0.78) {$0$};
\node at (-0.95,-0.78) {$0$};
\node at (0.95,-0.78) {$0$};
\end{tikzpicture}
\caption{A model veering tetrahedron: right-veering edges are red, left-veering edges
are blue}
\label{fig:veer}
\end{figure}

We say that a veering (or taut) triangulation is \emph{layered} 
if it can be built by stacking tetrahedra onto a triangulated surface and quotienting by a homeomorphism of the surface.

\subsubsection{The horizontal branched surface}

Every veering triangulation $\mathcal V$ of $M$ defines a canonical branched surface \(B=B(\mathcal V)\), called the \emph{horizontal branched surface} of $\mathcal V$, 
obtained by smoothing the \(2\)-skeleton in the manner prescribed by the transverse orientations. 

Let $B$ be the horizontal branched surface of a taut triangulation $\mathcal V$ of $M$. A \emph{dual train track} $\tau$ of $B$ is a trivalent train track $\tau \subset B$ such that, for each face $\mathfrak f \subset B$, $\tau_{\mathfrak f} = \tau \cap \mathfrak f$ has 
exactly one switch that lies in the interior of $\mathfrak f$. The edges of $\tau_{\mathfrak f}$ are called \emph{half-branches} of $\tau_{\mathfrak f}$. Since switches of $\tau$ are trivalent, the half-branches of $\tau_{\mathfrak f}$ admit a natural partition into two types: two are called \emph{small}, and the remaining one is called \emph{large}, as shown in \Cref{fig:dual_train_track}.

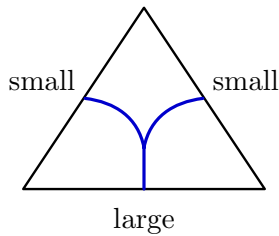
\begin{figure}[h]
    \centering
    \begin{tikzpicture}[scale=1, line cap=round, line join=round]

\coordinate (L) at (-1.6,0);
\coordinate (R) at (1.6,0);
\coordinate (T) at (0,2.4);

\coordinate (S) at (0,0);
\coordinate (M) at (0,0.55);
\coordinate (SL) at (-0.78,1.20);
\coordinate (SR) at (0.78,1.20);

\draw[black, line width=0.9pt] (L) -- (T) -- (R) -- cycle;

\draw[blue!80!black, line width=1.2pt]
  (S) -- (M)
  (M) .. controls (-0.08,0.95) and (-0.38,1.15) .. (SL)
  (M) .. controls ( 0.08,0.95) and ( 0.38,1.15) .. (SR);

\node at (-1.35,1.45) {small};
\node at ( 1.35,1.45) {small};
\node at (0,-0.42) {large};

\end{tikzpicture}
    \caption{a dual train track}
    \label{fig:dual_train_track}
\end{figure}

A dual train track $\tau$ in a horizontal branched surface $B$ is called the \emph{upper track} (resp.\,\emph{lower track}) if $\tau$ has the property that the large half-branch of $\tau_{\mathfrak f} = \tau \cap \mathfrak f$ meets the bottom (resp.\,top) diagonal of the tetrahedron immediately above (resp.\,below) $\mathfrak f$ for every face $\mathfrak f \subset B$.

\subsubsection{The edge-orientation homomorphism} 

We say that a veering triangulation $\mathcal V$ is \emph{edge-orientable} if the upper track $\tau=\tau_U$ of $\mathcal V$ is transversely orientable.
Following \cite{Parlak2023JT}, define the \emph{edge-orientation double cover}
$p:\mathcal V^{\mathrm{or}}\to \mathcal V$ as follows.
For each tetrahedron $\mathfrak t$ of $\mathcal V$, there are exactly two possible \emph{local edge-orientations}
on $\mathfrak t$; call them
$\mathfrak o$ and $-\mathfrak o$.  Let $X$ be the set of pairs $(\mathfrak t,\mathfrak o)$ where $\mathfrak t$ is a
tetrahedron of $\mathcal V$ and $\mathfrak o$ is a local edge-orientation on $\mathfrak t$.
Whenever two tetrahedra $\mathfrak t,\mathfrak t'$ are glued along a face $\mathfrak f$, the face-pairing identifies a local
choice $\mathfrak o$ on $\mathfrak t$ with a unique local choice $\mathfrak o'$ on $\mathfrak t'$ (namely, the one
whose restriction to $\mathfrak f$ matches the restriction of $\mathfrak o$ under the face identification).
Use these identifications to glue the disjoint union of the two copies of each tetrahedron indexed by
$\mathfrak o$ and $-\mathfrak o$.  The resulting space 
$p:\mathcal V^{\mathrm{or}}\to \mathcal V$ is called the edge-orientation double cover.

\begin{defn} \label{def:eohom}
The \emph{edge-orientation homomorphism} of $\mathcal V$ is the monodromy character of this cover:
\[
\omega:\pi_1(M)\to {\rm GL}_1\Z=\Z^\times,
\qquad
\omega(\gamma)=
\begin{cases}
1 & \text{if $\gamma$ lifts to a loop in }\mathcal V^{\mathrm{or}},\\
-1 & \text{otherwise.}
\end{cases}
\]
\end{defn}
In particular, the following are equivalent: 
(i) $\omega$ is trivial, (ii) $\mathcal V^{\mathrm{or}}$ is disconnected, (iii) $\mathcal V$ is edge-orientable.

\subsection{The taut and Teichm\"uller polynomials}
\label{ss.taut} 

In this subsection, we recall several results of Parlak \cite{Parlak2023JT}, 
which will play key roles in the proofs of our main results. 

\begin{theorem}[{\cite{Parlak2023JT}*{Theorem 5.7}}] \label{thm:Parlak} 
Let $\mathcal{V}$ be a veering triangulation of a $3$-manifold $M$ 
and let $\omega$ denote the edge-orientation homomorphism of $\mathcal{V}$ and $\alpha:\pi_1(M)\surj H_M = H_1(M;\Z)_{\rm free}$ the maximal free abelian quotient. 
Then, the taut polynomial and the twisted Alexander polynomial coincide, namely,
    \[\Theta_\mathcal{V} \doteq \Delta_{\omega}^{ \alpha}\]
holds in $\Z[H_M]$. 
\end{theorem}

\begin{lemma}[{\cite[Lemma 3.10]{Parlak2023JT}}] \label{lem:eo-to} 
Suppose that $\mathcal{V}$ is constructed from a pseudo-Anosov flow $\Psi$. 
Then $\mathcal{V}$ is edge-orientable if and only if 
the stable lamination $L$ of $\Psi$ is transversely orientable. 
\end{lemma} 

\begin{lemma}[{\cite[Lemma 4.6]{Parlak2023JT}}] \label{lem:eo-cover} 
Suppose that $\mathcal{V}$ is not edge-orientable.
Let $M^{\rm or}\to M$ denote the double cover corresponding to ${\rm Ker}\,\omega$ and let $\mathcal{V}^{\rm or}$ denote the lift of $\mathcal{V}$ to $M^{\rm or}$. 
Then $\mathcal{V}^{\rm or}$ is edge-orientable.
Namely, the edge-orientation double cover is edge-orientable. 
\end{lemma}

\begin{lemma}[The universality of the edge-orientation homomorphism, {\cite{Parlak2023JT}*{Lemma 4.9}}]
\label{lem:eo-universality} 
Let $q:\pi_1(M)\surj G$ be a surjective group homomorphism and let $M^q\to M$ denote the regular cover corresponding to ${\rm Ker}\,q$.
Let $\mathcal{V}^q$ denote the veering triangulation on $M^q$ which is a lift of that on $M$.
Then $\mathcal{V}^q$ is edge-orientable if and only if the edge-orientation homomorphism $\omega:\pi_1(M)\to\Z^\times$ factors through $q$.
\end{lemma}

\begin{lemma}[{\cite{LandryMinskyTaylor2024JEMS}*{Proposition 7.2}, \cite{Parlak2023JT}*{Lemma 7.1}}] 
\label{lem:tautTeich} 
Let $M$ be a fibered hyperbolic 3-manifold.  
Let $\mathcal{F}$ denote the fibered face corresponding to the suspension flow 
and $\mathcal{V}$ the layered veering triangulation of the exterior $M^\circ = M-{\rm sing}(\mathcal{F})$ of the singular fibers. 
Then, the taut polynomial $\Theta_\mathcal{V}$ specializes to the 
Teichm\"uller  polynomial $\Theta_\mathcal{F}$ as $\Theta_\mathcal{F}\doteq \iota_\ast(\Theta_\mathcal{V})$, 
where $\iota_\ast:H_{M^\circ}\surj H_M$ is the natural surjective homomorphism 
induced by the inclusion map $\iota:M^\circ\inj M$. 
In particular, if $\mathcal{F}$ is fully-punctured, then $M^\circ=M$ and
$\Theta_{\mathcal{F}} \doteq \Theta_{\mathcal V}$ holds in $\Z[H_M]$. 
\end{lemma}

\section{Being fully-punctured is a profinite property}  
\label{sec:punctured} 

In this section, we prove \Cref{thm:fully-punctured}, which is crucial for our Theorems \ref{thm:tautpoly} and \ref{thm:filled} and also plays a key role in our application. 
For the proof of the following result, we use lemmas from \cite{YiLiu2023Invent} and \cite{YiLiu2025Peking}.  

\begin{proof}[Proof of {\rm \Cref{thm:fully-punctured}}] 
In order to reduce the argument to the closed case, we first adapt Liu's argument in the proof of \cite[Lemma 9.5]{YiLiu2023Invent} to construct a finite cover $M^\ast_\varphi$ of $M$, which is the mapping torus of a pseudo-Anosov map $f^\ast_\varphi$ on a finite cover $S'_\varphi$ of $S_\varphi$ such that $f^\ast_\varphi$ has no $1$-pronged puncture and fixes every puncture of $S'_\varphi$, and such that the subgroup $\pi_1(M^\ast_\varphi)$ of $\pi_1(M)$ does not depend on the choices involved. 
Let $\wt{S}_\varphi\to S_\varphi$ denote the cover associated to the kernel of the natural surjection $\pi_1(S_\varphi)\surj H_1(S_\varphi;\Z/2\Z)$, let $K_\varphi$ denote the kernel of $\pi_1(\wt{S}_\varphi)\surj H_1(\wt{S}_\varphi;\Z/2\Z)$, regarded as a subgroup of $\pi_1(S_\varphi)$, and let $S'_\varphi\to S_\varphi$ denote the corresponding cover. 
Since a characteristic subgroup of a characteristic subgroup is again characteristic, $K_{\varphi}$ is a characteristic subgroup of finite index of $\pi_1(S_\varphi)$. 
Here, $K_{\varphi}$ being characteristic means that $K_{\varphi}$ is stable under the action of ${\rm Aut}(\pi_1(S_\varphi))$, so the monodromy map lifts to $S'_{\varphi}$.
Then the loop around each puncture of $S_\varphi$ does not lift to a closed curve in $S'_\varphi$, or equivalently, the covering degree at each puncture is at least two. 
Indeed, for a surface of finite type with at least two punctures, the class of the loop around each puncture is non-trivial in the mod $2$ homology. 
If $S_\varphi$ has at least two punctures, then the loop around each puncture does not lie in the kernel of $\pi_1(S_\varphi)\surj H_1(S_\varphi;\Z/2\Z)$, and hence not in $K_\varphi$. 
If $S_\varphi$ has exactly one puncture, then the loop around it is null-homologous, so it lifts to closed curves in $\wt{S}_\varphi$. 
Hence $\wt{S}_\varphi$ has at least two punctures, and the same argument for $\wt{S}_\varphi$ shows that the loop does not lie in $K_\varphi$. 
Since the covering degree at each puncture is at least two, the invariant foliations of $f_\varphi$ pull back to $S'_\varphi$ with no $1$-pronged puncture. 

We next put $Q=\pi_1(S_\varphi)/K_\varphi$, fix an element $t\in \pi_1(M)$ with $\varphi(t)=1$, let $\theta\in {\rm Aut}(Q)$ denote the automorphism induced by the conjugation by $t$, and let $r$ denote the order of $\theta$. 
We take $D$ to be a common multiple of $|Q|\cdot|{\rm Aut}(Q)|$ and $(2-\chi(S'_\varphi))!$, and let $f^\ast_\varphi$ denote the lift of the iteration $f_\varphi^{D}$ to $S'_\varphi$ induced by the conjugation by $t^{D}$. 
For each $\gamma\in \pi_1(S_\varphi)$, the identity $(\gamma t)^{D}=\bigl(\prod_{i=0}^{D-1}t^i\gamma t^{-i}\bigr)t^{D}$ holds, and the image of the product in $Q$ is the $(D/r)$-th power of $\ol{\gamma}\,\theta(\ol{\gamma})\cdots\theta^{r-1}(\ol{\gamma})$, where $\ol{\gamma}$ denotes the image of $\gamma$ in $Q$. 
Since $r$ divides $|{\rm Aut}(Q)|$ and $|Q|$ divides $D/r$, this image is trivial, and hence the subgroup $K_\varphi\langle t^{D}\rangle$ of $\pi_1(M)$ and the mapping class of $f^\ast_\varphi$ do not depend on the choice of $t$. 
Moreover, the number of punctures of $S'_\varphi$ is at most $2-\chi(S'_\varphi)$, so $f^\ast_\varphi$ fixes every puncture of $S'_\varphi$, and hence the punctures of $S'_\varphi$ correspond bijectively to the cusps of $M^\ast_\varphi$. 
\emph{The mapping torus $M^\ast_\varphi$ of $f^\ast_\varphi$ acting on $S'_\varphi$ is thus a finite cover of $M$ with $\pi_1(M^\ast_\varphi)=K_\varphi\langle t^{D}\rangle$}, with respect to the natural map $M_\varphi^\ast\to M$, which completes the construction. 

Define $f_\psi^\ast$ and $N_\psi^\ast$ similarly, using $K_\psi$ that corresponds to $K_\varphi$ via $\Phi$ and the same $D$. 
This is legitimate: $\Phi$ maps the closure of the kernel of the mod $2$ abelianization to the corresponding kernel in each of the two steps, so $K_\psi$ is the subgroup of $\pi_1(S_\psi)$ given by the same construction, and $Q$ and $\chi(S'_\varphi)$ correspond via $\Phi$. 
Then, $\Phi$ restricts to an isomorphism $\wh{\pi}_1 (M^\ast_\varphi) \congto \wh{\pi}_1 (N^\ast_\psi)$, which we again denote by $\Phi$, and $f_\varphi^\ast$ on $S'_\varphi$ and $f_\psi^\ast$ on $S'_\psi$ correspond. 

Since $f^\ast_\varphi$ has no 1-prong, it extends to a pseudo-Anosov map $f^\ast_\varphi$ on the filled closed surface $\overline{S'_\varphi}$, and similarly for $\psi$. 
The mapping torus of $f^\ast_\varphi$ on $\overline{S'_\varphi}$ is the result 
$\overline{M^\ast_\varphi}$ of Dehn filling on $M^\ast_\varphi$ along the boundary slopes of the fiber surface. 
Define $\overline{N^\ast_\psi}$ for $\psi$ similarly. 
The boundary slopes of the fiber surfaces are characterized as the kernels of the fibered classes on the peripheral subgroups, so they correspond to each other under the slope correspondence of \cite[Theorem A]{Xu2024fillings-arXiv}. 
Hence, by \cite[Theorem A]{Xu2024fillings-arXiv} (or by \cite[Corollary C]{WiltonZalesskii2019CM} together with \cite[Proof of Lemma 9.5]{YiLiu2023Invent}), 
$\Phi$ induces an isomorphism $\wh{\pi}_1(\overline{M^\ast_\varphi}) \congto \wh{\pi}_1(\overline{N^\ast_\psi})$, which we again denote by $\Phi$. 

In this setting, \cite[Lemma 5.2]{YiLiu2025Peking} ensures a bijective correspondence between the set of periodic orbits of $\ol{M_\varphi^\ast}$ and $\ol{N_\psi^\ast}$. 
Moreover, as the equation $L_m(f;\gamma^\ast \chi_\omega)={\rm ind}_m(f;O)$ in the proof of \cite[Lemma 8.5]{YiLiu2023Invent} ensures, 
the $m$-periodic index of each closed orbit may be read off from a finite quotient via twisted Lefschetz numbers. 
Since the prong number of a closed orbit is determined by the $m$-periodic indices over all $m\geq 1$, 
we may conclude that 
\emph{a puncture of $S'_\varphi$ is a singular point of $f_\varphi^\ast$ if and only if the corresponding puncture of $S'_\psi$ is a singular point of $f_\psi^\ast$}. 

Note that the interior stable/unstable foliations persist under taking a power (iteration) of the monodromy map, and locally lift to those on a finite cover. 
If $f_\varphi$ is fully-punctured, then $f_\varphi^\ast$ on $S_\varphi'$ is also fully-punctured. 
Since $f_\varphi^\ast$ on $\overline{S_\varphi'}$ and $f_\psi^\ast$ on $\overline{S_\psi'}$ have the same number of $k$-prong singular points for each $k\geq 3$ and they correspond via $\Phi$, 
$f_\psi^\ast$ on $S_\psi'$ is also fully-punctured.
Therefore, $f_\psi$ is fully-punctured.  
This completes the proof. 
\end{proof} 

\begin{cor} \label{cor:prong-correspondence} 
Let $M$ and $N$ be fibered cusped hyperbolic $3$-manifolds with an isomorphism $\Phi : \wh{\pi}_1(M) \congto \wh{\pi}_1(N)$, and let $\varphi$ and $\psi$ be the corresponding fibered classes. 
Under the correspondence of cusps induced by $\Phi$, each cusp of $M$ and the corresponding cusp of $N$ carry the same number of punctures of the fibers, with the same prong number. 
\end{cor}

\begin{proof}
For each cusp $T$ of $M$, the suspension flow restricted to $T$ is a linear flow with closed orbits, and the fiber $S_\varphi$ meets each closed orbit in the punctures lying on $T$; 
hence these punctures form a single orbit of the monodromy map, and their number equals the index $[\Z : \varphi(\pi_1(T))]$.
By the peripheral regularity (\Cref{thm:regularity}), the correspondence of cusps is induced by an
isomorphism $g: \pi_1(T)\congto\pi_1(T')$ with $\wh{\varphi}|_{\wh{\pi}_1(T)} = \wh{\psi}\circ\wh{g}$, so that $\varphi(\pi_1(T)) = \psi(\pi_1(T'))$ holds in $\Z$ and the numbers of punctures coincide. The prong numbers coincide by the third paragraph of the above proof. 
\end{proof}

\begin{remark} \label{rem:prf.thm1.2} 
(1) By \Cref{cor:prong-correspondence}, \Cref{thm:fully-punctured} holds verbatim for the narrower variant of the fully-punctured condition, additionally requiring that there are no $2$-pronged punctures (cf.\ \Cref{ss.pAmap}). 

(2) The correspondence of the closed orbits \cite[Lemma 5.2]{YiLiu2025Peking} is based on the fact that 
the conjugacy classes of the closed orbits are separated in the profinite completion, 
which follows from the conjugacy separability of compact orientable $3$-manifold groups \cite{HamiltonWiltonZalesskii2013JLMS}*{Theorem 1.3} (cf.\ \cite[Lemma 8.6]{YiLiu2023Invent}). 
The argument of the twisted Lefschetz numbers also works for punctured surfaces, 
where a puncture is $1$-pronged if and only if the fixed point index of the corresponding filled point vanishes \cite[Lemma 2.1]{JiangGuo1993Pacific}; 
hence, by carefully following Liu's argument, one may work directly on the base spaces to obtain an alternative proof of \Cref{thm:fully-punctured}, without passing to the closed case.

(3) The construction of $M^\ast_\varphi$ in the above proof is a variant of the one in \cite{YiLiu2023Invent}*{proof of Lemma 9.5}. We specify the cover $S'_\varphi$ by an explicit tower of characteristic subgroups, use only the property that the covering degree at each puncture is at least two, and choose the iteration power $D$ so that the independence of the choices is verified directly. The construction is thus self-contained.
\end{remark} 

\section{Transverse orientations are profinitely detected} 
\label{sec:transverse} 

In this section, we prove \Cref{thm:transverse}. 
In \Cref{ss:transverse-orientations}, we work with classical surface dynamics: we extend Rykken's homological criterion to punctured surfaces and strengthen it to a characterization of the transverse orientability (\Cref{thm:Rykken-punctured}). 
In \Cref{ss.profinite-transverse}, we combine this criterion with profinite arguments.  

\subsection{Surface dynamics: a homological criterion} 
\label{ss:transverse-orientations}

Let $S$ be a connected surface of finite type without boundary, 
either closed or punctured as in \Cref{ss.pAmap}, 
and let $f: S\to S$ be a pseudo-Anosov map with 
stable/unstable singular foliations $W^s=W^s(f)$ and $W^u=W^u(f)$ 
and with dilatation $\lambda=\lambda(f)>1$. 
Let ${\rm Sing}(f)\subset S$ denote the set of inner singularities, so that 
$W^s$ and $W^u$ are transverse $1$-dimensional foliations on $S-{\rm Sing}(f)$. 

\begin{defn}\label{def:transverse-orientation}
A \emph{transverse orientation} (or \emph{coorientation}) of $W^s$ is a continuous orientation of the quotient line bundle $N(W^s)=T(S-{\rm Sing}(f))/TW^s$ over $S-{\rm Sing}(f)$, 
where $TW^s\subset T(S-{\rm Sing}(f))$ denotes the line subbundle tangent to the leaves.
We say that $W^s$ is \emph{transversely orientable} if it admits a transverse orientation. 
The same terminology applies to $W^u$.

Suppose that $W^s$ is transversely orientable and fix a transverse orientation $\nu$. 
Since $f(W^s)=W^s$, the push-forward $f_*\nu$ is again a transverse orientation of $W^s$. 
Since $S-{\rm Sing}(f)$ is connected, 
we have $f_*\nu=\varepsilon\nu$ for a unique sign $\varepsilon=\varepsilon(f)\in\{\pm1\}$, 
which is independent of the choice of $\nu$. 
We say that $f$ \emph{preserves} (resp.\ \emph{reverses}) \emph{the transverse orientation of} $W^s$ 
if $\varepsilon(f)=1$ (resp.\ $-1$). 
\end{defn}

\begin{remark} \label{rem:coorientation}
(1) Since $W^s$ and $W^u$ are transverse and complementary on $S-{\rm Sing}(f)$, 
the composition $TW^u\inj T(S-{\rm Sing}(f))\surj N(W^s)$ is a canonical, 
$f$-equivariant isomorphism of line bundles. 
Hence a transverse orientation of $W^s$ is canonically the same as an orientation of the leaves of $W^u$; in particular, $W^s$ is transversely orientable if and only if $W^u$ is orientable, 
and $f$ preserves the transverse orientation of $W^s$ if and only if $f$ preserves the orientation of $W^u$. Note that this identification uses neither an orientation of $S$ nor the invariant measures.
This matches the hypothesis ``with oriented unstable manifolds'' in \cite[Theorem 3.3]{Rykken1999Michigan} below.

(2) The local model of a $k$-pronged singularity or puncture admits a transverse orientation on its punctured neighborhood if and only if $k$ is even. 
Hence, if $W^s$ is transversely orientable, then every inner singularity and every puncture of $f$ is even-pronged; in particular, there are no $1$-pronged punctures.
\end{remark}

We shall use the following homological criterion of Rykken.

\begin{lemma}[Rykken {\cite[Theorems 1.6, 1.7, 3.3]{Rykken1999Michigan}}] \label{lem:Rykken} 
Let $\ol{f}:\ol{S}\to\ol{S}$ be a pseudo-Anosov homeomorphism on 
a closed orientable surface whose stable foliation is transversely orientable, and 
let $\lambda=\lambda(\ol{f})$ and $\varepsilon=\varepsilon(\ol{f})$ in the notation of \Cref{def:transverse-orientation}. 
Then $\varepsilon\lambda$ is a simple eigenvalue of $\ol{f}_*: H_1(\ol{S};\R)\to H_1(\ol{S};\R)$, and every eigenvalue of $\ol{f}_*$ of absolute value $\lambda$ coincides with $\varepsilon\lambda$.
\end{lemma}

\begin{proof} 
If $\ol{S}$ is the torus $S_{1,0}$, then $\ol{f}$ is an Anosov map induced by a hyperbolic matrix $A\in \mathrm{SL}_2\Z$, and the assertions hold, since $\ol{f}_*$ is represented by $A$ itself, whose eigenvalues are $\varepsilon\lambda$ and $\varepsilon\lambda^{-1}$. So assume that $\ol{S}$ is hyperbolic. 

By \Cref{rem:coorientation} (1), the unstable foliation $W^u(\ol{f})$ is orientable, and $\ol{f}$ preserves {\rm(}resp.\ reverses{\rm)} its orientation if and only if $\varepsilon=1$ {\rm(}resp.\ $\varepsilon=-1${\rm)}. 
A Markov partition for $\ol{f}$ exists by \cite[Theorem 1.6]{Rykken1999Michigan} (due to Fathi--Shub), and its Markov matrix $A$ is a primitive non-negative integer matrix, 
so that $\lambda$ is a simple eigenvalue of $A$ exceeding the absolute values of all the other eigenvalues \cite[Theorem 1.7]{Rykken1999Michigan}. 
By \cite[Theorem 3.3]{Rykken1999Michigan}, the eigenvalues of $\ol{f}_*$ coincide with those of $\varepsilon A$ including multiplicity, 
with the possible exception of some zeros and roots of unity. 
Since $|\varepsilon\lambda|=\lambda>1$, which is neither zero nor a root of unity, the assertions follow.
\end{proof}

Now we extend the criterion to punctured surfaces and strengthen it to a characterization of the transverse orientability. 
The closed case of the following argument in (1) is due to Liu \cite[proof of Theorem 11.1]{YiLiu2023Duke}. 

\begin{theorem} \label{thm:Rykken-punctured}
Let $f:  S\to S$ be a pseudo-Anosov map on an orientable surface of finite type without boundary, with dilatation $\lambda=\lambda(f)$.

{\rm (1)} The stable foliation $W^s(f)$ is transversely orientable if and only if 
$f_*:  H_1(S;\R)\to H_1(S;\R)$ admits an eigenvalue of absolute value $\lambda$.

{\rm (2)} Suppose that these equivalent conditions hold and put $\varepsilon=\varepsilon(f)$. Then $\varepsilon\lambda$ is a simple eigenvalue of $f_*$, 
and every eigenvalue of $f_*$ of absolute value $\lambda$ coincides with $\varepsilon\lambda$. 
In particular, $f$ preserves {\rm(}resp.\ reverses{\rm)} the transverse orientation of $W^s(f)$ if and only if
$\lambda$ {\rm(}resp.\ $-\lambda${\rm)} is an eigenvalue of $f_*$. 
\end{theorem}

\begin{proof}
We first prove (2). If $S$ is closed, then the assertion is \Cref{lem:Rykken}. 
Suppose that $S$ has punctures, and let $\ol{S}$ denote the closed surface obtained by filling in the punctures of $S$, so that $S=\ol{S}-P$ for a finite set $P$. 
Since $f$ permutes the ends of $S$, it extends to a homeomorphism $\ol{f}: \ol{S}\to\ol{S}$ with $\ol{f}(P)=P$. 
By \Cref{rem:coorientation} (2), every puncture of $f$ is even-pronged; a $2$-pronged puncture fills in to a regular point, and a $2k$-pronged puncture with $k\geq 2$ fills in to a $2k$-pronged singularity. 
Hence the invariant foliations of $f$ extend to those of $\ol{f}$, and $\ol{f}$ is a pseudo-Anosov map on the closed orientable surface $\ol{S}$ with $\lambda(\ol{f})=\lambda$. 
A transverse orientation of $W^s(f)$ extends to that of $W^s(\ol{f})$, so $\varepsilon(\ol{f})=\varepsilon(f)$, and \Cref{lem:Rykken} applies to $\ol{f}$. Consider the exact sequence of $\R$-vector spaces
\[
0\To V\To H_1(S;\R)\To H_1(\ol{S};\R)\To 0,
\]
where $V$ is spanned by the classes of loops around the punctures. 
The maps are equivariant with respect to $f_*$ and $\ol{f}_*$, and $f_*$ preserves $V$, permuting the puncture loops up to sign. 
Hence the eigenvalues of $f_*|_V$ are roots of unity, and the multiset of eigenvalues of $f_*$ is the union of those of $f_*|_V$ and those of $\ol{f}_*$. 
Since $\lambda>1$, the eigenvalues of absolute value $\lambda$ of $f_*$ and of $\ol{f}_*$ coincide including multiplicity, and the assertion (2) for $f$ follows from that for $\ol{f}$. 

Next we prove (1). The ``only if'' part follows from (2), since $|\varepsilon\lambda|=\lambda$. For the ``if'' part, we show the contrapositive. Suppose that $W^s(f)$ is not transversely orientable. 
Let $p: \widetilde S\to S$ denote the orientation double cover of $W^s(f)$, 
namely, the double cover of $S-{\rm Sing}(f)$ determined by the orientation homomorphism of the line bundle $N(W^s)$, completed over ${\rm Sing}(f)$; 
it is branched exactly over the odd-pronged singularities and is connected by the assumption. 
The invariant foliations pull back, the foliation $W^s(\widetilde f):=p^*W^s(f)$ is transversely orientable by construction, 
and $f$ admits a lift $\widetilde f$, which is again a pseudo-Anosov map with $\lambda(\widetilde f)=\lambda$ and commutes with the nontrivial deck transformation $\sigma$. 
Fix a transverse orientation $\widetilde\nu$ of $W^s(\widetilde f)$; then $\sigma$ reverses $\widetilde\nu$ by the construction. 
The transverse measure of $W^s(\widetilde f)$, signed by $\widetilde\nu$, defines a nonzero class
$\widetilde\omega\in H^1(\widetilde S;\R)$ satisfying 
$\widetilde f^*\widetilde\omega =\varepsilon(\widetilde f)\lambda\,\widetilde\omega$ and
$\sigma^*\widetilde\omega=-\widetilde\omega$ (cf.\ \cite[Chapter 14]{FarbMargalit2012book}). 
By the assertion (2) applied to $\widetilde f$, the line $\R\widetilde\omega$ is the unique eigenspace of $\widetilde f^*$ for the eigenvalues of absolute value $\lambda$. 
Since $p^*:  H^1(S;\R)\to H^1(\widetilde S;\R)$ is injective,
$f$-equivariant, and has the image in the $(+1)$-eigenspace of $\sigma^*$, the map $f^*$, and hence $f_*$, admits no eigenvalue of absolute value $\lambda$.
\end{proof}

Let $M_f=S\times[0,1]/\bigl((x,1)\!\sim\!(f(x),0)\bigr)$ denote the mapping torus of $f$ with the suspension flow $\Psi$, and let ${\rm sing}(\Psi)$ denote the union of the singular orbits of $\Psi$, namely the suspension of ${\rm Sing}(f)$. 
The transverse orientability of a $2$-dimensional foliation on a $3$-manifold is defined analogously to \Cref{def:transverse-orientation} via its normal line bundle. 
The following elementary lemma translates the transverse orientability of the suspension of the stable foliation into a condition on the monodromy map; it will be combined with \Cref{lem:eo-to} in \Cref{sec:taut}. 

\begin{lemma} \label{lem:suspension} 
Let $\Lambda^s$ denote the suspension of $W^s(f)|_{S-{\rm Sing}(f)}$, which is a $2$-dimensional foliation on $M_f-{\rm sing}(\Psi)$. 
Then $\Lambda^s$ is transversely orientable if and only if $W^s(f)$ is transversely orientable and $f$ preserves its transverse orientation.
\end{lemma}

\begin{proof}
The restriction of $\Lambda^s$ to each fiber $(S-{\rm Sing}(f))\times\{t\}$ is a copy of $W^s(f)|_{S-{\rm Sing}(f)}$, 
and the normal bundle of $\Lambda^s$ restricts on each fiber to the normal bundle $N(W^s)$ in $S$. 
Hence a transverse orientation of $\Lambda^s$ restricts to a family $(\nu_t)_{t\in[0,1]}$ of transverse orientations of $W^s(f)$, 
which is constant in $t$ by continuity and satisfies $f_*\nu_1=\nu_0$ by the gluing $(x,1)\sim(f(x),0)$. 
Thus, if $\Lambda^s$ is transversely orientable, then $W^s(f)$ admits a transverse orientation $\nu$ with $f_*\nu=\nu$. 
Conversely, such $\nu$ defines a transverse orientation of $\Lambda^s$ by the constant extension in $t$.
Finally, if $W^s(f)$ is transversely orientable, then by \Cref{def:transverse-orientation} an $f$-invariant transverse orientation exists if and only if $\varepsilon(f)=1$, namely, $f$ preserves the transverse orientation.
\end{proof}

\subsection{Profinite arguments} 
\label{ss.profinite-transverse} 
We first recall the correspondence of fibration data due to Liu.

\begin{lemma} \label{lem:Liu-correspondence}
Let the setting be as in {\rm \Cref{thm:transverse}}, and let
$S_\varphi$ and $S_\psi$ denote the fibers corresponding to $\varphi$ and
$\psi$. Then the following hold.

{\rm (1)} The fibers $S_\varphi$ and $S_\psi$ are homeomorphic.

{\rm (2)} $\lambda(f_\varphi)=\lambda(f_\psi)$.
\end{lemma}

\begin{proof}
The equality $\chi(S_\varphi)=\chi(S_\psi)$ and the assertion (2) are established in the proof of \cite[Lemma 9.5]{YiLiu2023Invent}; the latter is based on the closed case \cite[Corollary 8.2]{YiLiu2023Invent}. 
By \Cref{cor:prong-correspondence}, the fibers $S_\varphi$ and $S_\psi$ have the same number of punctures. 
Since the fibers are connected orientable surfaces with the same Euler characteristic and the same number of punctures, the assertion (1) follows.
\end{proof}

Note that the correspondence of fibered classes in the sense above is well-defined: by Liu's theorem \cite[Theorem 5.1, Corollary 5.3]{YiLiu2023Invent} and the regularity (\Cref{thm:regularity}), the induced isomorphism
$\Phi^\star : H^1(N;\Z) \congto H^1(M;\Z)$ maps fibered classes to fibered classes, and $\varphi = \Phi^\star(\psi)$ holds if and only if $\wh{\varphi} = \wh{\psi} \circ \Phi$. 

We now prove that the preservation of the transverse orientation is a profinite property of fibered classes (\Cref{thm:transverse}). 

\begin{proof}[Proof of {\rm \Cref{thm:transverse}}]
Write $\pi = \pi_1(M)$ and $\varpi = \pi_1(N)$, and regard $\varphi: \pi\twoheadrightarrow\Z$ and $\psi:\varpi\twoheadrightarrow\Z$. 
Their kernels $\Gamma_M={\rm Ker}\,\varphi$ and $\Gamma_N={\rm Ker}\,\psi$ are finitely generated free groups, 
being the fundamental groups of the punctured fiber surfaces $S_\varphi$ and $S_\psi$. 
Since $\wh{\varphi}=\wh{\psi}\circ\Phi$, the isomorphism $\Phi$ restricts to an isomorphism
$\Phi|_{\wh{\Gamma}_M}: \wh{\Gamma}_M\congto\wh{\Gamma}_N$ 
of the closures of $\Gamma_M$ and $\Gamma_N$, which coincide with their profinite completions; 
indeed, every finite-index characteristic subgroup of $\Gamma_M$ is normal in $\pi$ with the quotient virtually $\Z$, 
hence residually finite, so the profinite topology of $\pi$ induces the full profinite topology on $\Gamma_M$. Thus we obtain the following commutative diagram with exact rows:
\[
\begin{tikzcd}
1 \ar[r] & \wh{\Gamma}_M \ar[r] \ar[d, "\Phi|_{\wh{\Gamma}_M}"'] &
\wh{\pi} \ar[r, "\wh{\varphi}"] \ar[d, "\Phi"] & \wh{\Z} \ar[r] \ar[d, equal]
& 1 \\
1 \ar[r] & \wh{\Gamma}_N \ar[r] & \wh{\varpi} \ar[r, "\wh{\psi}"'] &
\wh{\Z} \ar[r] & 1.
\end{tikzcd}
\]

Fix elements $t_M\in\pi$ and $t_N\in\varpi$ with $\varphi(t_M)=1$ and $\psi(t_N)=1$. 
By the mapping torus structures $\pi=\Gamma_M\rtimes\langle t_M\rangle$ and $\varpi=\Gamma_N\rtimes\langle t_N\rangle$, 
the conjugations $c_{t_M}:\Gamma_M\to\Gamma_M$ and $c_{t_N}:\Gamma_N\to\Gamma_N$ represent the outer classes of the monodromy maps $f_\varphi$ and $f_\psi$,
so that the induced automorphisms on $H_1(\Gamma_M;\Z)=H_1(S_\varphi;\Z)$ and $H_1(\Gamma_N;\Z)=H_1(S_\psi;\Z)$ coincide with $(f_\varphi)_*$ and $(f_\psi)_*$, respectively.


By the commutativity of the diagram, $\Phi(t_M)\,t_N^{-1}\in\wh{\Gamma}_N$, so the conjugations by $\Phi(t_M)$ and by $t_N$ differ by an inner automorphism of $\wh{\Gamma}_N$ and hence induce the same automorphism on $H_1(\wh{\Gamma}_N;\wh{\Z})\cong H_1(\Gamma_N;\Z)\otimes_{\Z}\wh{\Z}$.
Since $\Phi|_{\wh{\Gamma}_M}$ intertwines the conjugations, namely, $\Phi\circ c_{t_M}=c_{\Phi(t_M)}\circ\Phi$ holds on $\wh{\Gamma}_M$, the induced isomorphism $H_1(\Gamma_M;\Z)\otimes_{\Z}\wh{\Z}\congto H_1(\Gamma_N;\Z)\otimes_{\Z}\wh{\Z}$ conjugates $(f_\varphi)_*$ to $(f_\psi)_*$. 
Therefore, fixing bases of the free $\Z$-modules $H_1(\Gamma_M;\Z)$ and  $H_1(\Gamma_N;\Z)$ of the same rank $n$ and regarding them as bases of the corresponding $\wh{\Z}$-modules, 
we see that the matrices $(f_\varphi)_*,(f_\psi)_*\in{\rm GL}_n\Z$ are conjugate in ${\rm GL}_n\wh{\Z}$ via $\Phi|_{\wh{\Gamma}_M}$. 
In particular, they have the same characteristic polynomial in $\wh{\Z}[x]$, and hence in $\Z[x]$ by the injectivity of $\Z\inj \wh{\Z}$. 

Now put $\lambda=\lambda(f_\varphi)$, which coincides with $\lambda(f_\psi)$ by \Cref{lem:Liu-correspondence}~(2). 
Since the characteristic polynomials of $(f_\varphi)_*$ and $(f_\psi)_*$ coincide, $(f_\varphi)_*$ admits an
eigenvalue of absolute value $\lambda$ if and only if $(f_\psi)_*$ does.
Hence, by \Cref{thm:Rykken-punctured}~(1), the stable foliation $W^s(f_\varphi)$ is transversely orientable if and only if $W^s(f_\psi)$ is transversely orientable, proving the first assertion. 
In this case, by \Cref{thm:Rykken-punctured}~(2), $\varepsilon(f_\varphi)\lambda$ is an eigenvalue of $(f_\varphi)_*$, and hence of $(f_\psi)_*$; since every eigenvalue of $(f_\psi)_*$ of absolute value $\lambda$ coincides with $\varepsilon(f_\psi)\lambda$, we conclude $\varepsilon(f_\varphi)=\varepsilon(f_\psi)$, 
namely, $f_\varphi$ preserves the transverse orientation of $W^s(f_\varphi)$ if and only if $f_\psi$ preserves that of $W^s(f_\psi)$.
\end{proof}

\section{The taut polynomial is a profinite invariant} \label{sec:taut} 
In this section, we prove \Cref{thm:tautpoly} and \Cref{thm:filled}: 
the taut polynomial of the layered veering triangulation associated to a fibered face is a profinite invariant whenever the monodromy maps are fully-punctured (\Cref{ss.fully-punctured}), 
and so are the Teichm\"{u}ller polynomials of suitable Dehn fillings (\Cref{ss.filled}). 
We conclude with remarks on possible generalizations (\Cref{ss.tautpoly-Remarks}).

\subsection{Fully-punctured cases} \label{ss.fully-punctured} 
Throughout this subsection, let the setting be as in \Cref{thm:tautpoly}.
Let $\Psi_{\mathcal{F}}$ denote the suspension pseudo-Anosov flow corresponding to $\mathcal{F}$ and $\mathcal{V}_{\mathcal{F}}$ the associated layered veering triangulation. 
Since the fibrations in $\mathcal{F}$ are fully-punctured, the flow $\Psi_{\mathcal{F}}$ has no singular orbits, so $\mathcal{V}_{\mathcal{F}}$ is a veering triangulation of $M$ itself, and the stable lamination $\Lambda^s_{\mathcal{F}}$ of $\Psi_{\mathcal{F}}$ in the sense of \Cref{lem:eo-to} is the suspension of the stable foliation $W^s(f_{\varphi})$ in the sense of \Cref{lem:suspension}, for every primitive fibered class $\varphi\in\R_+\mathcal{F}$. 
The same applies to $\mathcal{G}$ by \Cref{thm:fully-punctured}. 
We first prove the assertion (1), then 
establish the compatibility of the edge-orientation homomorphisms (\Cref{lem:eo-compatible}), 
and deduce the assertion (2). 

\begin{proof}[Proof of {\rm \Cref{thm:tautpoly} (1)}]
Fix a primitive fibered class $\varphi\in\R_+\mathcal{F}$ and let $\psi\in\R_+\mathcal{G}$ denote the corresponding fibered class via $\Phi$.
By \Cref{lem:eo-to} and \Cref{lem:suspension}, the following are equivalent:
\begin{enumerate}
\item[(i)] $\mathcal{V}_{\mathcal{F}}$ is edge-orientable;
\item[(ii)] $\Lambda^s_{\mathcal{F}}$ is transversely orientable;
\item[(iii)] $W^s(f_{\varphi})$ is transversely orientable and $f_{\varphi}$ preserves its transverse orientation.
\end{enumerate}
The same holds for $\mathcal{G}$ and $\psi$. 
By \Cref{thm:transverse}, the conditions (iii) for $\varphi$ and for $\psi$ are equivalent. This completes the proof.
\end{proof}

\begin{lemma} \label{lem:eo-compatible} 
Let the setting be as in {\rm \Cref{thm:tautpoly}}, and let $\omega_{\mathcal{F}}:\pi_1(M)\to\Z^{\times}$ and $\omega_{\mathcal{G}}:\pi_1(N)\to\Z^{\times}$ denote the edge-orientation homomorphisms of $\mathcal{V}_{\mathcal{F}}$ and $\mathcal{V}_{\mathcal{G}}$, respectively. 
Then $\omega_{\mathcal{F}}$ and $\omega_{\mathcal{G}}$ are compatible with $\Phi$; namely, the composition map
\[
\upsilon: \pi_1(N)\overset{\iota_N}{\inj}\wh{\pi}_1(N)
\overset{\Phi^{-1}}{\To}\wh{\pi}_1(M)
\overset{\wh{\omega}_{\mathcal{F}}}{\To}\Z^{\times}
\]
coincides with $\omega_{\mathcal{G}}$.
\end{lemma}

\begin{proof}
If $\mathcal{V}_{\mathcal{F}}$ is edge-orientable, then 
so is $\mathcal{V}_{\mathcal{G}}$ by the assertion (1), 
and $\omega_{\mathcal{F}}$, $\omega_{\mathcal{G}}$, and $\upsilon$ are all trivial; 
hence $\upsilon=\omega_{\mathcal{G}}$ holds. 

Suppose instead that $\mathcal{V}_{\mathcal{F}}$ is not edge-orientable, 
so that neither is $\mathcal{V}_{\mathcal{G}}$ by the assertion (1), and
$\omega_{\mathcal{F}}$ and $\omega_{\mathcal{G}}$ are surjective. 
Then $\wh{\omega}_{\mathcal{F}}$ and $\wh{\upsilon}=\wh{\omega}_{\mathcal{F}}\circ\Phi^{-1}$ are surjective as well, 
and \Cref{lem:NS} ensures that $\upsilon:\pi_1(N)\to\Z^{\times}$ is surjective.

Let $p_M: M^{\rm or}\to M$ and $p_N: N^{\upsilon}\to N$ denote the double covers corresponding to ${\rm Ker}\,\omega_{\mathcal{F}}$ and ${\rm Ker}\,\upsilon$, respectively. 
We verify that the pair $(M^{\rm or},N^{\upsilon})$ satisfies the hypotheses of \Cref{thm:tautpoly}. Both are cusped hyperbolic $3$-manifolds, being finite covers of $M$ and $N$. 
Since the closures of their fundamental groups in the profinite completions are 
${\rm Ker}\,\wh{\omega}_{\mathcal{F}}$ and ${\rm Ker}\,\wh{\upsilon}=\Phi({\rm Ker}\,\wh{\omega}_{\mathcal{F}})$, 
the isomorphism $\Phi$ restricts to an isomorphism $\Phi^{\rm or}:\wh{\pi}_1(M^{\rm or})\congto\wh{\pi}_1(N^{\upsilon})$. 
The pullback $\varphi^{\rm or}=\varphi\circ(p_M)_*$ of a fibered class $\varphi\in\R_+\mathcal{F}$ is again a fibered class, whose fibration is the lift of that of $\varphi$; 
let $\mathcal{F}^{\rm or}$ denote its fibered face, and define $\mathcal{G}^{\upsilon}$ for $N^{\upsilon}$ similarly. 
By the universality of the profinite completion, $\wh{\varphi}^{\rm or}=\wh{\psi}^{\upsilon}\circ\Phi^{\rm or}$ holds for corresponding classes $\varphi$ and $\psi$, 
so $\mathcal{F}^{\rm or}$ and $\mathcal{G}^{\upsilon}$ are a corresponding pair of fibered faces via $\Phi^{\rm or}$. 
The invariant foliations of the lifted monodromy maps are the lifts of those of the original ones, 
so the fibrations in $\mathcal{F}^{\rm or}$ and $\mathcal{G}^{\upsilon}$ are again fully-punctured.

The suspension flow corresponding to $\mathcal{F}^{\rm or}$ is the lift of $\Psi_{\mathcal{F}}$. 
Since the Agol--Gu\'{e}ritaud construction is performed equivariantly on the orbit space of the lifted flow on the universal cover, 
the associated veering triangulation is the lift $\mathcal{V}^{\rm or}_{\mathcal{F}}$ of $\mathcal{V}_{\mathcal{F}}$ to $M^{\rm or}$ (cf.\ \cite[Section 4]{LandryMinskyTaylor2023ETDS}), 
and similarly for $\mathcal{G}^{\upsilon}$. 
Now, since $p_M$ is the edge-orientation double cover, \Cref{lem:eo-cover} ensures that
$\mathcal{V}^{\rm or}_{\mathcal{F}}$ is edge-orientable. 
Applying \Cref{thm:tautpoly} (1) to the pair $(M^{\rm or},N^{\upsilon})$ with the corresponding fibered faces $\mathcal{F}^{\rm or}$ and $\mathcal{G}^{\upsilon}$, 
we see that $\mathcal{V}^{\upsilon}_{\mathcal{G}}$ is also edge-orientable.

Therefore, \Cref{lem:eo-universality} (the universality of the edge-orientation homomorphism) 
applied to $q=\upsilon$ ensures that $\omega_{\mathcal{G}}$ factors through $\upsilon:\pi_1(N)\surj\Z^{\times}$. 
Since every endomorphism of $\Z^{\times}$ is either trivial or the identity and $\omega_{\mathcal{G}}$ is nontrivial, we conclude $\omega_{\mathcal{G}}=\upsilon$. 
\end{proof}

\begin{proof}[Proof of {\rm \Cref{thm:tautpoly} (2)}]
By \Cref{lem:eo-compatible}, we have $\omega_{\mathcal{G}}=\wh{\omega}_{\mathcal{F}}\circ\Phi^{-1}\circ\iota_N$.
The two continuous homomorphisms $\wh{\omega}_{\mathcal{G}}$ and $\wh{\omega}_{\mathcal{F}}\circ\Phi^{-1}$ on $\wh{\pi}_1(N)$ agree on the dense subset $\iota_N(\pi_1(N))$, so they coincide. 
Hence $\wh{\omega}_{\mathcal{F}}=\wh{\omega}_{\mathcal{G}}\circ\Phi$. 
Hence \Cref{thm:twisted} applies to the pairs $(\omega_{\mathcal{F}},\alpha)$ and $(\omega_{\mathcal{G}},\beta)$, 
where $\alpha$ and $\beta$ denote the maximal free abelian quotient maps, and yields
$\Delta^{\alpha}_{\omega_{\mathcal{F}}}\doteq \Delta^{\beta}_{\omega_{\mathcal{G}}}$ 
via the induced isomorphism $\Phi_{\star}: H_M\congto H_N$, 
up to multiplication by units of $\Z[H_M]$. 
Since $\mathcal{V}_{\mathcal{F}}$ and $\mathcal{V}_{\mathcal{G}}$ are veering triangulations of $M$ and $N$ themselves, 
\Cref{thm:Parlak} together with \Cref{lem:tautTeich} gives 
$\Theta_{\mathcal{F}}\doteq\Delta^{\alpha}_{\omega_{\mathcal{F}}}$ in $\Z[H_M]$ and
$\Theta_{\mathcal{G}}\doteq\Delta^{\beta}_{\omega_{\mathcal{G}}}$ in $\Z[H_N]$. 
Combining these, we conclude that $\Theta_{\mathcal{F}}$ and $\Theta_{\mathcal{G}}$ coincide via $\Phi_{\star}$, up to multiplication by units of $\Z[H_M]$.
\end{proof}

\subsection{Filled cases} \label{ss.filled} 
We first analyze, in the classical setting, how ``the two-step fillings'' $M\to M'^{\circ} \to M'$ on a manifold $M$ affect the singular orbits and the fibered faces. 

\begin{lemma} \label{lem:filling-analysis}
Let $M$ be a fibered cusped hyperbolic $3$-manifold with a fully-punctured fibered face $\mathcal{F}$. 
Let $M'$ be a hyperbolic $3$-manifold obtained by Dehn filling of $M$, 
admitting a fibration that restricts to a fibration in $\R_+\mathcal{F}$, and let $\mathcal{F}'$ denote its fibered face. Then the following hold.

{\rm (1)} The punctures of the fibers lying on each filled cusp $T$ share a common prong number $k_T\geq 2$, 
and ${\rm sing}(\mathcal{F}')$ consists precisely of the cores of the filling solid tori at the cusps with $k_T\geq 3$. 

{\rm (2)} The exterior $M'^{\circ}=M'-{\rm sing}(\mathcal{F}')$ is the result of the Dehn filling of $M$ along the sub-collection of the given slopes at the cusps with $k_T=2$. 
It is a fibered cusped hyperbolic $3$-manifold, its fibration restricts to the given fibration of $M$, and its fibered face $\mathcal{F}'^{\circ}$ is fully-punctured. 

{\rm (3)} The layered veering triangulation of $M'^{\circ}$ associated to $\mathcal{F}'$ coincides with $\mathcal{V}_{\mathcal{F}'^{\circ}}$.
\end{lemma}

\begin{proof}
Since the fibration of $M'$ restricts to that of $M$, the fiber $S_{\varphi'}$ is obtained from $S_{\varphi}$ by filling in the punctures lying on the filled cusps, and the punctures on each filled cusp $T$ form a single orbit of the monodromy map, sharing a common prong number $k_T$. 
Since $M'$ is hyperbolic, the monodromy map $f_{\varphi'}$ is pseudo-Anosov, and by the uniqueness of the invariant foliations, they extend those of $f_{\varphi}$. Hence the filled points on each filled cusp $T$ are $k_T$-pronged: 
the case $k_T=1$ does not occur, since a pseudo-Anosov map admits no $1$-pronged inner singularity; the filled points are regular points if $k_T=2$, and inner singularities if $k_T\geq 3$. Since $\mathcal{F}$ is fully-punctured, the singular orbits of $\Psi_{\mathcal{F}'}$ are precisely the cores of the filling solid tori at the cusps with $k_T\geq 3$. 
This proves (1).

By (1), the exterior $M'^{\circ}$ is the asserted Dehn filling of $M$, and it is the mapping torus of the fully-punctured pseudo-Anosov map obtained by restricting $f_{\varphi'}$ to the exterior of its inner singularities (cf.\ \Cref{ss.pAmap}). 
Since $M'$ is a fibered hyperbolic $3$-manifold, its fiber $S_{\varphi'}$ has a negative Euler characteristic, and hence so does $S_{\varphi'^{\circ}}=S_{\varphi'}-{\rm Sing}(f_{\varphi'})$. 
Hence $M'^{\circ}$ is a fibered hyperbolic $3$-manifold whose fibration restricts to that of $M$, and its fibered face $\mathcal{F}'^{\circ}$ is fully-punctured. 
Moreover, $M'^{\circ}$ is cusped: if $M'$ is cusped, then so is $M'^{\circ}$, while if $M'$ is closed, then the closed fiber $S_{\varphi'}$ has negative Euler characteristic, so that ${\rm Sing}(f_{\varphi'})\neq\emptyset$. This proves (2).

Finally, the restriction of $\Psi_{\mathcal{F}'}$ to $M'^{\circ}$ is the suspension flow $\Psi_{\mathcal{F}'^{\circ}}$. 
Since the Agol--Gu\'{e}ritaud construction depends only on the flow with the singular orbits removed, the associated layered veering triangulations coincide. This proves (3).
\end{proof}

\begin{proof}[Proof of {\rm \Cref{thm:filled}}]
We apply \Cref{lem:filling-analysis} to $(M',\mathcal{F}')$ and $(N',\mathcal{G}')$. 
By \Cref{cor:prong-correspondence}, a filled cusp of $M$ satisfies $k_T=2$ if and only if the corresponding filled cusp of $N$ does. 
Hence $M'^{\circ}$ and $N'^{\circ}$ are the results of a corresponding pair of Dehn fillings on $M$ and $N$. Applying \cite{Xu2024fillings-arXiv}*{Theorem A}, which allows empty slopes, and observing that the isomorphism there is the descent of the given $\Phi$ to the quotients of the profinite completions (see the proof of \cite{Xu2024fillings-arXiv}*{Theorem A}), 
we obtain isomorphisms $\Phi'^{\circ}: \wh{\pi}_1(M'^{\circ})\congto\wh{\pi}_1(N'^{\circ})$ and $\Phi': \wh{\pi}_1(M')\congto\wh{\pi}_1(N')$, 
compatible with the quotient maps induced by the fillings.

The pair $(M'^{\circ},N'^{\circ})$ with $(\mathcal{F}'^{\circ}, \mathcal{G}'^{\circ})$ satisfies the hypotheses of \Cref{thm:tautpoly}: 
both are fibered cusped hyperbolic $3$-manifolds with fully-punctured fibered faces by \Cref{lem:filling-analysis} (2), and 
the fibered class of $M'^{\circ}$ restricting to $\varphi\in \R_+\mathcal{F}$ corresponds via $\Phi'^{\circ}$ to that restricting to the corresponding class $\psi\in \R_+\mathcal{G}$ by the compatibility with the quotient maps, 
so $\mathcal{F}'^{\circ}$ and $\mathcal{G}'^{\circ}$ are a corresponding pair of fibered faces. 
Therefore, \Cref{thm:tautpoly} yields that $\Theta_{\mathcal{V}_{\mathcal{F}'^{\circ}}}$ and $\Theta_{\mathcal{V}_{\mathcal{G}'^{\circ}}}$ coincide via $\Phi'^{\circ}_{\star}: H_{M'^{\circ}}\congto H_{N'^{\circ}}$, up to multiplication by units.

Next, applying \Cref{lem:tautTeich} to $(M',\mathcal{F}')$ and $(N',\mathcal{G}')$ together with \Cref{lem:filling-analysis} (3), we obtain
\[
\Theta_{\mathcal{F}'}\doteq
\iota_{M'^{\circ}*}(\Theta_{\mathcal{V}_{\mathcal{F}'^{\circ}}}),\qquad
\Theta_{\mathcal{G}'}\doteq
\iota_{N'^{\circ}*}(\Theta_{\mathcal{V}_{\mathcal{G}'^{\circ}}}),
\]
where $\iota_{M'^{\circ}*}:\Z[H_{M'^{\circ}}]\surj\Z[H_{M'}]$ denotes the surjective ring homomorphism induced by the inclusion $M'^{\circ}\inj M'$, and similarly for $N$.

Finally, by the regularity (\Cref{thm:regularity}, which covers the closed case as well), the isomorphism $\Phi'$ induces an isomorphism $\Phi'_{\star}: H_{M'}\congto H_{N'}$, and the diagram
\[
\begin{tikzcd}
\Z[H_{M'^{\circ}}] \ar[r, "\Phi'^{\circ}_{\star}", "\cong"']
\ar[d, two heads, "\iota_{M'^{\circ}*}"'] &
\Z[H_{N'^{\circ}}] \ar[d, two heads, "\iota_{N'^{\circ}*}"] \\
\Z[H_{M'}] \ar[r, "\Phi'_{\star}"', "\cong"] & \Z[H_{N'}]
\end{tikzcd}
\]
commutes. Indeed, the two compositions $H_{M'^{\circ}}\to H_{N'}$ are homomorphisms of finitely generated free abelian groups whose profinite completions coincide, 
since $\Phi'^{\circ}$ and $\Phi'$ are compatible with the quotient maps of the fillings; 
by the injectivity of $\Z\inj\wh{\Z}$, they coincide. 
Applying the vertical maps to the coincidence above, 
we conclude that $\Theta_{\mathcal{F}'}$ and $\Theta_{\mathcal{G}'}$ coincide via $\Phi'_{\star}$, 
up to multiplication by units of $\Z[H_{M'}]$.
\end{proof}

\subsection{Remarks} \label{ss.tautpoly-Remarks}

\begin{remark} \label{rem:without-fp}
In \Cref{thm:filled}, we partially removed the fully-punctured assumption of \Cref{thm:tautpoly}; 
it is natural to ask whether this assumption can be removed entirely, that is, 
if we can start directly from $\wh{\pi}_1(M')\cong\wh{\pi}_1(N')$ to obtain the coincidence of the Teichm\"{u}ller polynomials of $\mathcal{F}'$ and $\mathcal{G}'$. 
If we could realize the drilling along all the singular orbits of a pseudo-Anosov flow purely group-theoretically (cf.\,\cite{GHMOSW2024-arXiv}) and the process persists under the profinite completion, 
which is the direction opposite to \cite{Xu2024fillings-arXiv}*{Theorem A}, then this question would have an affirmative answer. 
\end{remark}

\begin{remark}
Although the taut polynomial is defined for any veering triangulation, 
\Cref{thm:tautpoly} cannot be directly extended to the non-layered case. 
A non-fibered face of the Thurston norm ball may be dynamically represented by 
several distinct pseudo-Anosov flows without perfect fits, which correspond 
to distinct veering triangulations and hence a whole set of taut polynomials 
rather than a single one. This situation is studied in detail in \cite{Parlak_mutation}. 
In addition to this multiplicity issue, a fundamental difficulty lies in the following question.
\end{remark}

\begin{question} \label{q:without-pf} 
Is having a pseudo-Anosov flow without perfect fits a profinite property?
\end{question}

\section{Application: Small normalized dilatations} \label{sec:App} 
In this section, we illustrate a strategy for finding cusped fibered hyperbolic 3-manifolds that are profinitely rigid by using the taut polynomials, the veering census, and Tsang's estimate. 

In the context of the small dilatation problem, 
Tsang \cite{TsangCC2025-arXiv}*{Theorem 1.4} proved that the set 
$\mca{D}=\{\lambda(f)^{-\chi(S)}\mid f\act S\}$ 
of normalized dilatations of fully-punctured pseudo-Anosov maps is the union of 
\begin{center}
$\{\frac{3+\sqrt{5}}{2},\frac{4+\sqrt{12}}{2}, (\text{Lehmer's number})^9, \frac{5+\sqrt{21}}{2}, |LT_{1,2}|^3, \frac{6+\sqrt{32}}{2}\}$
\end{center}
and a dense subset of $[\mu^4, \infty)$, where $\mu=\frac{1+\sqrt{5}}{2}$ denotes the golden ratio. 
To prove this, he sharpened previous estimates of Agol--Tsang \cite{AgolTsang2024AGT}, as stated below, 
and combined them with the lower bound $\mu^4$ of Hironaka--Tsang \cite{HironakaTsang2025GGD} for the maps whose punctures form at least two orbits.  

\begin{theorem}[Tsang {\cite{TsangCC2025-arXiv}*{Theorems 1.7, 1.8}}] \label{thm:Tsang}
Let $f: S\to S$ be a fully-punctured pseudo-Anosov map with dilatation $\lambda$. 
Then the layered veering triangulation associated to the fibration of the mapping torus $M_f$ consists of at most $\frac{1}{2}(\lambda^{-\chi(S)})^{2}$ tetrahedra. 

If in addition $S$ is once-punctured with $\lambda(f)^{-\chi(S)}\leq \mu^4 \approx 6.854$, then 
that of $M_f$ consists of at most $16$ $(< \frac{1}{2}(\mu^4)^2 \approx 23.49)$ tetrahedra. 
\end{theorem}

Although the statements in \cite{TsangCC2025-arXiv} assert the existence of veering triangulations within the bounds, the proofs bound the number of tetrahedra of the layered veering triangulation itself; see also the abstract of \cite{TsangCC2025-arXiv}.  

By virtue of our \Cref{thm:fully-punctured}, we derive the following. 

\begin{cor} \label{cor:small_dilatation} 
Let $M$ be the mapping torus of a fully-punctured pseudo-Anosov map $f: S\to S$ with dilatation $\lambda$, where $S$ is a punctured surface with $\chi(S)<0$, and let $N$ be a $3$-manifold with $\wh{\pi}_1(N)\cong\wh{\pi}_1(M)$. 
Then $N$ is the mapping torus of a fully-punctured pseudo-Anosov map on a surface homeomorphic to $S$, with the same dilatation $\lambda$. 
In particular, the layered veering triangulations of $M$ and $N$ associated to these fibrations consist of at most $\frac{1}{2}(\lambda^{-\chi(S)})^{2}$ tetrahedra; 
if moreover $M$ has only one cusp and $\lambda^{-\chi(S)}\leq\mu^{4}$, then they consist of at most
$16$ tetrahedra. 
\end{cor}

\begin{proof}
By the profinite detection of the geometrization (see \Cref{ss.profiniterigidity}), $N$ is a cusped hyperbolic $3$-manifold. 
Take the fibered class $\psi$ of $N$ corresponding to that of $M$ (see \Cref{ss.profinite-transverse}). 
By \Cref{lem:Liu-correspondence}~(1),~(2) and \Cref{thm:fully-punctured}, the monodromy map $f_{\psi}$ is fully-punctured, its fiber is homeomorphic to $S$, and $\lambda(f_{\psi})=\lambda$; 
in particular, the normalized dilatations of $f$ and $f_{\psi}$ coincide. 
The bounds now follow from \Cref{thm:Tsang}, where for the last assertion note that $N$ has only one
cusp as well (see \Cref{ss.profiniterigidity}).  
\end{proof}

Theorems \ref{thm:fully-punctured}, \ref{thm:tautpoly}, and \Cref{cor:small_dilatation} provide a practical strategy for finding profinitely rigid examples, explained as follows: 

\begin{strategy} \label{strategy} 
Let $M=M_f$ be as in \Cref{cor:small_dilatation}. In order to show that $\pi_1(M)$ is profinitely rigid among $3$-manifold groups, by \Cref{cor:small_dilatation}, it suffices to enumerate the members of the census \cite{VeeringCensus} with at most $\frac{1}{2}(\lambda^{-\chi(S)})^{2}$ tetrahedra that arise as the layered veering triangulations of fully-punctured mapping tori over a surface homeomorphic to $S$ with the same dilatation, and to distinguish $M$ within this finite list, e.g.,\ by the taut polynomials (\Cref{thm:tautpoly}). 
Moreover, if $M$ has only one cusp and $\lambda^{-\chi(S)}\leq\mu^{4}$, then the bound improves to $16$ tetrahedra, so that the current census \cite{VeeringCensus} suffices for the enumeration. 
\end{strategy} 

\begin{proof}[Proof of {\rm \Cref{cor:rigidity}}] 
Let $M$ be one of the three mapping tori listed in  \Cref{cor:rigidity}, and let $N$ be a $3$-manifold with $\wh{\pi}_1(N)\cong\wh{\pi}_1(M)$.  
Since the fibers are once-punctured with $\lambda^{-\chi}\leq\mu^{4}$, 
\Cref{cor:small_dilatation} ensures that $N$ is the mapping torus of a fully-punctured pseudo-Anosov map on the same once-punctured surface with the same dilatation, 
whose layered veering triangulation consists of at most $16$ tetrahedra and hence belongs to the veering census \cite{VeeringCensus}. 
The census \cite{VeeringCensus} lists a total of $87047$ veering triangulations with at most $16$ tetrahedra, among which $51766$ are layered \cite{TsangCC2025-arXiv}.  
By \cite{TsangCC2025-arXiv}*{Theorems 1.4, 1.7, 1.8, Tables 1, 2}, the census contains precisely $17$ members that arise as the layered veering triangulations of the mapping tori in
\[
\{M_f \mid f\curvearrowright S_{g,1},\ g\geq 1,\ 
\lambda(f)^{-\chi(S_{g,1})}\leq\mu^{4}\},
\]
where $f$ runs over fully-punctured orientation-preserving pseudo-Anosov
maps; 
among them, exactly the three items in \Cref{cor:rigidity} have $g>1$. 
The item (1) is the unique member with the fiber $S_{5,1}$, and the items (2) and (3) are the two members with the fiber $S_{2,1}$, which share the normalized dilatation $|LT_{1,2}|^{3}$ while their taut polynomials
$a^{4}+a^{3}-a^{2}+a+1$ and $a^{4}-a^{3}-a^{2}-a+1$ do not coincide up to the ambiguity in \Cref{thm:tautpoly}~(2). 
Hence, in each case, the layered veering triangulation of $N$ coincides with the indicated one, and $N$ is homeomorphic to $M$. 
Therefore, the fundamental groups of the three mapping tori are profinitely rigid among $3$-manifold groups. 
\end{proof} 

\begin{remark}
For fibered $3$-manifolds, the almost profinite rigidity of Liu relies on the finiteness of pseudo-Anosov maps with bounded dilatation on surfaces of a given Euler characteristic \cite[Theorem 9.2]{YiLiu2023Invent}; the dilatation and the Euler characteristic are profinitely determined, 
while distinct pseudo-Anosov maps may share them. 
The items (2) and (3) of \Cref{cor:rigidity} exhibit such a pair, which the taut polynomials distinguish. 
\end{remark}

\begin{remark} \label{rem.eg} 
Item (1) in \Cref{cor:rigidity} is the mapping torus of a map on $S_{5,1}$, 
which turns out to be the exterior of the pretzel knot $P(-2,3,7)$ in $S^3$. 
Hence, it is an example of a three-tangle Montesinos knot, 
whose profinite rigidities were proved by the first author and Xu \cite{CheethamWestXu2026-arXiv-twobridgelinks} in parallel with the present work, in a different way from ours. 
Its taut polynomial is the image of the Lehmer polynomial under $a\mapsto -a$. 

Items (2) and (3) are the mapping tori of maps on $S_{2,1}$ lifted from a minimum-dilatation 5-braid. 
The constant $LT_{1,2}$ is the maximal real root of $x^4-x^3-x^2-x+1$ and is known as one of the Lanneau--Thiffeault numbers. 

See also \cite[around Tables 1,2]{TsangCC2025-arXiv} and \cite{Segerman:veering_census_with_data} for more information. 
\end{remark}

\begin{remark} 
\label{rem:census-S11}  
The mapping tori of the other $14$ members in the proof of \Cref{cor:rigidity} are those over the once-punctured torus $S_{1,1}$. 
Their fundamental groups are profinitely rigid among $3$-manifold groups as well; 
this is verified by \Cref{strategy} combined with \Cref{lem:Liu-correspondence}~(2) and the numbers of homomorphisms onto small cyclic groups, independently of \cite{BridsonReidWilton2017}. 
We will discuss a general framework of effective detection for $S_{1,1}$-bundles elsewhere \cite{Ueki-F2byZ-arXiv}.
\end{remark}


\begin{remark} \label{rem:k>2cusped} 
If we consider all the surfaces $S_{g,n}$, then the values $\lambda^{-\chi}$ accumulate at $\mu^4$. 
For each fixed $n$, however, they need not accumulate at $\mu^4$. 
Currently, the census covers the cases up to 16 tetrahedra. 
If the census grows or the estimation becomes sharper, then a similar strategy would work for more examples. 
\end{remark}

\begin{question} \label{q.number}
Is the number of tetrahedra in the layered veering triangulation 
a profinite invariant of a fibered face? 
\end{question}

\begin{remark} \label{rem:Xu-q}
In light of \cite{Xu2026-fullrigidity-arXiv} (see \Cref{rem:Xu-fullrigidity}), \Cref{q.number} and the question in
\Cref{rem:without-fp} have affirmative answers a posteriori, and so does
\Cref{q:without-pf} among hyperbolic $3$-manifolds. Direct and effective
approaches to these questions remain of interest.
\end{remark}

\bibliographystyle{amsalpha}
\bibliography{CMUY_tautpoly_v2.arXiv.bbl}

\begin{bibdiv}
\begin{biblist}

\bib{AschenbrennerFriedlWilton2015}{book}{
      author={Aschenbrenner, Matthias},
      author={Friedl, Stefan},
      author={Wilton, Henry},
       title={3-manifold groups},
      series={EMS Series of Lectures in Mathematics},
   publisher={European Mathematical Society (EMS), Z\"urich},
        date={2015},
        ISBN={978-3-03719-154-5},
         url={https://doi.org/10.4171/154},
      review={\MR{3444187}},
}

\bib{Agol2011Contemp}{incollection}{
      author={Agol, Ian},
       title={Ideal triangulations of pseudo-{A}nosov mapping tori},
        date={2011},
   booktitle={Topology and geometry in dimension three},
      series={Contemp. Math.},
      volume={560},
   publisher={Amer. Math. Soc., Providence, RI},
       pages={1\ndash 17},
         url={https://doi.org/10.1090/conm/560/11087},
      review={\MR{2866919}},
}

\bib{AtiyahMacdonald}{book}{
      author={Atiyah, M.~F.},
      author={Macdonald, I.~G.},
       title={Introduction to commutative algebra},
   publisher={Addison-Wesley Publishing Co., Reading, Mass.-London-Don Mills,
  Ont.},
        date={1969},
      review={\MR{0242802}},
}

\bib{AgolTsang2024AGT}{article}{
      author={Agol, Ian},
      author={Tsang, Chi~Cheuk},
       title={Dynamics of veering triangulations: infinitesimal components of
  their flow graphs and applications},
        date={2024},
        ISSN={1472-2747,1472-2739},
     journal={Algebr. Geom. Topol.},
      volume={24},
      number={6},
       pages={3401\ndash 3453},
         url={https://doi.org/10.2140/agt.2024.24.3401},
      review={\MR{4812222}},
}

\bib{BenardConway2020Fourier}{article}{
      author={Benard, Leo},
      author={Conway, Anthony},
       title={A multivariable {C}asson-{L}in type invariant},
        date={2020},
        ISSN={0373-0956,1777-5310},
     journal={Ann. Inst. Fourier (Grenoble)},
      volume={70},
      number={3},
       pages={1029\ndash 1084},
         url={https://doi.org/10.5802/aif.3330},
      review={\MR{4117056}},
}

\bib{BergeronSengunVenkatesh2016}{article}{
      author={Bergeron, Nicolas},
      author={\c{S}eng\"un, Mehmet~Haluk},
      author={Venkatesh, Akshay},
       title={Torsion homology growth and cycle complexity of arithmetic
  manifolds},
        date={2016},
        ISSN={0012-7094,1547-7398},
     journal={Duke Math. J.},
      volume={165},
      number={9},
       pages={1629\ndash 1693},
         url={https://doi.org/10.1215/00127094-3450429},
      review={\MR{3513571}},
}

\bib{BesserDeninger1999}{article}{
      author={Besser, Amnon},
      author={Deninger, Christopher},
       title={{$p$}-adic {M}ahler measures},
        date={1999},
        ISSN={0075-4102},
     journal={J. Reine Angew. Math.},
      volume={517},
       pages={19\ndash 50},
         url={http://dx.doi.org/10.1515/crll.1999.093},
      review={\MR{1728549 (2001d:11070)}},
}

\bib{BoileauFriedl2020AMS}{incollection}{
      author={Boileau, Michel},
      author={Friedl, Stefan},
       title={The profinite completion of 3-manifold groups, fiberedness and
  the {T}hurston norm},
        date={2020},
   booktitle={What's next?---the mathematical legacy of {W}illiam {P}.
  {T}hurston},
      series={Ann. of Math. Stud.},
      volume={205},
   publisher={Princeton Univ. Press, Princeton, NJ},
       pages={21\ndash 44},
         url={https://doi.org/10.2307/j.ctvthhdvv.5},
      review={\MR{4205634}},
}

\bib{BJZR2023.problems}{unpublished}{
      author={Bridson, Martin},
      author={Jaikin-Zapirain, Andrei},
      author={Reid, Alan~W.},
       title={A list of problems suggested by participants of {W}orkshop on
  {P}rofinite {R}igidity},
        date={2023},
         url={https://www.icmat.es/RT/2023/GGTLDGT/week_3_problem_list.pdf},
  note={\url{https://www.icmat.es/RT/2023/GGTLDGT/week_3_problem_list.pdf}},
}

\bib{Bourbaki1961AlgCommChap1-2}{book}{
      author={Bourbaki, N.},
       title={\'el\'ements de math\'ematique. {F}ascicule {XXVII}. {A}lg\`ebre
  commutative. {C}hapitre 1: {M}odules plats. {C}hapitre 2: {L}ocalisation},
      series={Actualit\'es Scientifiques et Industrielles [Current Scientific
  and Industrial Topics]},
   publisher={Hermann, Paris},
        date={1961},
      volume={No. 1290},
      review={\MR{217051}},
}

\bib{Boyd1981JNT}{article}{
      author={Boyd, David~W.},
       title={Kronecker's theorem and {L}ehmer's problem for polynomials in
  several variables},
        date={1981},
        ISSN={0022-314X,1096-1658},
     journal={J. Number Theory},
      volume={13},
      number={1},
       pages={116\ndash 121},
         url={https://doi.org/10.1016/0022-314X(81)90033-0},
      review={\MR{602452}},
}

\bib{BridsonReid2020AMS}{incollection}{
      author={Bridson, Martin~R.},
      author={Reid, Alan~W.},
       title={Profinite rigidity, fibering, and the figure-eight knot},
        date={2020},
   booktitle={What's next?---the mathematical legacy of {W}illiam {P}.
  {T}hurston},
      series={Ann. of Math. Stud.},
      volume={205},
   publisher={Princeton Univ. Press, Princeton, NJ},
       pages={45\ndash 64},
         url={https://doi.org/10.2307/j.ctvthhdvv.6},
      review={\MR{4205635}},
}

\bib{Brown1982GTM}{book}{
      author={Brown, Kenneth~S.},
       title={Cohomology of groups},
      series={Graduate Texts in Mathematics},
   publisher={Springer-Verlag, New York-Berlin},
        date={1982},
      volume={87},
        ISBN={0-387-90688-6},
      review={\MR{672956}},
}

\bib{BridsonReidWilton2017}{article}{
      author={Bridson, Martin~R.},
      author={Reid, Alan~W.},
      author={Wilton, Henry},
       title={Profinite rigidity and surface bundles over the circle},
        date={2017},
        ISSN={0024-6093},
     journal={Bull. Lond. Math. Soc.},
      volume={49},
      number={5},
       pages={831\ndash 841},
         url={https://doi.org/10.1112/blms.12076},
      review={\MR{3742450}},
}

\bib{BenardTangeTranUeki-Whitehead}{unpublished}{
      author={B\'enard, L\'eo},
      author={Tange, Ryoto},
      author={Tran, Anh~T.},
      author={Ueki, Jun},
       title={Multiplicity of non-acyclic {$\rm {SL}_2$}-representations and
  ${L}$-functions of the odd-twisted {W}hitehead links},
        date={2026},
        note={preprint ver.2. arXiv:2303.15941},
}

\bib{BergeronVenkatesh2013}{article}{
      author={Bergeron, Nicolas},
      author={Venkatesh, Akshay},
       title={The asymptotic growth of torsion homology for arithmetic groups},
        date={2013},
        ISSN={1474-7480},
     journal={J. Inst. Math. Jussieu},
      volume={12},
      number={2},
       pages={391\ndash 447},
         url={http://dx.doi.org/10.1017/S1474748012000667},
      review={\MR{3028790}},
}

\bib{CartanEilenbergbook}{book}{
      author={Cartan, Henri},
      author={Eilenberg, Samuel},
       title={Homological algebra},
      series={Princeton Landmarks in Mathematics},
   publisher={Princeton University Press, Princeton, NJ},
        date={1999},
        ISBN={0-691-04991-2},
        note={With an appendix by David A. Buchsbaum, Reprint of the 1956
  original},
      review={\MR{1731415}},
}

\bib{Cimasoni2004BSMM}{article}{
      author={Cimasoni, David},
       title={Studying the multivariable {A}lexander polynomial by means of
  {S}eifert surfaces},
        date={2004},
        ISSN={1405-213X},
     journal={Bol. Soc. Mat. Mexicana (3)},
      volume={10},
       pages={107\ndash 115},
      review={\MR{2199342}},
}

\bib{CheethamWestXu2026-arXiv-twobridgelinks}{unpublished}{
      author={{Cheetham-West}, Tamunonye},
      author={Xu, Xiaoyu},
       title={Profinite rigidity for two-bridge links and 3-tangle montesinos
  links},
        date={2026},
        note={preprint. arXiv:2603.20822},
}

\bib{CheethamWestYao2025-arXiv-Apoly}{unpublished}{
      author={{Cheetham-West}, Tamunonye},
      author={Yao, Youheng},
       title={Finite covers and strict boundary slopes of cusped hyperbolic
  3-manifolds},
        date={2025},
        note={preprint. arXiv:2506.12289},
}

\bib{DixonFormanekPolandRibes1982JPAA}{article}{
      author={Dixon, John~D.},
      author={Formanek, Edward~W.},
      author={Poland, John~C.},
      author={Ribes, Luis},
       title={Profinite completions and isomorphic finite quotients},
        date={1982},
        ISSN={0022-4049,1873-1376},
     journal={J. Pure Appl. Algebra},
      volume={23},
      number={3},
       pages={227\ndash 231},
         url={https://doi.org/10.1016/0022-4049(82)90098-6},
      review={\MR{644274}},
}

\bib{EisenbudGTM}{book}{
      author={Eisenbud, David},
       title={Commutative algebra},
      series={Graduate Texts in Mathematics},
   publisher={Springer-Verlag, New York},
        date={1995},
      volume={150},
        ISBN={0-387-94268-8; 0-387-94269-6},
         url={https://doi.org/10.1007/978-1-4612-5350-1},
        note={With a view toward algebraic geometry},
      review={\MR{1322960}},
}

\bib{Fenley2012GT}{article}{
      author={Fenley, S\'ergio},
       title={Ideal boundaries of pseudo-{A}nosov flows and uniform convergence
  groups with connections and applications to large scale geometry},
        date={2012},
        ISSN={1465-3060,1364-0380},
     journal={Geom. Topol.},
      volume={16},
      number={1},
       pages={1\ndash 110},
         url={https://doi.org/10.2140/gt.2012.16.1},
      review={\MR{2872578}},
}

\bib{FLP1979Asterisque66-67}{book}{
      editor={Fathi, Albert},
      editor={Laudenbach, Fran\c{c}ois},
      editor={Po\'{e}naru, Valentin},
       title={Travaux de {T}hurston sur les surfaces},
      series={Ast\'{e}risque},
   publisher={Soci\'{e}t\'{e} Math\'{e}matique de France, Paris},
        date={1979},
      volume={66--67},
        note={S\'{e}minaire Orsay, With an English summary},
      review={\MR{568308}},
}

\bib{FarbMargalit2012book}{book}{
      author={Farb, Benson},
      author={Margalit, Dan},
       title={A primer on mapping class groups},
      series={Princeton Mathematical Series},
   publisher={Princeton University Press, Princeton, NJ},
        date={2012},
      volume={49},
        ISBN={978-0-691-14794-9},
      review={\MR{2850125}},
}

\bib{Fried1979Ast}{incollection}{
      author={Fried, David},
       title={Fibrations over {$S^1$} with pseudo-anosov monodromy},
    language={fr},
        date={1979},
   booktitle={Travaux de thurston sur les surfaces - s\'eminaire orsay},
      series={Ast\'erisque},
   publisher={Soci\'et\'e math\'ematique de France},
       pages={251\ndash 266},
         url={www.numdam.org/item/AST_1979__66-67__251_0/},
}

\bib{Fried1988}{incollection}{
      author={Fried, David},
       title={Cyclic resultants of reciprocal polynomials},
        date={1988},
   booktitle={Holomorphic dynamics ({M}exico, 1986)},
      series={Lecture Notes in Math.},
      volume={1345},
   publisher={Springer, Berlin},
       pages={124\ndash 128},
         url={http://dx.doi.org/10.1007/BFb0081399},
      review={\MR{980956}},
}

\bib{FrankelSchleimerSegerman2019-arXiv-v5}{unpublished}{
      author={Frankel, Steven},
      author={Schleimer, Saul},
      author={Segerman, Henry},
       title={From veering triangulations to link space and back again},
        date={2025},
        note={preprint. arXiv:1911.00006},
}

\bib{FriedlVidussi2011survey}{incollection}{
      author={Friedl, Stefan},
      author={Vidussi, Stefano},
       title={A survey of twisted {A}lexander polynomials},
        date={2011},
   booktitle={The mathematics of knots},
      series={Contrib. Math. Comput. Sci.},
      volume={1},
   publisher={Springer, Heidelberg},
       pages={45\ndash 94},
         url={https://doi.org/10.1007/978-3-642-15637-3_3},
      review={\MR{2777847}},
}

\bib{GHMOSW2024-arXiv}{unpublished}{
      author={Groves, Daniel},
      author={Ha\"issinsky, Peter},
      author={Manning, Jason~F.},
      author={Osajda, Damian},
      author={Sisto, Alessandro},
      author={Walsh, Genevieve~S.},
       title={Drilling hyperbolic groups},
        date={2026},
        note={preprint. arXiv:2406.14667},
}

\bib{GJZPZ2014}{article}{
      author={Grunewald, Fritz},
      author={Jaikin-Zapirain, Andrei},
      author={Pinto, Aline G.~S.},
      author={Zalesskii, Pavel~A.},
       title={Normal subgroups of profinite groups of non-negative deficiency},
        date={2014},
        ISSN={0022-4049,1873-1376},
     journal={J. Pure Appl. Algebra},
      volume={218},
      number={5},
       pages={804\ndash 828},
         url={https://doi.org/10.1016/j.jpaa.2013.10.003},
      review={\MR{3149636}},
}

\bib{Greenberg1973AJM}{article}{
      author={Greenberg, Ralph},
       title={The {I}wasawa invariants of {${\bf \Gamma }$}-extensions of a
  fixed number field},
        date={1973},
        ISSN={0002-9327,1080-6377},
     journal={Amer. J. Math.},
      volume={95},
       pages={204\ndash 214},
         url={https://doi.org/10.2307/2373652},
      review={\MR{332712}},
}

\bib{VeeringCensus}{misc}{
      author={Giannopoulos, Andreas},
      author={Schleimer, Saul},
      author={Segerman, Henry},
       title={A census of veering structures},
        date={2026},
        note={\url{https://github.com/henryseg/Veering}},
}

\bib{Hanany2026regularity-arXiv}{unpublished}{
      author={Hanany, Liam},
       title={Regularity of profinite isomorphisms of hyperbolic 3-manifolds},
        date={2026},
        note={preprint. arXiv:2607.04530},
}

\bib{Hempel1987}{incollection}{
      author={Hempel, John},
       title={Residual finiteness for {$3$}-manifolds},
        date={1987},
   booktitle={Combinatorial group theory and topology ({A}lta, {U}tah, 1984)},
      series={Ann. of Math. Stud.},
      volume={111},
   publisher={Princeton Univ. Press, Princeton, NJ},
       pages={379\ndash 396},
      review={\MR{895623}},
}

\bib{Hillar2005}{article}{
      author={Hillar, Christopher~J.},
       title={Cyclic resultants},
        date={2005},
        ISSN={0747-7171},
     journal={J. Symbolic Comput.},
      volume={39},
      number={6},
       pages={653\ndash 669},
         url={http://dx.doi.org/10.1016/j.jsc.2005.01.001},
      review={\MR{2167674}},
}

\bib{HughesKielak2025RMI}{article}{
      author={Hughes, Sam},
      author={Kielak, Dawid},
       title={Profinite rigidity of fibring},
        date={2025},
        ISSN={0213-2230,2235-0616},
     journal={Rev. Mat. Iberoam.},
      volume={41},
      number={4},
       pages={1305\ndash 1336},
         url={https://doi.org/10.4171/rmi/1524},
      review={\MR{4912921}},
}

\bib{HarperSaveliev2010PJM}{article}{
      author={Harper, Eric},
      author={Saveliev, Nikolai},
       title={A {C}asson-{L}in type invariant for links},
        date={2010},
        ISSN={0030-8730,1945-5844},
     journal={Pacific J. Math.},
      volume={248},
      number={1},
       pages={139\ndash 154},
         url={https://doi.org/10.2140/pjm.2010.248.139},
      review={\MR{2734168}},
}

\bib{HironakaTsang2025GGD}{article}{
      author={Hironaka, Eriko},
      author={Tsang, Chi~Cheuk},
       title={Standardly embedded train tracks and pseudo-{A}nosov maps with
  minimum expansion factor},
        date={2025},
        ISSN={1661-7207,1661-7215},
     journal={Groups Geom. Dyn.},
      volume={19},
      number={4},
       pages={1263\ndash 1317},
         url={https://doi.org/10.4171/ggd/805},
      review={\MR{4957786}},
}

\bib{HamiltonWiltonZalesskii2013JLMS}{article}{
      author={Hamilton, Emily},
      author={Wilton, Henry},
      author={Zalesskii, Pavel~A.},
       title={Separability of double cosets and conjugacy classes in 3-manifold
  groups},
        date={2013},
        ISSN={0024-6107,1469-7750},
     journal={J. Lond. Math. Soc. (2)},
      volume={87},
      number={1},
       pages={269\ndash 288},
         url={https://doi.org/10.1112/jlms/jds040},
      review={\MR{3022716}},
}

\bib{JiangGuo1993Pacific}{article}{
      author={Jiang, Bo~Ju},
      author={Guo, Jian~Han},
       title={Fixed points of surface diffeomorphisms},
        date={1993},
        ISSN={0030-8730,1945-5844},
     journal={Pacific J. Math.},
      volume={160},
      number={1},
       pages={67\ndash 89},
         url={http://projecteuclid.org/euclid.pjm/1102624565},
      review={\MR{1227504}},
}

\bib{JaikinZapirain2020GT}{article}{
      author={Jaikin-Zapirain, Andrei},
       title={Recognition of being fibered for compact 3-manifolds},
        date={2020},
        ISSN={1465-3060},
     journal={Geom. Topol.},
      volume={24},
      number={1},
       pages={409\ndash 420},
         url={https://doi.org/10.2140/gt.2020.24.409},
      review={\MR{4080486}},
}

\bib{MKim2020}{inproceedings}{
      author={Kim, Minhyong},
       title={Arithmetic {C}hern-{S}imons theory {I}},
        date={2020},
   booktitle={Galois covers, {G}rothendieck-{T}eichm\"{u}ller {T}heory and
  {D}essins d'{E}nfants, interactions between geometry, topology, number theory
  and algebra, {L}eicester, {UK}, {J}une 2018},
      editor={Neumann, Frank},
      editor={Schroll, Sibylle},
      series={Springer Proceedings in Mathematics \& Statistics},
      volume={330},
   publisher={Springer, Cham},
         url={https://doi.org/10.1007/978-3-030-51795-3},
        note={(cf. arXiv:1510.05818)},
      review={\MR{4166922}},
}

\bib{Kitayama2022survey}{incollection}{
      author={Kitayama, Takahiro},
       title={A survey of the {T}hurston norm},
        date={2022},
   booktitle={In the tradition of {T}hurston {II}. {G}eometry and groups},
   publisher={Springer, Cham},
       pages={149\ndash 199},
         url={https://doi.org/10.1007/978-3-030-97560-9_5},
      review={\MR{4472050}},
}

\bib{Le2018AIF}{article}{
      author={L\^e, Thang T.~Q.},
       title={Growth of homology torsion in finite coverings and hyperbolic
  volume},
        date={2018},
        ISSN={0373-0956,1777-5310},
     journal={Ann. Inst. Fourier (Grenoble)},
      volume={68},
      number={2},
       pages={611\ndash 645},
         url={https://doi.org/10.5802/aif.3173},
      review={\MR{3803114}},
}

\bib{Lackenby2000GT}{article}{
      author={Lackenby, Marc},
       title={Taut ideal triangulations of 3-manifolds},
        date={2000},
        ISSN={1465-3060,1364-0380},
     journal={Geom. Topol.},
      volume={4},
       pages={369\ndash 395},
         url={https://doi.org/10.2140/gt.2000.4.369},
      review={\MR{1790190}},
}

\bib{Lawton1983JNT}{article}{
      author={Lawton, Wayne~M.},
       title={A problem of {B}oyd concerning geometric means of polynomials},
        date={1983},
        ISSN={0022-314X,1096-1658},
     journal={J. Number Theory},
      volume={16},
      number={3},
       pages={356\ndash 362},
         url={https://doi.org/10.1016/0022-314X(83)90063-X},
      review={\MR{707608}},
}

\bib{Le2014CMH}{article}{
      author={Le, Thang},
       title={Homology torsion growth and {M}ahler measure},
        date={2014},
        ISSN={0010-2571,1420-8946},
     journal={Comment. Math. Helv.},
      volume={89},
      number={3},
       pages={719\ndash 757},
         url={https://doi.org/10.4171/CMH/332},
      review={\MR{3260847}},
}

\bib{YiLiu2020JAMS}{article}{
      author={Liu, Yi},
       title={Virtual homological spectral radii for automorphisms of
  surfaces},
        date={2020},
        ISSN={0894-0347,1088-6834},
     journal={J. Amer. Math. Soc.},
      volume={33},
      number={4},
       pages={1167\ndash 1227},
         url={https://doi.org/10.1090/jams/949},
      review={\MR{4155222}},
}

\bib{YiLiu2022ICM}{inproceedings}{
      author={Liu, Yi},
       title={Surface automorphisms and finite covers},
        date={2022},
   booktitle={Proceedings of the {I}nternational {C}ongress of {M}athematicians
  ---{S}eoul 2022},
   publisher={EMS Press},
       pages={2792\ndash 2805},
}

\bib{YiLiu2023Invent}{article}{
      author={Liu, Yi},
       title={Finite-volume hyperbolic 3-manifolds are almost determined by
  their finite quotient groups},
        date={2023},
        ISSN={0020-9910,1432-1297},
     journal={Invent. Math.},
      volume={231},
      number={2},
       pages={741\ndash 804},
         url={https://doi.org/10.1007/s00222-022-01155-4},
      review={\MR{4542705}},
}

\bib{YiLiu2023Duke}{article}{
      author={Liu, Yi},
       title={Mapping classes are almost determined by their finite quotient
  actions},
        date={2023},
        ISSN={0012-7094,1547-7398},
     journal={Duke Math. J.},
      volume={172},
      number={3},
       pages={569\ndash 631},
         url={https://doi.org/10.1215/00127094-2022-0047},
      review={\MR{4548418}},
}

\bib{YiLiu2025Peking}{article}{
      author={Liu, Yi},
       title={Finite quotients, arithmetic invariants, and hyperbolic volume},
        date={2025},
        ISSN={2096-6075,2524-7182},
     journal={Peking Math. J.},
      volume={8},
      number={1},
       pages={143\ndash 189},
         url={https://doi.org/10.1007/s42543-023-00077-1},
      review={\MR{4870949}},
}

\bib{LandryMinskyTaylor2023ETDS}{article}{
      author={Landry, Michael~P.},
      author={Minsky, Yair~N.},
      author={Taylor, Samuel~J.},
       title={Flows, growth rates, and the veering polynomial},
        date={2023},
        ISSN={0143-3857,1469-4417},
     journal={Ergodic Theory Dynam. Systems},
      volume={43},
      number={9},
       pages={3026\ndash 3107},
         url={https://doi.org/10.1017/etds.2022.63},
      review={\MR{4624489}},
}

\bib{LandryMinskyTaylor2024JEMS}{article}{
      author={Landry, Michael~P.},
      author={Minsky, Yair~N.},
      author={Taylor, Samuel~J.},
       title={A polynomial invariant for veering triangulations},
        date={2024},
        ISSN={1435-9855,1435-9863},
     journal={J. Eur. Math. Soc. (JEMS)},
      volume={26},
      number={2},
       pages={731\ndash 788},
         url={https://doi.org/10.4171/jems/1368},
      review={\MR{4705661}},
}

\bib{Marden1974}{article}{
      author={Marden, Albert},
       title={The geometry of finitely generated kleinian groups},
        date={1974},
        ISSN={0003-486X},
     journal={Ann. of Math. (2)},
      volume={99},
       pages={383\ndash 462},
         url={https://doi.org/10.2307/1971059},
      review={\MR{349992}},
}

\bib{McMullen2000ASENS}{article}{
      author={McMullen, Curtis~T.},
       title={Polynomial invariants for fibered 3-manifolds and {T}eichm\"uller
  geodesics for foliations},
        date={2000},
        ISSN={0012-9593},
     journal={Ann. Sci. \'Ecole Norm. Sup. (4)},
      volume={33},
      number={4},
       pages={519\ndash 560},
         url={https://doi.org/10.1016/S0012-9593(00)00121-X},
      review={\MR{1832823}},
}

\bib{McMullen2013CMH}{article}{
      author={McMullen, Curtis~T.},
       title={Entropy on {R}iemann surfaces and the {J}acobians of finite
  covers},
        date={2013},
        ISSN={0010-2571,1420-8946},
     journal={Comment. Math. Helv.},
      volume={88},
      number={4},
       pages={953\ndash 964},
         url={https://doi.org/10.4171/CMH/308},
      review={\MR{3134416}},
}

\bib{Milnor1957Isotopy}{incollection}{
      author={Milnor, John},
       title={Isotopy of links},
        date={1957},
   booktitle={Algebraic geometry and topology. a symposium in honor of s.
  lefschetz},
   publisher={Princeton University Press, Princeton, N.J.},
       pages={280\ndash 306},
      review={\MR{0092150}},
}

\bib{Mostow1973AnnMS}{book}{
      author={Mostow, G.~D.},
       title={Strong rigidity of locally symmetric spaces},
      series={Annals of Mathematics Studies},
   publisher={Princeton University Press, Princeton, NJ; University of Tokyo
  Press, Tokyo},
        date={1973},
      volume={No. 78},
      review={\MR{385004}},
}

\bib{Murasugi1966TransAMS}{article}{
      author={Murasugi, Kunio},
       title={On {M}ilnor's invariant for links},
        date={1966},
        ISSN={0002-9947,1088-6850},
     journal={Trans. Amer. Math. Soc.},
      volume={124},
       pages={94\ndash 110},
         url={https://doi.org/10.2307/1994437},
      review={\MR{198453}},
}

\bib{Neukirch}{book}{
      author={Neukirch, J{\"u}rgen},
       title={Algebraic number theory},
      series={Grundlehren der Mathematischen Wissenschaften [Fundamental
  Principles of Mathematical Sciences]},
   publisher={Springer-Verlag, Berlin},
        date={1999},
      volume={322},
        ISBN={3-540-65399-6},
         url={http://dx.doi.org/10.1007/978-3-662-03983-0},
        note={Translated from the 1992 German original and with a note by
  Norbert Schappacher, With a foreword by G. Harder},
      review={\MR{1697859 (2000m:11104)}},
}

\bib{NikolovSegal2007AnnMath1}{article}{
      author={Nikolov, Nikolay},
      author={Segal, Dan},
       title={On finitely generated profinite groups. {I}. {S}trong
  completeness and uniform bounds},
        date={2007},
        ISSN={0003-486X,1939-8980},
     journal={Ann. of Math. (2)},
      volume={165},
      number={1},
       pages={171\ndash 238},
         url={https://doi.org/10.4007/annals.2007.165.171},
      review={\MR{2276769}},
}

\bib{Parlak2023JT}{article}{
      author={Parlak, Anna},
       title={The taut polynomial and the {A}lexander polynomial},
        date={2023},
        ISSN={1753-8416,1753-8424},
     journal={J. Topol.},
      volume={16},
      number={2},
       pages={720\ndash 756},
         url={https://doi.org/10.1112/topo.12302},
      review={\MR{4637975}},
}

\bib{Parlak_mutation}{article}{
      author={Parlak, Anna},
       title={Mutations and faces of the {T}hurston norm ball dynamically
  represented by multiple distinct flows},
        date={2025},
        ISSN={1465-3060,1364-0380},
     journal={Geom. Topol.},
      volume={29},
      number={4},
       pages={2105\ndash 2173},
         url={https://doi.org/10.2140/gt.2025.29.2105},
      review={\MR{4929474}},
}

\bib{Prasad1973Invent}{article}{
      author={Prasad, Gopal},
       title={Strong rigidity of {${\bf Q}$}-rank {$1$} lattices},
        date={1973},
        ISSN={0020-9910,1432-1297},
     journal={Invent. Math.},
      volume={21},
       pages={255\ndash 286},
         url={https://doi.org/10.1007/BF01418789},
      review={\MR{385005}},
}

\bib{Reid-ICM2018}{inproceedings}{
      author={Reid, Alan~W.},
       title={Profinite rigidity},
        date={2018},
   booktitle={Proceedings of the {I}nternational {C}ongress of {M}athematicians
  ---{R}io de {J}aneiro 2018. {V}ol. {II}. {I}nvited lectures},
   publisher={World Sci. Publ., Hackensack, NJ},
       pages={1193\ndash 1216},
      review={\MR{3966805}},
}

\bib{Rykken1999Michigan}{article}{
      author={Rykken, E.},
       title={Expanding factors for pseudo-{A}nosov homeomorphisms},
        date={1999},
        ISSN={0026-2285,1945-2365},
     journal={Michigan Math. J.},
      volume={46},
      number={2},
       pages={281\ndash 296},
         url={https://doi.org/10.1307/mmj/1030132411},
      review={\MR{1704217}},
}

\bib{RibesZalesskii2010}{book}{
      author={Ribes, Luis},
      author={Zalesskii, Pavel},
       title={Profinite groups},
     edition={Second},
      series={Ergebnisse der Mathematik und ihrer Grenzgebiete. 3. Folge. A
  Series of Modern Surveys in Mathematics [Results in Mathematics and Related
  Areas. 3rd Series. A Series of Modern Surveys in Mathematics]},
   publisher={Springer-Verlag, Berlin},
        date={2010},
      volume={40},
        ISBN={978-3-642-01641-7},
         url={http://dx.doi.org/10.1007/978-3-642-01642-4},
      review={\MR{2599132}},
}

\bib{Sakuma-RIMS2011}{incollection}{
      author={Sakuma, Makoto},
       title={Problem session},
        date={2012February},
   booktitle={Gometric and analytic approaches to representations of a group
  and representation spaces, 2011},
      series={RIMS K\^{o}ky\^{u}roku},
   publisher={Res. Inst. Math. Sci. (RIMS), Kyoto},
       pages={110\ndash 112},
  url={https://www.kurims.kyoto-u.ac.jp/~kyodo/kokyuroku/contents/pdf/1777-15.pdf},
}

\bib{Segerman:veering_census_with_data}{misc}{
      author={Segerman, Henry},
       title={\texttt{veering\_census\_with\_data.txt}},
        date={2026},
  url={https://math.okstate.edu/people/segerman/veering/veering_census_with_data.txt},
  note={\url{https://math.okstate.edu/people/segerman/veering/veering_census_with_data.txt}},
}

\bib{Serre-CohGal}{book}{
      author={Serre, Jean-Pierre},
       title={Cohomologie galoisienne},
     edition={Fifth},
      series={Lecture Notes in Mathematics},
   publisher={Springer-Verlag, Berlin},
        date={1994},
      volume={5},
        ISBN={3-540-58002-6},
      review={\MR{1324577 (96b:12010)}},
}

\bib{Smyth1981}{article}{
      author={Smyth, C.~J.},
       title={On measures of polynomials in several variables},
        date={1981},
        ISSN={0004-9727},
     journal={Bull. Austral. Math. Soc.},
      volume={23},
      number={1},
       pages={49\ndash 63},
         url={http://dx.doi.org/10.1017/S0004972700006894},
      review={\MR{615132}},
}

\bib{SilverWilliams2002Topology}{article}{
      author={Silver, Daniel~S.},
      author={Williams, Susan~G.},
       title={Mahler measure, links and homology growth},
        date={2002},
        ISSN={0040-9383},
     journal={Topology},
      volume={41},
      number={5},
       pages={979\ndash 991},
         url={https://doi.org/10.1016/S0040-9383(01)00014-3},
      review={\MR{1923995}},
}

\bib{WPThurston1986HS2}{unpublished}{
      author={Thurston, William~P.},
       title={Hyperbolic structures on 3-manifolds, {II}: {S}urface groups and
  3-manifolds wichi fiber over the circle},
        date={1986},
        note={arXiv:math/9801019},
}

\bib{Thurston1986norm}{article}{
      author={Thurston, William~P.},
       title={A norm for the homology of {$3$}-manifolds},
        date={1986},
        ISSN={0065-9266,1947-6221},
     journal={Mem. Amer. Math. Soc.},
      volume={59},
      number={339},
       pages={i\ndash vi and 99\ndash 130},
      review={\MR{823443}},
}

\bib{Traldi1984}{article}{
      author={Traldi, Lorenzo},
       title={Milnor's invariants and the completions of link modules},
        date={1984},
        ISSN={0002-9947},
     journal={Trans. Amer. Math. Soc.},
      volume={284},
      number={1},
       pages={401\ndash 424},
         url={http://dx.doi.org/10.2307/1999294},
      review={\MR{742432}},
}

\bib{TsangCC2025-arXiv}{misc}{
      author={Tsang, Chi~Cheuk},
       title={On the set of normalized dilatations of fully-punctured
  pseudo-{Anosov} maps},
         how={Preprint, {arXiv}:2306.10245 [math.{GT}] (2025)},
        date={2026},
         url={https://arxiv.org/abs/2306.10245},
        note={to appear in Ann. Inst. Fourier},
}

\bib{TatenoUeki2025JLMS}{article}{
      author={Tateno, Sohei},
      author={Ueki, Jun},
       title={The {I}wasawa invariants of {$\Bbb Z_p^d$}-covers of links},
        date={2025},
        ISSN={0024-6107,1469-7750},
     journal={J. Lond. Math. Soc. (2)},
      volume={111},
      number={6},
       pages={Paper No. e70183, 70},
         url={https://doi.org/10.1112/jlms.70183},
      review={\MR{4920377}},
}

\bib{Ueki5}{article}{
      author={Ueki, Jun},
       title={The profinite completions of knot groups determine the
  {A}lexander polynomials},
        date={2018},
        ISSN={1472-2747},
     journal={Algebr. Geom. Topol.},
      volume={18},
      number={5},
       pages={3013\ndash 3030},
         url={https://doi.org/10.2140/agt.2018.18.3013},
      review={\MR{3848406}},
}

\bib{Ueki6}{article}{
      author={Ueki, Jun},
       title={Profinite rigidity for twisted {A}lexander polynomials},
        date={2021},
        ISSN={0075-4102},
     journal={J. Reine Angew. Math.},
      volume={771},
       pages={171\ndash 192},
         url={https://doi.org/10.1515/crelle-2020-0014},
      review={\MR{4234095}},
}

\bib{Ueki-F2byZ-arXiv}{unpublished}{
      author={Ueki, Jun},
       title={{E}ffective detection of ${F}_2$-by-$\mathbb{Z}$ groups},
        date={2026},
        note={preprint. arXiv:},
}

\bib{Wilkes2018NZ}{article}{
      author={Wilkes, Gareth},
       title={Profinite completions, cohomology and {JSJ} decompositions of
  compact 3-manifolds},
        date={2018},
        ISSN={1171-6096},
     journal={New Zealand J. Math.},
      volume={48},
       pages={101\ndash 113},
      review={\MR{3884906}},
}

\bib{Wilkes2018JA}{article}{
      author={Wilkes, Gareth},
       title={Profinite rigidity of graph manifolds and {JSJ} decompositions of
  3-manifolds},
        date={2018},
        ISSN={0021-8693},
     journal={J. Algebra},
      volume={502},
       pages={538\ndash 587},
         url={https://doi.org/10.1016/j.jalgebra.2017.12.039},
      review={\MR{3774902}},
}

\bib{WiltonZalesskii2017GT}{article}{
      author={Wilton, Henry},
      author={Zalesskii, Pavel},
       title={Distinguishing geometries using finite quotients},
        date={2017},
        ISSN={1465-3060},
     journal={Geom. Topol.},
      volume={21},
      number={1},
       pages={345\ndash 384},
         url={https://doi.org/10.2140/gt.2017.21.345},
      review={\MR{3608716}},
}

\bib{WiltonZalesskii2017JLMS}{article}{
      author={Wilton, Henry},
      author={Zalesskii, Pavel},
       title={Pro-{$p$} subgroups of profinite completions of 3-manifold
  groups},
        date={2017},
        ISSN={0024-6107,1469-7750},
     journal={J. Lond. Math. Soc. (2)},
      volume={96},
      number={2},
       pages={293\ndash 308},
         url={https://doi.org/10.1112/jlms.12067},
      review={\MR{3708951}},
}

\bib{WiltonZalesskii2019CM}{article}{
      author={Wilton, Henry},
      author={Zalesskii, Pavel},
       title={Profinite detection of 3-manifold decompositions},
        date={2019},
        ISSN={0010-437X},
     journal={Compos. Math.},
      volume={155},
      number={2},
       pages={246\ndash 259},
         url={https://doi.org/10.1112/s0010437x1800787x},
      review={\MR{3896566}},
}

\bib{Xu2024fillings-arXiv}{unpublished}{
      author={Xu, Xiaoyu},
       title={Profinite rigidity witnessed by {D}ehn fillings of cusped
  hyperbolic 3-manifolds},
        date={2024},
        note={preprint. arXiv:2412.05229},
}

\bib{Xu2025regularity-arXiv}{unpublished}{
      author={Xu, Xiaoyu},
       title={On regularity of profinite isomorphisms between cusped hyperbolic
  3-manifolds and the {$A$}-polynomial},
        date={2025},
         url={https://arxiv.org/abs/2506.21105},
        note={preprint. arXiv:2506.21105},
}

\bib{Xu2025AdvMath}{article}{
      author={Xu, Xiaoyu},
       title={Profinite almost rigidity in 3-manifolds},
        date={2025},
        ISSN={0001-8708,1090-2082},
     journal={Adv. Math.},
      volume={480},
       pages={Paper No. 110505, 79},
         url={https://doi.org/10.1016/j.aim.2025.110505},
      review={\MR{4953004}},
}

\bib{Xu2026-fullrigidity-arXiv}{unpublished}{
      author={Xu, Xiaoyu},
       title={Profinite rigidity in lattices of ${{\rm PSL}}(2,\mathbb{C})$},
        date={2026},
        note={preprint. arXiv:2608.07350},
}

\end{biblist}
\end{bibdiv}

\end{document}